\documentclass[12pt,reqno]{amsart}
\pdfoutput=1
\usepackage[table]{xcolor}
\usepackage{colortbl}
\usepackage{booktabs}
\usepackage{multirow}
\usepackage{subfig}
\usepackage[headings]{fullpage}
\usepackage{amssymb,bbm}
\usepackage{epstopdf}
\usepackage{graphicx}
\usepackage{texdraw}
\usepackage{url}
\usepackage[all]{xy}
\usepackage[T1]{fontenc}
\usepackage{mathtools}
\usepackage{float}   
\usepackage[bookmarks=true,%
    colorlinks=true,%
    linkcolor=blue,%
    citecolor=blue,%
    filecolor=blue,%
    menucolor=blue,%
    urlcolor=blue,%
    breaklinks=true]{hyperref}
\usepackage{stmaryrd}

\newtheorem{theorem}{Theorem}

\theoremstyle{definition}
\newtheorem{proposition}[theorem]{Proposition}
\newtheorem{lemma}[theorem]{Lemma}

\newtheorem{corollary}[theorem]{Corollary}
\newtheorem{conjecture}[theorem]{Conjecture}

\def\BN{\mathbbm N}
\def\BZ{\mathbbm Z}
\def\BQ{\mathbbm Q}
\def\BR{\mathbbm R}
\def\BC{\mathbbm C}

\def\calI{\mathcal I}

\def\IZ{\mathbb{Z}}

\def\calP{\mathcal P}
\def\calI{\mathcal I}

\def\calW{\mathcal W}

\def\calS{\mathcal S}
\def\s{\sigma}

\def\SL{\mathrm{SL}}
\def\PSL{\mathrm{PSL}}
\def\GL{\mathrm{GL}}

\def\tq{\tilde{q}}
\def\tx{\tilde{x}}
\def\ty{\tilde{y}}
\def\tw{\tilde{w}}

\def\={\;=\;}
\def\+{\,+\,}
\def\-{\,-\,}

\def\be{\begin{equation}}
\def\ee{\end{equation}}

\def\v{\varepsilon}

\def\Li{\mathrm{Li}}

\def\th{\theta}

\def\ve{\varepsilon}
\def\bb{\mathsf{b}}
\def\ri{\mathrm{i}}
\def\rd{\mathrm{d}}
\def\wt{\widetilde}
\def\tx{\tilde{x}}

\def\re{{\rm e}}
\def\ri{{\rm i}}
\def\rd{{\rm d}}

\def\Ahat{\widehat{A}}
\def\Bhat{\widehat{B}}

\def\Im{\mathrm{Im}}
\def\Re{\mathrm{Re}}
\def\diag{\mathrm{diag}}
\def\ub{u_{\mathsf{b}}}

\newenvironment{psmall}{\left(\begin{smallmatrix}}{\end{smallmatrix}\right)}

\def\mc{\mathcal}
\def\md{\mathbf}
\def\mf{\mathfrak}
\def\mfS{\mathfrak{S}}
\def\ms{\mathsf}

\def\cO{O}
\def\IC{\mathbb{C}}

\def\IN{\mathbb{N}}

\def\IR{\mathbb{R}}
\def\IZ{\mathbb{Z}}

\def\cJ{\mathcal{J}}

\def\cH{\mathcal{H}}
\def\cW{\mathcal{W}}

\newcommand{\knot}[1]{\md{#1}}

\newcommand{\wh}[1]{\widehat{#1}}
    
\newcommand{\nn}{\nonumber \\}

\graphicspath{{figs/}}

\begin{document}
\title[Peacock patterns and resurgence in complex Chern-Simons
theory]{Peacock patterns and resurgence in
  complex Chern-Simons theory} \author{Stavros Garoufalidis}
\address{
  International Center for Mathematics, Department of Mathematics \\
  Southern University of Science and Technology \\
  Shenzhen, China \newline
  {\tt \url{http://people.mpim-bonn.mpg.de/stavros}}}
\email{stavros@mpim-bonn.mpg.de}
\author{Jie Gu}
\address{D\'epartement de Physique Th\'eorique, Universit\'e de Gen\`eve \\
Universit\'e de Gen\`eve, 1211 Gen\`eve 4, Switzerland}
\email{jie.gu@unige.ch}
\author{Marcos Mari\~no}
\address{Section de Math\'ematiques et D\'epartement de Physique Th\'eorique\\
Universit\'e de Gen\`eve, 1211 Gen\`eve 4, Switzerland  \newline
  {\tt \url{http://www.marcosmarino.net}}}
\email{marcos.marino@unige.ch}



\date{7 April 2021}

\begin{abstract}
  The partition function of complex Chern-Simons theory on a
  3-manifold with torus boundary reduces to a finite dimensional
  state-integral which is a holomorphic function of a complexified
  Planck's constant $\tau$ in the complex cut plane and an entire
  function of a complex parameter $u$. This gives rise to a vector of
  factorially divergent perturbative formal power series whose Stokes
  rays form a peacock-like pattern in the complex plane.

  We conjecture that these perturbative series are resurgent, their
  trans-series involve two non-perturbative variables, their Stokes
  automorphism satisfies a unique factorization property and that it
  is given explicitly in terms of a fundamental matrix solution to a (dual)
  linear $q$-difference equation. We further conjecture that a
  distinguished entry of the Stokes automorphism matrix is
  the 3D-index of Dimofte--Gaiotto--Gukov. We provide proofs of our
  statements regarding the $q$-difference equations and their
  properties of their fundamental solutions and illustrate our
  conjectures regarding the Stokes matrices with numerical calculations
  for the two simplest hyperbolic $\knot{4}_1$ and $\knot{5}_2$ knots.
\end{abstract}

\maketitle

{\footnotesize
\tableofcontents
}


\section{Introduction}
\label{sec.intro}

\subsection{Chern--Simons theory with compact and complex
  gauge group}
\label{sub.intro}

Chern--Simons gauge theory, introduced by Witten in his seminal
paper~\cite{witten-jones} as a quantum field theory proposal of the
Jones polynomial~\cite{Jones}, remains one of the most fascinating
quantum field theories. It gives a powerful framework to study the
quantum topology of knots and three-manifolds, and at the same time it
provides a rich yet tractable model to explore general aspects of
quantum field theories.

In~\cite{witten-jones}, Witten analyzed in detail Chern--Simons gauge
theory with a compact gauge group (such as $\mathrm{SU}(2)$). Its
partition function on a 3-manifold $M$ with torus boundary components
depends on a quantized version $k \in \BZ$ of Planck's constant (or
equivalently, on a complex root of unity $\re^{2\pi \ri k}$), as well
as a discrete color (a finite dimensional irreducible representation
of $\mathrm{SU}(2)$) per boundary component of $M$. A more powerful
Chern--Simons theory with complex gauge group (such as $\SL(2,\BC)$)
was introduced by Witten~\cite{witten:complexCS} and developed
extensively by Gukov~\cite{gukov}.  A key feature of complex
Chern--Simons theory is that the partition function $Z_M(u;\tau)$ for
a 3-manifold $M$ with torus boundary components depends analytically
on a complex parameter $\tau$ (where $\tau=1/k$ in the Chern--Simons
theory with compact gauge group) as well as on a complex parameter $u$
per each boundary component of $M$ that plays the role of the holonomy
of a peripheral curve.  The analytic dependence of $Z_M(u;\tau)$ on
the parameters $u$ and $\tau$ allows one to formulate questions of
complex analysis and complex geometry which would be difficult, or
impossible, to do in Chern--Simons theory with compact gauge group.

There is a key difference between Chern--Simons theory with compact
versus complex gauge group: the former is an exactly solvable theory,
meaning that the partition function can be computed by a finite
state-sum, a consequence of the fact that it is a TQFT in 3
dimensions.  On the other hand, the situation with complex
Chern--Simons theory is more mysterious. For reasons that are not
entirely understood, the partition function $Z_M(u;\tau)$ for
manifolds with torus boundary components reduces to a
finite-dimensional integral (the so called-state integral) whose
integrand is a product of Faddeev's quantum dilogarithm
functions~\cite{Faddeev}, assembled out of an ideal triangulation of
the manifold. This was the approach taken by
Andersen-Kashaev~\cite{AK, AK-review} and Dimofte~\cite{dimofte-rev}
following prior ideas of~\cite{Hikami,DGLZ}. Focusing for simplicity
on the case of a 3-manifold with a single torus boundary component
(such as the complement of a hyperbolic knot in $S^3$), the
state-integral $Z_M(u;\tau)$ is a holomorphic function of
$\tau \in \BC'=\BC\setminus (-\infty,0]$ and 
$u \in \BC$ that satisfies a pair of linear $q$-difference equations.
The existence of these equations for a state-integral follows from the
closure properties of Zeilberger's theory of $q$-holonomic functions
and the the quasi-periodicity properties of Faddeev's quantum
dilogarithm, in much the same way as the $q$-holonomicity of the
colored Jones polynomial of a knot follows from a state-sum
formula~\cite{GL}. In fact, it is conjectured that the linear
$q$-difference equation satisfied by the colored Jones polynomial of a
knot coincides with the linear $q$-difference equations of
$Z_M(u;\tau)$ (see e.g.~\cite{AM} and Sections~\ref{sec.41}
and~\ref{sec.52} below for examples).

\subsection{Resurgence and the Stokes automorphism}
\label{sub.wall}
  
The global function $Z_M(u;\tau)$ gives rise to a vector
$\Phi(x;\tau)$ of perturbative series in $\tau$ whose coefficients are
meromorphic functions of $u$. These series are typically factorially
divergent and a key question is a description of the analytic
continuation of their Borel transform in Borel plane, their
trans-series and their Stokes automorphisms. This is a typical
question in perturbative quantum field theory where resurgence aims to
reproduce analytic functions from factorially divergent series (for an
introduction to resurgence, see for instance~\cite{msauzin,abs,mmreview,mmbook}),
and where Chern--Simons theory with a compact or complex gauge group is an
excellent case to analyze. Some aspects of resurgence in Chern--Simons
theory were studied
in~\cite{Ga:resurgence,CG,mmreview,gmp,gh-res,GZ:qseries,GZ:kashaev}.  The
multi-valuedness of the complex Chern-Simons action dictates that the
transseries are assembled out of monomials in $\tx$ and $\tq$ where
$\tq=\re^{-2 \pi \ri/\tau}$ and $\tx=\re^{u/\tau}$.

Our discoveries are summarized as follows:
\begin{itemize}
\item The singularities of the series $\Phi(x;\tau)$ in Borel
  plane are arranged in horizontal lines $2\pi \ri$ apart, and within
  these lines in finitely many points $\log x$ apart. This defines a
  collection of Stokes lines in a peacock-like pattern (see
  Figure~\ref{fig:stokes_rays}) whose corresponding Stokes
  automorphisms satisfy a unique factorization property with integer
  Stokes constants.
\item The Stokes automorphism $\sf{S}(x;q)$ along a half-plane is
  a fundamental matrix solution to a (dual) linear q-difference
  equation, hence fully computable.
\item The function $Z_M(u;\tau)$ is one entry of a matrix-valued collection
  of descendant partition functions which are a fundamental solution
  to a $q$-holonomic system in two variables.
\end{itemize}

\begin{figure}[htpb!]
  \centering
  \includegraphics[width=0.3\linewidth]{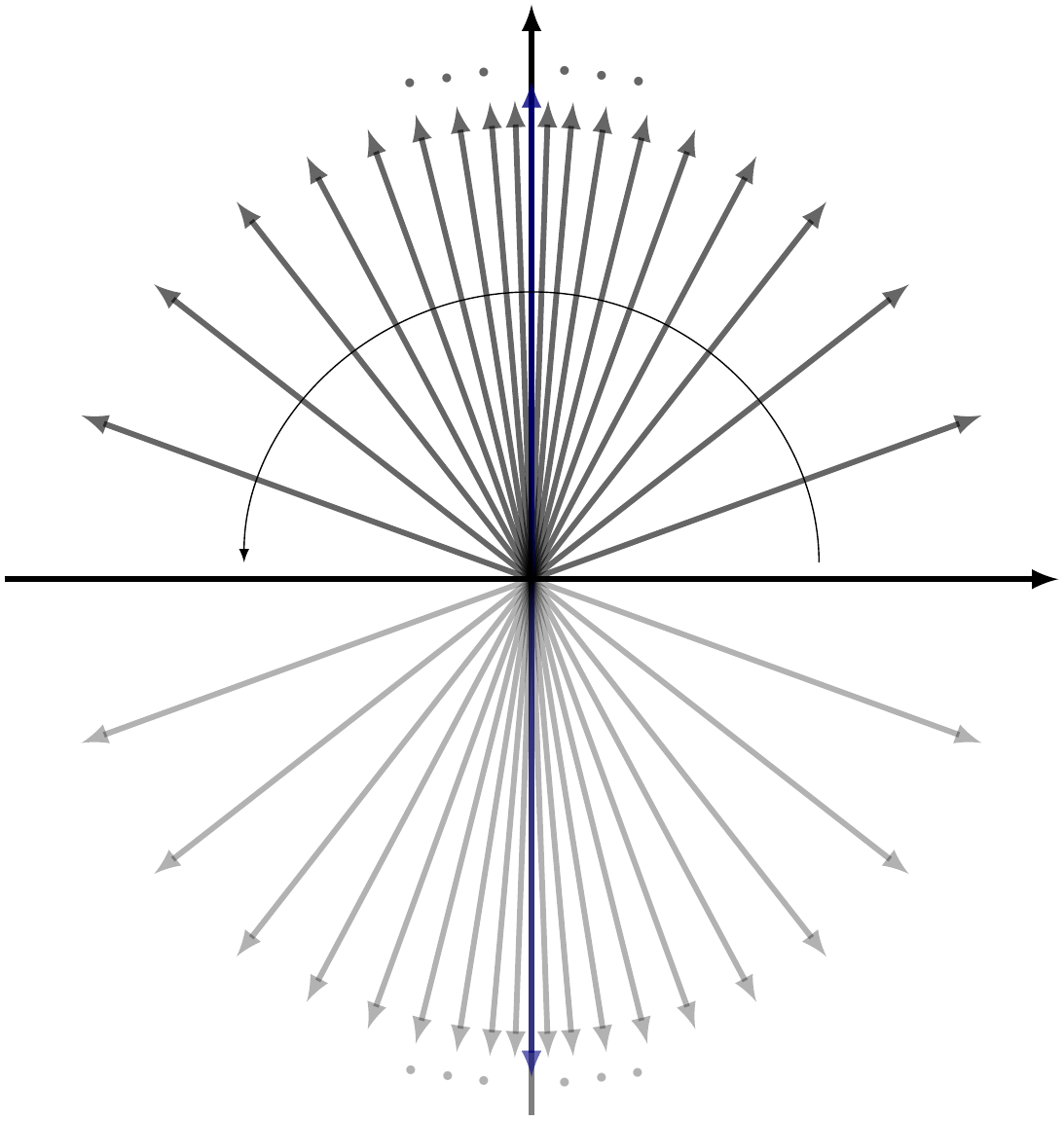}
  \caption{A peacock arrangement of Stokes rays in the complex Borel
    plane.}
  \label{fig:stokes_rays}
\end{figure}
  
The arrangement of the singularities in Borel plane
is reminiscent of a ``stability datum'' of
Kontsevich-Soibelman~\cite{KS,KS:stability, KS:wall-crossing} where the
corresponding integers are often called DT-invariants or BPS degeneracies. 
The Stokes automorphisms along half-planes are analogous to the spectrum
generators in
Gaiotto-Moore-Neitzke~\cite{gmn, gmn2, gmn3}. Our integers are locally
constant functions of a complex parameter $x$ and their jumping along
a wall-crossing will be the topic of a subsequent publication.

Our paper gives a concrete realization of these abstract ideas of
perturbative series and their resurgence, Stokes automorphisms and
their wall-crossing formulas for the case of complex Chern--Simons
theory, and illustrate our results with the 3-manifolds of the two
simplest hyperbolic knot complements, the complements of the
$\knot{4}_1$ and the $\knot{5}_2$ knots.

\subsection{A $q$-holonomic module of the partition function
  and its descendants}
\label{sub.results}

In this section we discuss a $q$-holonomic module associated to the
partition function and its descendants. This module and its
fundamental solutions are crucial to our exact computation of the
Stokes matrices in Section~\ref{sub.pert} below. One advantage of
introducing this module before we discuss resurgence of perturbative
series is that the former has been established mathematically in many
cases, whereas the latter remains a mathematical challenge.

We begin our discussion with a factorization of the state-integral
\begin{equation}
  \label{Zfacx}
  Z_M(\ub;\tau) = B(\tx,\tq^{-1})^T \Delta(\tau) B(x;q),
  \qquad (\tau \in \BC\setminus\BR)
\end{equation}
where
\begin{equation}
  \label{ub}
  \ub = \frac{u}{2\pi \bb} \,
\end{equation}
(this rescaling is dictated by the asymptotics of Faddeev's quantum
dilogarithm), $\Delta(\tau)$ is a diagonal matrix with diagonal
entries a 24-th root of unity times an integer power of
$\re^{\frac{\pi \ri}{24}(\tau+\tau^{-1})}$, 
$B(x;q)=(B^{(j)}(x;q))_{j=1}^r$ is a vector of holomorphic blocks, and
\begin{equation}
  \label{xq}
  q=\re^{2\pi \ri \tau}, \qquad \tq= \re^{-2 \pi \ri/\tau}, \qquad
  x=\re^{u}, \qquad \tx = \re^{u/\tau}, \qquad \tau=\bb^2 \,.
\end{equation}
The above notation is consistent with the literature in modular forms
and Jacobi forms~\cite{EZ} and indicates that $u \in \BC$ can be thought of
as a Jacobi variable. The
factorization~\eqref{Zfacx} was first noted in a related context in
\cite{pas} and further developed in complex Chern--Simons theory
in~\cite{Beem,dimofte-state}. We find that this factorization persists
to descendant state-integrals parametrized by a pair of (Jacobi-like)
variables $m$ and $\mu$ (see Equations \eqref{Z41x-desc} and
\eqref{Z52x-desc} for the definition of descendant state-integrals for
the knots $\knot{4}_1$, $\knot{5}_2$) 

\begin{equation}
  \label{ZM-desc}
  Z_{M,m,\mu}(\ub;\tau) 
  =
  (-1)^{m+\mu} q^{m/2}\tq^{\mu/2} B_{-\mu}(\tx;\tq^{-1})^T
  \Delta(\tau) B_m(x;q), \qquad (m, \mu \in \BZ)
\end{equation}
where $Z_{M,0,0}=Z_M$. The holomorphic blocks determine a matrix
$W_m(x;q)$ defined by
\begin{equation}
  \label{Wm}
  W_m(x;q) =
  \begin{pmatrix}
    B^{(1)}_{m}(x;q) & \ldots & B^{(r)}_{m}(x;q) \\
    \vdots & & \vdots \\
    B^{(1)}_{m+r-1}(x;q) & \ldots & B^{(r)}_{m+r-1}(x;q)
  \end{pmatrix}
\end{equation}
with the following properties

\begin{itemize}
\item[(a)] The entries of $W_m(x;q)$ are holomorphic functions of
  $|q| \neq 1$, meromorphic functions of $x \in \BC^*$ with poles in
  $x \in q^\BZ$ of order at most $r$, and have Taylor series
  expansions in $(1-x)^{-r}\BZ[x^\pm][[q^{\frac{1}{2}}]]$ whose
  monomials $x^k q^{n/2}$ satisfy $n= O(k^2)$.
\end{itemize}
In other words, the support of the monomials in $x$ and $q$
in $(1-x)^r W_m(x;q)$ is similar to the one of Jacobi forms (in their
holomorphic version of Eichler-Zagier~\cite[Eqn.(3)]{EZ} or the
meromorphic version from Zweger's thesis~\cite[Chpt.3]{Zwegers:thesis})
and of the admissible functions of Kontsevich-Soibelman~\cite{KS}.

\begin{itemize}
\item[(b)] The matrix 
  \begin{equation}
    \label{Zfacx-desc}
    W_{m,\mu}(u;\tau) =
    W_{-\mu}(\tx;\tq^{-1}) \Delta(\tau) W_{m}(x;q)^T
  \end{equation}
  defined for $\tau=\BC\setminus\BR$, extends to a holomorphic
  function of $\tau \in \BC'$ and $u \in \BC$ for
  all integers $m$ and $\mu$.
\end{itemize}
More precisely, if we define the normalized descendant integral
\begin{equation}
  \label{Zr}
  z_{M,m,\mu}(u;\tau) =
  (-1)^{m+\mu}q^{-m/2}\tq^{-\mu/2}
  Z_{M,m,\mu}(u;\tau) \,,
\end{equation}
then $W_{m,\mu}(u;\tau) =(z_{m+i,\mu+j}(\ub;\tau))_{i,j=0}^{r-1}$.
The above statement is remarkable in two ways: (i) $W_m(x;q)$ is a
holomorphic function of $\tau \in \BC\setminus\BR$ that cannot be
extended holomorphically over the positive reals, yet
$W_{m,\mu}(u;\tau)$ holomorphically extends over the positive reals
and (ii) $W_m(x;q)$ is a meromorphic function of $u$ with
singularities, yet $W_{m,\mu}(u;\tau)$ is an entire function of $u$.
Property (ii) is common in quantum mechanics, where the wave-function
is often entire whereas its WKB expansion is singular at the turning
points. The same behavior is also observed in the case of open topological
strings in~\cite{Marino:exact}.


\begin{itemize}
\item[(c)] We have an orthogonality relation
  \begin{equation}
    \label{Wortho}
    W_{-1}(x;q) \, W_{-1}(x;q^{-1})^T \in \GL(r,\BZ[x^\pm]) \,.
  \end{equation}
\item[(d)] The columns of $W_m(x;q)$, as functions of $(x,m)$, form a
  $q$-holonomic module of rank $r$.
\end{itemize}
The factorization~\eqref{ZM-desc} and (d) implies that the
annihilator $\calI_M$ of $z_{M,m,\mu}(\ub;\tau)$ as a function of
$(x,m)$ coincides with the annihilator of $W_m(x;q)$.  The latter is a
left ideal in the Weyl algebra $\calW$ over $\BQ[q^\pm]$ generated by
the pairs $(S_x,x)$ and $(S_m,q^m)$ of $q$-commuting operators which
act on a function $f(x,m;q)$ by
\begin{align}
  (S_x f)(x,m;q) &= f(qx,m;q) & (x f)(x,m;q) &= xf(x,m;q) \\
  (S_m f)(x,m;q) &= f(x,m+1;q) & (q^m f)(x,m;q) &= q^m f(x,m;q) \,.
\end{align}
Properties (a)-(d) above define meromoprhic quantum Jacobi forms, a
concept which is further studied in~\cite{GMZ}.  Although the above
statements are largely conjectural for the partition function of
complex Chern--Simons theory, we have the following result (see
Theorems~\ref{thm.41a}, \ref{thm.41b}, \ref{thm.52a}, \ref{thm.52b}
below).
  
\begin{theorem}
  \label{thm.4152W}
  The above statements hold for the $\knot{4}_1$ and $\knot{5}_2$
  knots.
\end{theorem}

We also study the Taylor series expansion of
Equations~\eqref{ZM-desc}, \eqref{Zfacx-desc} and~\eqref{Wortho} at
$u=0$ noting that the left hand side of the above equations are entire
functions of $u$ whereas the right hand side are a priori meromorphic
functions of $u$ with a pole of order $r$ at zero.  More precisely, in
Sections~\ref{sub.41x=1} and~\ref{sub.52x=1} we prove the following.
    
\begin{theorem}
  \label{thm.Zu=0}
  For the $\knot{4}_1$ and $\knot{5}_2$ knots,
  Equations~\eqref{ZM-desc}, \eqref{Zfacx-desc} and~\eqref{Wortho} can
  be expanded in Taylor series at $u=0$ whose constant terms are
  expressed in terms of the $q$-series of~\cite{GGM:resurgent}.
\end{theorem}

\subsection{Perturbative series and their resurgence}
\label{sub.pert}

We now discuss the resurgence properties of the asymptotic
expansion of the state-integral $Z_M(u;\tau)$. Once we fix an
integral presentation of $Z_M(u;\tau)$, the critical points of
the integrand are described by an affine curve $S$ defined
by a polynomial equation
\begin{equation}
\label{S}
S: p(x,y)=0 \,.
\end{equation}
We denote by $\calP$ the labeling set of the branches
$y=y_\s(x)$ of $S$. The perturbative expansion of the
state-integral  has the form
\begin{equation}
  \label{Phix}
  \Phi(x,y;\tau) =
  \re^{\frac{V(u,v)}{2\pi \ri \tau}} \varphi(x,y;\tau), \qquad
  \varphi(x,y;\tau) \in \frac{1}{\sqrt{\ri\delta}} \BQ[x^\pm, y^\pm,
  \delta^{-1}][[2\pi \ri \tau]]
\end{equation}
where $V: S^* \to \BC/(2\pi \ri)$, $S^*$ is the exponentiated defined
by $p(x,y)=0$ with $x=\re^{u}$, $y=-\re^{v}$ and
$\delta \in \BQ(x,y)$ is the so-called 1-loop invariant.  
The asymptotic series $\varphi(x,y;\tau)$ satisfies
$\varphi(x,y;0) = 1$.  The 8-th root of unity that appears as a
prefactor in $\phi(x,y;\tau)$ exactly matches with one appearing in
the asymptotics of the Kashaev invariant noticed
in~\cite[Sec.1]{GZ:kashaev}. After choosing local branches, we define
the vector
$\varphi(x;\tau)=(\varphi_\s(x;\tau))_{\s\in\calP} =
(\varphi(x,y_\s(x);\tau))_{\s \in \calP}$ of asymptotic series.
Recall the vector of holomorphic blocks $B(x;q)$ from~\eqref{ZM-desc}.
We now discuss the relation between the asymptotics of
$B(x;q)$ when $q=\re^{2\pi \ri \tau}$ and $\tau$ approaches zero (in
sectors) and the Borel resummation $s(\Phi)$ of the vector of power
series $\Phi(x;\tau)$.

The next conjecture summarizes the singularities of $\Phi(x;\tau)$
in the Borel plane, the relation between the asymptotics of the
holomorphic blocks with the Borel resummation $s_\th(\Phi)(x;\tau)$
as well as the properties of the Stockes automorphism matrices
$\ms{S}$, whose detailed definition is given in Section~\ref{sec.borel}.

\begin{conjecture}
  \label{conj.asy}
  \rm{(a)} The singularities of $\Phi_{\s}(x;\tau)$ in the Borel plane
  are a subset of
  \begin{equation}
    \label{Phising}
    \{ \iota_{\s,\s'}^{(\ell,k)} \,\, |
    \,\, \s' \in \calP, \,\, k, \ell \in \BZ, \,\, k=O(\ell^2) \} 
  \end{equation}
  where
  \begin{equation}
    \label{iota}
    \iota_{\s,\s'}^{(\ell,k)} = \frac{V(\s) - V(\s')}{2\pi\ri} +
    2\pi\ri k + \ell \log x \qquad (\s,\s' \in \calP, \,\, k, \ell \in \BZ) \,.
  \end{equation}
  In particular, the set of trans-series is labeled by three indices,
  $\sigma \in \calP$ and $k, \ell \in \IZ$, and they are of the form
  $\Phi_\s(\tau) \tq^k \tx^\ell$.  \newline
  \rm{(b)} On each ray $\rho$ in the complement of the singularities
  of $\Phi(x;\tau)$ in Borel plane, there exist a matrix
  $M_\rho(\tx;\tq)$ with entries in $\BZ[\tx^{\pm}][[\tq]]$ such that
  \begin{equation}
    \label{abrelx}
    \Delta(\tau) B(x;\tau) = M_\rho(\tx;\tq) s_\rho(\Phi)(x;\tau) \,. 
  \end{equation}
  \rm{(c)} The Stokes matrices $\ms{S}^{+}(x,q)$, $\ms{S}^{-}(x,q^{-1})$
  are given by
  \begin{equation}
  \label{2S}
  \ms{S}^{+}(x;q) \stackrel{\cdot}=
  W_{-1}(x^{-1};q^{-1})\cdot W_{-1}(x;q)^T, \qquad
  \ms{S}^{-}(x;q^{-1}) \stackrel{\cdot}= W_{-1}(x;q)\cdot W_{-1}(x^{-1};q^{-1})^T
  \end{equation}
  where $\stackrel{\cdot}=$ means equality up to multiplication on the
  left and on the right by a matrix in $\GL(r,\BZ[x^\pm])$.  \newline
  \rm{(d)} The Stokes matrix $\ms{S}$ uniquely determines the Stokes
  matrices at each Stokes ray, and the Stokes constants are integers
  corresponding to the Donaldson--Thomas invariants in
  \cite{KS,KS:wall-crossing,KS:stability} and the BPS degeneracies 
  in \cite{gmn,gmn2}.  \newline
  \rm{(e)} The Stokes matrices satisfy the inversion relation
    \begin{equation}
    \label{Ssym}
    \ms{S}^+(x;q)^T \ms{S}^-(x^{-1};q) = \md{1}\,.
  \end{equation}
\end{conjecture}

The Stokes automorphism matrix has an interesting
connection to physics which we now discuss.  Given a hyperbolic knot
$K$ in $S^3$, one can construct a three-dimensional $\mc{N}=2$
supersymmetric theory $T_M$ associated to the knot complement
$M = S^3\backslash K$ \cite{DGG1} (see also~\cite{ter-yama}), whose
BPS invariants are conjectured to coincide with the Stokes constants
of $s(\Phi)(x;\tau)$.  This conjecture can be made more precise in the
following manner. The BPS invariants of $T_M$ are encoded in the
3D-index $\mc{I}_{K}(m,e)(q)$ labeled by two integers $(m,e)$ called
magnetic and electric fluxes respectively~\cite{DGG1}. One can further
define the 3D-index in the fugacity basis (also known as the rotated
index) by~\cite{DGG2}
\begin{equation}
  \text{Ind}^{\text{rot}}_{K}(m,\zeta;q) =
  \sum_{e\in\IZ}\mc{I}_{K}(m,e)(q)\zeta^e.
\end{equation}
The 3D-index is a topological invariant of hyperbolic 3-manifolds with
at least one cusp (see~\cite{GHRS}).  And it can be evaluated using
holomorphic blocks $B_{K}^\alpha(x;q)$ by \cite{Beem}
\begin{equation}
  \text{Ind}_{K}^{\text{rot}}(m,\zeta;q) = \sum_{\alpha}
  B_K^\alpha(q^{m/2}\zeta;q)B_K^\alpha(q^{m/2}\zeta^{-1};q^{-1}).
  \label{eq:IndB}
\end{equation}
We have observed the following relation between the Stokes
matrix and the rotated 3D-index, and have proven it for the case of
the $\knot{4}_1$ and $\knot{5}_2$ knots using the explicit formulas
for the Stokes matrices.

\begin{conjecture}
  \label{conj.S3D}
  For every hyperbolic knot $K$, we have
  \begin{equation}
    \ms{S}^+(x;q) \stackrel{\cdot\cdot}=
    \left(\text{Ind}^{\text{rot}}_K(j-i,q^{\frac{j+i}{2}}x;q)
    \right)_{i,j=0,1,\ldots}
    \label{eq:SInd}
  \end{equation}
  where $\stackrel{\cdot\cdot}=$ means equality up to multiplication
  on the left and on the right by a matrix in $\GL(r,\BZ(x,q))$.  In
  particular, we find
  \begin{equation}
    \label{S3Dx}
    \ms{S}^+_{\s_1\s_1}(x;q) =
    \text{Ind}^{\text{rot}}_{K}(0,x;q)
    \,,
  \end{equation}
  where the equality is exact.  This holds true for the $\knot{4}_1$
  and the $\knot{5}_2$ knots.
\end{conjecture}

One consequence of~\eqref{abrelx} is that (after multiplying both
terms of~\eqref{abrelx} by the inverse of $M_R(\tx;\tq)$), we can
express the Borel resummation of the factorially divergent series
$\Phi(\tau)$ in terms of descendant state-integrals which are
holomorphic functions in the cut plane
$\BC'=\BC\setminus (-\infty,0]$.

Another consequence of the $q$-holonomic module defined by the
annihilation ideal $\calI_M$ is a refinement of the $\Ahat$-polynomial
of a knot as well as a new $\Bhat$-polynomial whose classical limit is
new.  The refinement comes in the form of a new variable $q^m$ where
$m$ is the descendant variable, whose geometric meaning is not
understood but might be related to some kind of quantum K-theory, or
perhaps related to the Weil-Gelfand-Zak transform of~\cite{AK:new}.
This refinement does not seem to be directly related to other refinements
of the $\Ahat$-polynomial, as those considered in
\cite{av-knots,super,GLL:HOMFLY}. At any rate, the $q$-holonomic ideal
$\calI_M$ contains unique
polynomials $\Ahat_M(S_x,x,q^m,q) \in \calW$ and $\Bhat_M(S_m,q^m,x,q)$
(of lowest degree, content-free) that annihilate the functions
$z_{M,m,k}(u;\tau)$ in the variables $(m,x)$.

\begin{conjecture}
  \label{conj.ann}
  When $M=S^3\setminus K$ is the complement of a knot $K$, then
  \begin{itemize}
  \item[(a)] $\Ahat_M(S_x,x,1,q)$ is the homogeneous
    $\Ahat$-polynomial of the knot~\cite{Ga:AJ} and
    $\Ahat_M(S_x,x,1,1)$ is the $A$-polynomial of the knot with
    meridian variable $x^2$ and longitude variable $S_x$~\cite{CCGLS}
  \item[(b)] $\Bhat_M(y,x,1,1)$ is the defining polynomial of the curve
    $S$.
  \end{itemize}
\end{conjecture}

In theorems~\ref{thm.41c} and~\ref{thm.52c} we prove the following.

\begin{theorem}
  \label{thm.4152ann}
  Conjecture~\ref{conj.ann} holds for the $\knot{4}_1$ and
  $\knot{5}_2$ knots.
\end{theorem}

\subsection{Disclaimers}
\label{sub.disclaimers}

We end this introduction with some comments and disclaimers. 

The first is that that there is no canonical labeling of holomorphic
blocks by $\calP$. Instead, the holomorphic blocks $B(x;q)$ is an
$r \times 1$ vector, $M_R$ are $r \times \calP$ matrices for all $R$,
$W_m(x;q)$ are $r \times r$ matrices and $\ms{S}$ are
$\calP \times \calP$ matrices and where $r$ is the cardinality of
$\calP$.

The second is that the entries of the matrix $W_m(x;q)$ are
holomorphic functions of $q^{1/N}$ for $|q| \neq 1$, where $N$ is a
natural number (the ``level'' of the knot) being one for the
$\knot{4}_1$ and $\knot{5}_2$ knots, but being $2$ for the
$(-2,3,7)$-pretzel knot.  For instance, the entries of the matrix
$W_0(x;q)$ are power series in $q^{1/2}$~\cite{GZ:kashaev}.  This
phenomenon was observed first in~\cite{GZ:kashaev} in a related
matrix-valued Kashaev invariant of the knot as well as
in~\cite{GZ:qseries} in a matrix of $q$-series associated to the three
simplest hyperbolic knots, and replaces the modular group $\SL(2,\BZ)$
by its congruence subgroup $\SL(2,N)$. In our current paper, we will
assume that $N=1$.

The third comment involves the crucial question of topological invariance.
Strictly speaking, the curve $S$ in Equation~\eqref{S} and the vector
of power series $\Phi(x;\tau)$ depend on an integral representation of
$Z_M(u;\tau)$, determined for instance by a suitable ideal triangulation of
$M$ as was done in~\cite{AK}. On the other hand, the vector of power series
$\Phi(x;\tau)$, its Stokes matrix $\ms{S}(x;q)$ and the $q$-holonomic
module generated by the matrix $W_m(x;\tau)$ are expected to be topological
invariants of $M$. Even if we fix an ideal triangulation, and we fix the
$q$-holonomic module, the fundamental solution matrix $W_m(x;\tau)$ in
general has a potential ambiguity, which we now discuss.

\begin{lemma}
  \label{lem.Wamb}
  Suppose that a matrix $W_m(x;q)$ satisfies the following properties: 
  \begin{itemize}
  \item
    It factorizes the state-integral~\eqref{ZM-desc},
  \item
    It is a fundamental solution matrix to a $q$-holonomic module,
  \item
    It satisfies the orthogonality equation~\eqref{Wortho}
  \item
    It satisfies the analytic conditions of (a) above.
  \end{itemize}
  Then, $W_m(x;q)$ is uniquely determined up to right multiplication
  by a diagonal matrix of signs.
\end{lemma}

\begin{proof}
  Any two fundamental solutions of a $q$-holonomic system differ by
  multiplication by a diagonal matrix $\diag(E(x;q))$. If both
  fundamental solutions satisfy~\eqref{ZM-desc} and~\eqref{Wortho}, it
  follows that each $E(x;q)$ satisfies
\begin{equation}
  \label{EE}
  E(\tx;\tq^{-1}) E(x;q) =1 , \qquad E(x;q)E(x;q^{-1})=1 \,.
\end{equation}
Thus, $E(x;q^{-1})=E(\tx;\tq^{-1})$ and after replacing $\bb$ by $\bb^{-1}$,
it implies that $E(x;q)=E(\tx;\tq)$.
It follows that $E(qx;q)=E(x;q)$ and $E(\tq \tx;\tq)=E(\tx;\tq)$.  In
other words $E$ is elliptic. Condition (a) implies that the poles of
$E(x;q)$ are a subset of $\ri \bb \BZ + \ri \bb^{-1} \BZ$ for $|q| \neq 1$.
It follows from $E(x;q) E(x;q^{-1})=1$ that both the poles and the zeros
of $E(x;q)$ are a subset of $\ri \bb \BZ + \ri \bb^{-1} \BZ$ and each
pole and zero has order at most $r$. Thus, $E(x;q)$ and $1/E(x;q)$ is a
polynomial in the Weierstrass polynomial $p(x;q)$ with coefficients
independent of $x$, and this implies that $E(x;q)=g(q)$ is independent of $x$,
where $g$ is a modular function with no zeros in the upper half-plane,
hence $g$ is a modular unit~\cite{Kubert-Lang}. There is none for $\SL(2,\BZ)$
(see~\cite{Kubert-Lang}), hence $g(q) = \pm 1$. Hence, $W_m(q;x)$ is
well-defined up to right multiplication by a diagonal matrix of signs. 
\end{proof}

\subsection{Further directions}
\label{sub.problems}

In this short section we make some comments about future directions. 
The factorization of the state-integral~\eqref{Zfacx} and its
descendant version~\eqref{ZM-desc} into a matrix points towards a
TQFT in 4 dimensions where the vector space associated to a 3-manifold is
labeled by $\calP$.

In another direction, 
as shown in \cite{klmvw,gmn2}, in $\mathcal{N}=2$ theories in four
dimensions, the BPS invariants can be studied by applying WKB methods
to their Seiberg--Witten curve. Since, in complex Chern--Simons
theory, the A-polynomial curve plays in a sense the r\^ole of a
Seiberg--Witten curve~\cite{gukov}, one could study it with the
techniques of \cite{klmvw,gmn2}, further extended in
\cite{esw,blm1,blm2} to curves in exponentiated variables.  It would
be interesting to see one can obtain in this way the BPS invariants
directly from the A-polynomial of the hyperbolic knot.

Peacock patterns of Borel singularities, with integer Stokes constants, are likely to appear in 
problems controlled by a quantum curve in exponentiated variables. An important example is 
topological string theory on Calabi--Yau threefolds, and indeed, peacock patterns can be observed in e.g. \cite{cms}. It would 
be very interesting to understand the resurgent structure in these examples, and work along this direction is in progress.

\subsection*{Acknowledgements} 
The authors would like to thank Jorgen Andersen, Maxim Kontsevich,
Pietro Longhi and Don Zagier for enlightening conversations. The work
of J.G.  has been supported in part by the NCCR 51NF40-182902 ``The
Mathematics of Physics'' (SwissMAP). The work of M.M. has been
supported in part by the ERC-SyG project ``Recursive and Exact New
Quantum Theory" (ReNewQuantum), which received funding from the
European Research Council (ERC) under the European Union's Horizon
2020 research and innovation program, grant agreement No. 810573.


\section{Borel resummation and Stokes automorphisms}
\label{sec.borel}

\subsection{Borel resummation}
\label{sub.borel}

In this section we briefly review the process of Borel resummation of
a factorially divergent series, its Laplace integral along rays and
the corresponding Stokes automorphism across a Stokes ray. The
material in this section is classical and well-known and is explained
in detail in the books~\cite{Costin:asymptotics,Miller:applied,msauzin},
and in the references therein. We will be following the physics convention
of Borel resummation as found for example in~\cite[Sec. 3.2]{mmbook}
and~\cite[Sec.42.5]{Zinn-Zustin:QFT}, which differs by a factor of
$\tau$ from the Borel resummation found in the math literature.

Borel resummation is a 2-step process to pass from a factorially
divergent series $F(\tau)$ to the analytic function $s(F)(\tau)$
defined in the right half-plane $\Re(\tau)>0$ summarized in the
following diagram
\begin{equation}
  \label{borel-sum}
  F(\tau) \leadsto 
  \wh{F}(\zeta) \leadsto 
  s(F)(\tau) \,.
\end{equation}
Here one starts with a Gevrey-1 a formal power series $F(\tau)$
\begin{equation}
  F(\tau) = \sum_{n=0}^\infty f_n \tau^{n} ,\quad f_n = O(C^n \, n!)
\end{equation}
and defines its Borel stransform $\wh{F}(\zeta)$ by
\begin{equation}
  \wh{F}(\zeta) = \sum_{n=0}^\infty \frac{f_n}{n!} \zeta^n \,.
\end{equation}
It follows that $\wh{F}$ is the germ of an analytic function at
$\zeta=0$. If it analytically continues to an $L^1$-analytic function
along the ray $\rho_\theta:=\re^{\ri\theta}\IR_+$ where
$\theta = \arg \tau$, we define its Laplace transform by 
\begin{equation}
  s_{\theta}(F)(\tau)=
  \int_0^\infty \wh{F}(\tau\zeta) \re^{-\zeta} \rd \zeta=
  \frac{1}{\tau} \int_{\rho_\theta}
  \wh{F}(\zeta)\re^{-\zeta/\tau} \rd \zeta
  \label{eq:brlsum}
\end{equation}
The function $s(F)(\tau)$ is often called the Borel resummation of the
formal power series $F$, and we often suppress the subscript $\theta=0$
when $\tau$ is real and positive.
If we vary $\theta = \arg \tau$ and we do not encounter singularities
of $\wh{F}$, the function $s_\th(F)(\tau)$ is locally analytic. Thus,
the problem is to understand the analytic continuation of $\wh{F}$ and
to analyze what happens to the Borel resummation $s_\theta(F)(\tau)$
when $\theta =\arg(\tau)$ crosses a Stokes ray, i.e., a ray in Borel
plane that contains one or more singularities of $\wh{F}$. This is
described by a Stokes automorphism.

\subsection{Stokes automorphism}
\label{sub.stokes}

We will specialize our discussion to the series of interest, namely to
the Borel transform $\wh{\Phi}(x;\zeta)$ of the vector of series
$\Phi(x;\tau)$. The singularities of $\Phi(x;\tau)$ are conjectured to
be in the set~\eqref{Phising} that generates a set of Stokes rays
whose complement is a countable union of open cones in Borel plane.
When $\th$ is in a fixed such cone $C$, the Laplace transform
$s_\th(\Phi)(x;\tau)$ depends on $C$ but not on $\th$.  To compare two
adjacent such cones, let $\iota_{\s,\s'}^{(\ell,k)}$ denote one of the
singularities of $\Phi_\s(x;\tau)$, $\theta$ denote its argument and
$\rho=\re^{\ri\theta}\BR_+$ denote the corresponding Stokes ray.  When
$x$ is generic, a Stokes ray contains a single singularity and the
Laplace integrals to the right and the left of $\rho$ are related by
\begin{equation}
  s_{\theta^+}(\Phi_\s)(x;\tau) =  s_{\theta^-}(\Phi_\s)(x;\tau) +
  \mc{S}_{\s,\s'}^{(\ell,k)}\tx^\ell \tq^k
  s_{\theta^-}(\Phi_{\s'})(x;\tau),
  \label{eq:Sauto}
\end{equation}
where $\mc{S}_{\s,\s'}^{(\ell,k)}$ is the Stokes constant. In matrix
form, the above formula reads
\begin{equation}
\label{lap2}
  s_{\theta^+}(\Phi)(x;\tau) =
  \mfS_{\theta}(\tx;\tq) s_{\theta^-}(\Phi)(x;\tau) 
\end{equation}
where
\begin{equation}
\label{Grho}
\mfS_\th(\tx;\tq)=
I + \mc{S}_{\s,\s'}^{(\ell,k)}\, \tx^\ell \, \tq^k \, E_{\s,\s'}
\end{equation}
where $E_{\s,\s'}$ is the elementary matrix with $(\s,\s')$-entry $1$
and all other entries zero.

More generally, consider two non-Stokes rays $\rho_{\th^+}$ and
$\rho_{\th^-}$ whose arguments satisfy $0 \leq \th^+-\th^- \leq
\pi$. Then, the Laplace integrals are related by
\begin{equation}
\label{lap3}
  s_{\theta^+}(\Phi)(x;\tau) =
  \mfS_{\theta^- \rightarrow \theta^+}(\tx;\tq) s_{\theta^-}(\Phi)(x;\tau) 
\end{equation}
where the Stokes matrices satisfy the factorization property
\begin{equation}
  \mfS_{\theta^- \rightarrow \theta^+ }(\tx,\tq) =
  \prod_{\th^- < \th < \th^+}^{\longleftarrow} \mfS_{\th}(\tx,\tq).
  \label{eq:SRR-fac}
\end{equation}
where the ordered product is taken over the Stokes rays in the cone
genarated by $\rho_{\th^-}$ and $\rho_{\th^+}$. This factorization is
well-known in the classical litetature on WKB (see for instance
Voros~\cite[p.228]{voros} who called it the ``radar method''  for
obvious visual reasons). In our case, there are four
special non-Stokes rays denoted by 
\begin{equation}
\label{4I}
I=\re^{\ri \epsilon}\BR_+,
\qquad II=\re^{\ri (\pi-\epsilon)}\BR_+,
\qquad III=\re^{\ri (\pi+\epsilon)}\BR_+,
\qquad IV=\re^{\ri (2\pi-\epsilon)}\BR_+
\end{equation}
(for $\epsilon>0$ and sufficently small) that belong to the four
distinguished cones (labeled $I,II,III,IV$) adjacent to the real axis
and free of Stokes lines. The corresponding Stokes matrices
\begin{equation}
  \label{SG}
  \ms{S}^+(\tx;\tq) =
  \mf{S}_{I\rightarrow II}(\tx;\tq)\, \mf{S}_{IV\rightarrow I}(\tx;\tq),
  \qquad
  \ms{S}^-(\tx;\tq) =
  \mf{S}_{III\rightarrow IV}(\tx;\tq)\, \mf{S}_{II\rightarrow III}(\tx;\tq)
\end{equation}
that swap two complementary and nearly horizontal half-planes
separated by a line $L$ are the ones that appear in
Conjecture~\ref{conj.asy}. They are related to the matrices
$M_\rho$ in the second part of that conjecture by
\begin{align}
  \mfS_{I\rightarrow II}(\tx;\tq)
  &= (M_{II}(\tx;\tq))^{-1}\cdot M_{I}(\tx;\tq),\quad |\tq|<1\nn
    \mfS_{III\rightarrow IV}(\tx;\tq)
  &= (M_{IV}(\tx;\tq^{-1}))^{-1}\cdot M_{III}(\tx;\tq^{-1}),\quad |\tq|<1\nn
    \mfS_{IV\rightarrow I}(\tx;\tq)
  &= (M_{I}(\tx;\tq))^{-1}\cdot M_{IV}(\tx;\tq),\nn
    \mfS_{II\rightarrow III}(\tx;\tq)
  &= (M_{III}(\tx;\tq))^{-1}\cdot M_{II}(\tx;\tq) \,.
    \label{4G}
\end{align}

We now come to an important feature of our resurgent series,
a unique factorization property for the Stokes matrices reminiscent
of the ``stability data'' description of DT-invariants in
Kontsevich-Soibelman~\cite{KS,KS:stability, KS:wall-crossing}), and
of the properties of BPS spectrum generators in
Gaiotto-Moore-Neitzke~\cite{gmn, gmn2, gmn3}. 

\begin{lemma}
  \label{lem.fac}
  $\ms{S}$
  uniquely determines $\mfS_{\th}$ for all $\th$. 
\end{lemma}

\begin{proof}
  Without loss of generality, we will show that $\ms{S}^+$ uniquely
  determines the Stokes matrices $\mfS_\th$ for all $\th$ such that
  $-\varepsilon < \th < \pi - \varepsilon$ for $\varepsilon>0$ and
  sufficiently small. We have
  \begin{equation}
    \label{SScal}
    \ms{S}^+(\tx;\tq) = \prod_{\s,\s',k,\ell}^{\longleftarrow}
    (I + \calS_{\s,\s'}^{(\ell,k)} \tx^\ell \, \tq^k E_{\s,\s'})
  \end{equation}
  where the product is over all the singularities above the line $L$.
  The entries of the above matrices are in the ring
  $\BZ[\tx^\pm][[\tq]]$.  For each fixed natural number $N$, there are
  only finitely many horizontal lines of singularities in Borel plane,
  of height at most $N$ and within those, there are finitely many
  $\tx$-dots. It follows by induction on $k$ that the finite
  collection $\{\calS_{\s,\s'}^{(\ell,k)}\}_\ell$ is uniquely
  determined from $\ms{S}^+(\tx;\tq)$.
\end{proof}

It follows that we can repackage the information of the Stokes
constants in two matrices $\calS^+$ and $\calS^-$ defined by

\begin{equation}
  \label{calS}
  \calS^\pm(\tx;\tq) = \sum_{\ell,k} \calS_{\s,\s'}^{(\ell,k)}
  \tx^\ell \, \tq^k E_{\s,\s'}
\end{equation}
where the sum in $\calS^+$ (resp. $\calS^-$) is over the singularities
above (resp., below) $L$. The matrices $\calS^\pm(\tx;\tq)$ appear to
have some positivity properties; see Sections~\ref{sub.41stokes}
and~\ref{sub.52stokes} below for the $\knot{4}_1$ and the $\knot{5}_2$
knots.
    

\section{A summary of the story when $u=0$}
\label{sec.summary}

In this section we recall briefly the results
from~\cite{GGM:resurgent} for our two sample hyperbolic knots, the
$\knot{4}_1$ and the $\knot{5}_2$ knot.

\subsection{The $\knot{4}_1$ knot when $u=0$}
\label{sub.41x=1}

The state integral of the $\knot{4}_1$ at $u=0$ is given by
\begin{equation}
  \label{Z410}
  Z_{\knot{4}_1}(0;\tau) = 
  \int_{\BR+\ri 0} \Phi_\bb(v)^2 \, \re^{-\pi \ri v^2} \rd v \,.
\end{equation}
The critical points of the integrand are the logarithms of the
solutions $\xi_1= {\rm e}^{2 \pi \ri/6}$ and
$\xi_2= \re^{-2 \pi \ri/6}$ of the polynomial equation
\begin{equation}
  \label{41xi}
  (1-y)(1-y^{-1})=1 \,.
\end{equation}
The labeling set $\calP=\{\s_1,\s_2\}$, where $\s_1$ corresponds to
the geometric representation of the $\knot{4}_1$ and $\s_2$ to the
complex conjugate of the geometric representation. Observe that
$\xi_1$ (resp., $\xi_2$) lie in the trace field $\BQ(\sqrt{-3})$
(resp., its complex conjugate) of the $\knot{4}_1$ knot, where
$\BQ(\sqrt{-3})$ is a subfield of the complex numbers with $\sqrt{-3}$
taken to have positive imaginary part.

The first ingredient is a vector of formal power series 
\be
\Phi(\tau)=
\begin{pmatrix} \Phi_{\s_1}(\tau)\\ \Phi_{\s_2}(\tau) \end{pmatrix} 
\ee
defined by the asymptotic expansion of the state-integral~\eqref{Z410}
at each of the two critical points, and which has the form
\be
\Phi_{\s}(\tau) = \exp\left( { V(\s) \over 2\pi \ri \tau} \right)
\varphi_\s(\tau),
\ee
satisfies the symmetry
$\Phi_{\s_2}(\tau) = {\rm \ri} \Phi_{\s_1}(-\tau)$, where 
\begin{equation}
  \label{volume-41}
  V(\s_1)= \ri \text{Vol}(\knot{4}_1)
  = \ri\, 2 {\rm Im}\, {\rm Li}_2(\re^{\ri \pi/3})
  = \ri\, 2.029883\dots \,.
\end{equation}
with $\text{Vol}(\knot{4}_1)$ being the hyperbolic volume of
$\knot{4}_1$ and the first few terms of
$\varphi_{\s_1}(\tau/(2\pi \ri)) \in 3^{-1/4} \BQ(\sqrt{-3})[[\tau]]$
are given by
\begin{equation}
  \label{varphi41}
  \varphi_{\s_1}\left(\frac{\tau}{2\pi \ri}\right) =
  \frac{1}{\sqrt[4]3}\, 
  \Bigl(1 \+ \frac{11\tau}{72\sqrt{-3}}
  \+ \frac{697\tau^2}{2\,(72\sqrt{-3})^2}
  \+ \frac{724351\tau^3}{30\,(72\sqrt{-3})^3} \+\cdots\Bigr)\,. 
\end{equation}

The second ingredient is a vector
$G(q)=\begin{psmall} G^0(q) \\ G^1(q)
\end{psmall}$ of $q$-series defined for $|q|<1$ by
\begin{subequations}
\begin{align}
\label{gm0}
G^0(q) &=\sum_{n=0}^\infty (-1)^n \frac{q^{n(n+1)/2}}{(q;q)_n^2}
\\
\label{Gm0}
G^1(q) &=\sum_{n=0}^\infty (-1)^n \frac{q^{n(n+1)/2}}{(q;q)_n^2}
\left(E_1(q) + 2 \sum_{j=1}^n \frac{1+q^j}{1-q^j} \right) \,,
\end{align}
\end{subequations}
where $E_1(q)=1-4\sum_{n=1}^\infty q^n/(1-q^n)$ is the Eisenstein
series, and extended to $|q|>1$ by $G^0(q^{-1})=G^0(q)$ and
$G^1(q^{-1})=-G^1(q)$.  These series are motivated by, and appear in,
the factorization of the state-integral of the $\knot{4}_1$ knot given
in~\cite[Cor.1.7]{GK:qseries}
\begin{equation}
  \label{Z410-fac}
  Z_{\knot{4}_1}(0;\tau) =
  - \frac{\ri}{2} \left(\frac{q}{\tq}\right)^{\frac{1}{24}} 
  \left(
    \sqrt{\tau} G^0(\tq) G^1(q) - \frac{1}{\sqrt{\tau}} G^0(q) G^1(\tq) 
  \right), \qquad (\tau \in \BC\setminus\BR)
\end{equation}
where 
\begin{equation}
\label{qqt}
q=\re^{2 \pi \ri \tau}, \qquad \tq=\re^{-2\pi \ri/\tau} \,.
\end{equation}
The above factorization follows by applying the residue theorem to the
integrand of~\eqref{Z410}, a meromorphic function of $v$ with
prescribed zeros and poles. In particular, the integrand
of~\eqref{Z410} determines the $q$-hypergeometric formula for the
vector $G(q)$ of $q$-series.  Below, given a $q$-series $H(q)$ defined
on $|q| \neq 1$, we denote by $h(\tau)= H(\re^{2\pi \ri\tau})$ the
corresponding holomorphic function in $\BC\setminus\BR$.

The vector $G(q)$ of $q$-series and the vector of asymptotic
series $\Phi(\tau)$ come together when we consider the asymptotics
of $\diag(\frac{1}{\sqrt{\tau}},\sqrt{\tau}) g(\tau)$
in the $\tau$-plane (as was studied in~\cite{GZ:qseries}) and compare
them with the Borel summed vector $\Phi$.
Recall that when the Borel transform of an asymptotic series has
singular points $\iota_i$ in the Borel plane, the rays (Stokes rays)
emanating from the origin with angle $\theta = \arg \iota_i$
divide the complex plane into different sectors.  When one crosses
into neighboring sectors, the Borel sum of the asymptotic series
undergoes Stokes automorphism.
In the case of the vector of asymptotic series $\Phi(\tau)$, the
singularities of the Borel transforms of its two component asymptotic
series are located at
\begin{equation}
  \iota_{i,j} = \frac{V(\s_i) - V(\s_j)}{2\pi\ri}\,,\quad i,j=1,2,\;i\neq j,
\end{equation}
as well as
\begin{equation}
  2\pi\ri k,\; \iota_{i,j}+2\pi\ri k\,,
  \quad k\in \IZ_{\neq 0},
\end{equation}
forming vertical towers as illustrated in
Figure~\ref{fig:sings-u0-41}. In particular, the two singularities
$\iota_{1,2}, \iota_{2,1}$ are on the positive and the negative real
axis. We pick out four sectors which separate 
the two singularities on the real axis and all the others, and label
them by $I,II,III,IV$, as illustrated in Figure~\ref{fig:secs-u0-41}.
The relation between the vector $G(q)$ and the Borel summed vector
$\Phi(\tau)$ depend on the sector $R$.
In~\cite{GGM:resurgent}, we found out
that we do not get an agreement, but rather both sides agree up to
powers of the exponentially small quantity $\tq$, and what is more,
several coefficients of those powers were numerically recognized to be
integers. In other words, we found that
\begin{equation}
\label{MR}
\diag(\frac{1}{\sqrt{\tau}},\sqrt{\tau}) g(\tau)
= M_R (\tq) \, s_R(\Phi)(\tau) \,.
\end{equation}
where $\diag(v)$ denotes the diagonal matrix with diagonal given by
$v$ and $M_R(q)$ is a matrix of $q$-series with integer coefficients.

\begin{figure}[htpb!]
\leavevmode
\begin{center}
  \includegraphics[height=7cm]{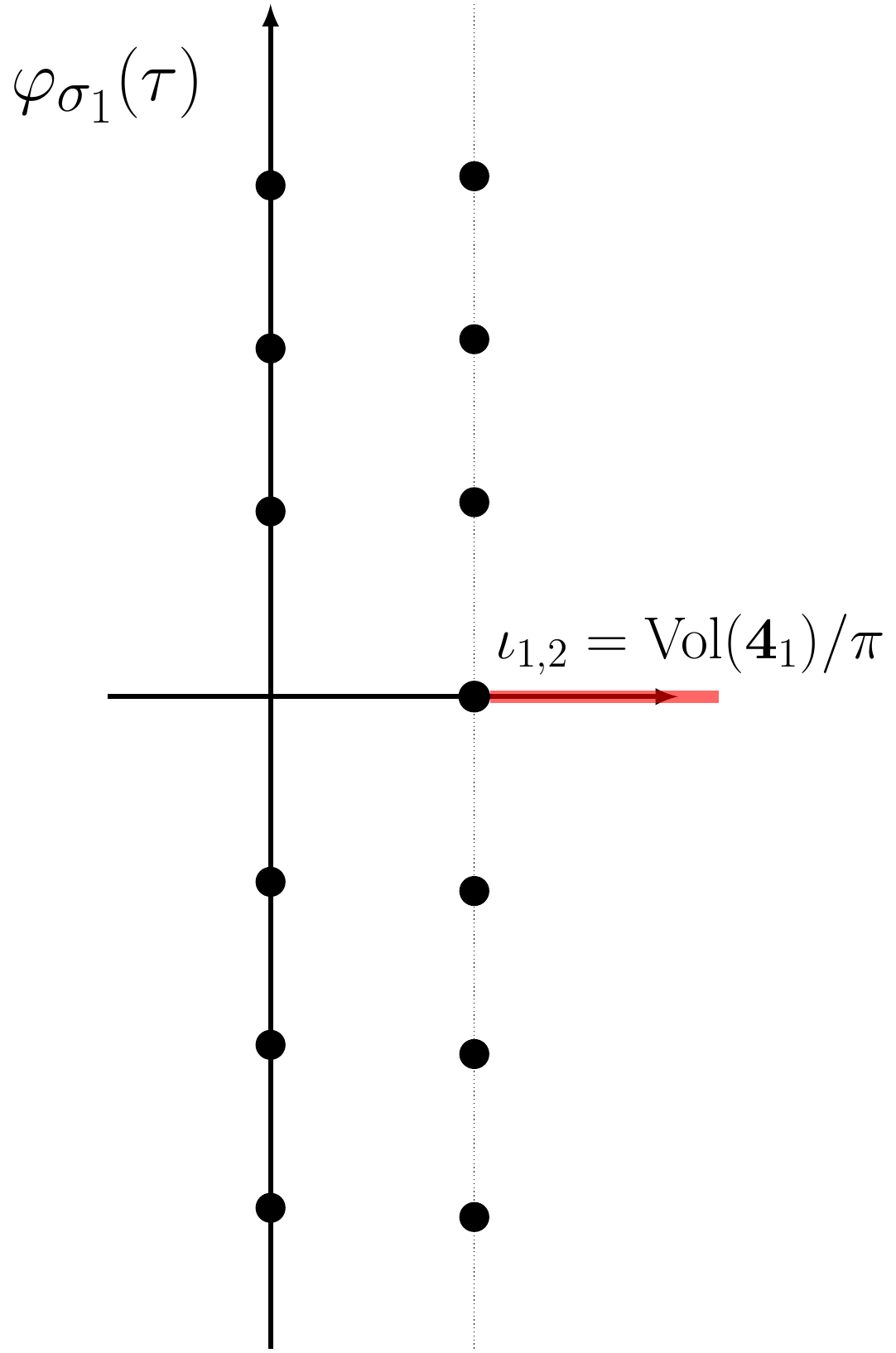} \hspace{5ex}
  \includegraphics[height=7cm]{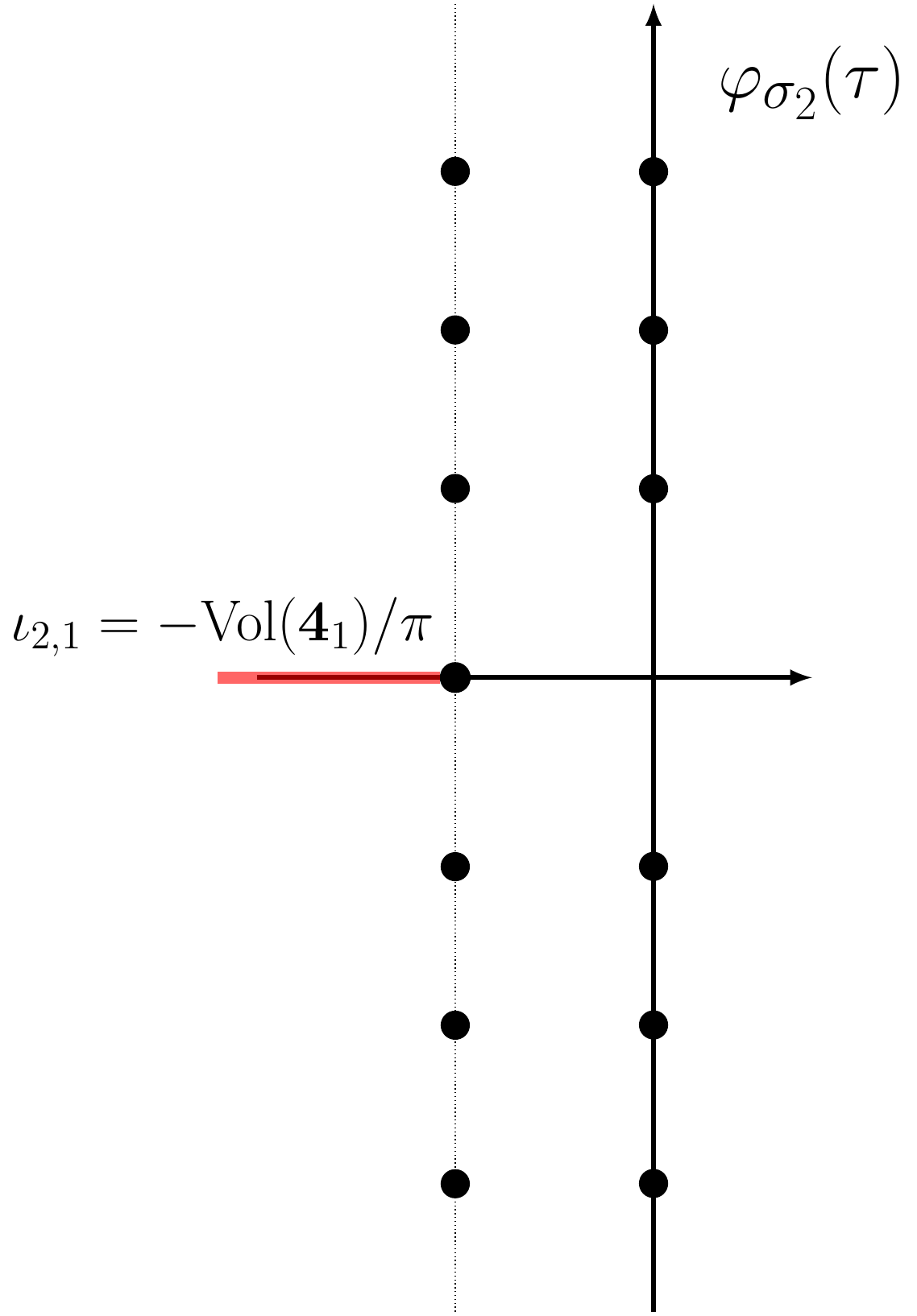}
\end{center}
\caption{The singularities in the Borel plane for the series
  $\varphi_{\sigma_{j}} (0;\tau)$ for $j=1,2$ of knot $\knot{4}_1$.}
\label{fig:sings-u0-41}
\end{figure}

\begin{figure}[htpb!]
\leavevmode
\begin{center}
\includegraphics[height=7cm]{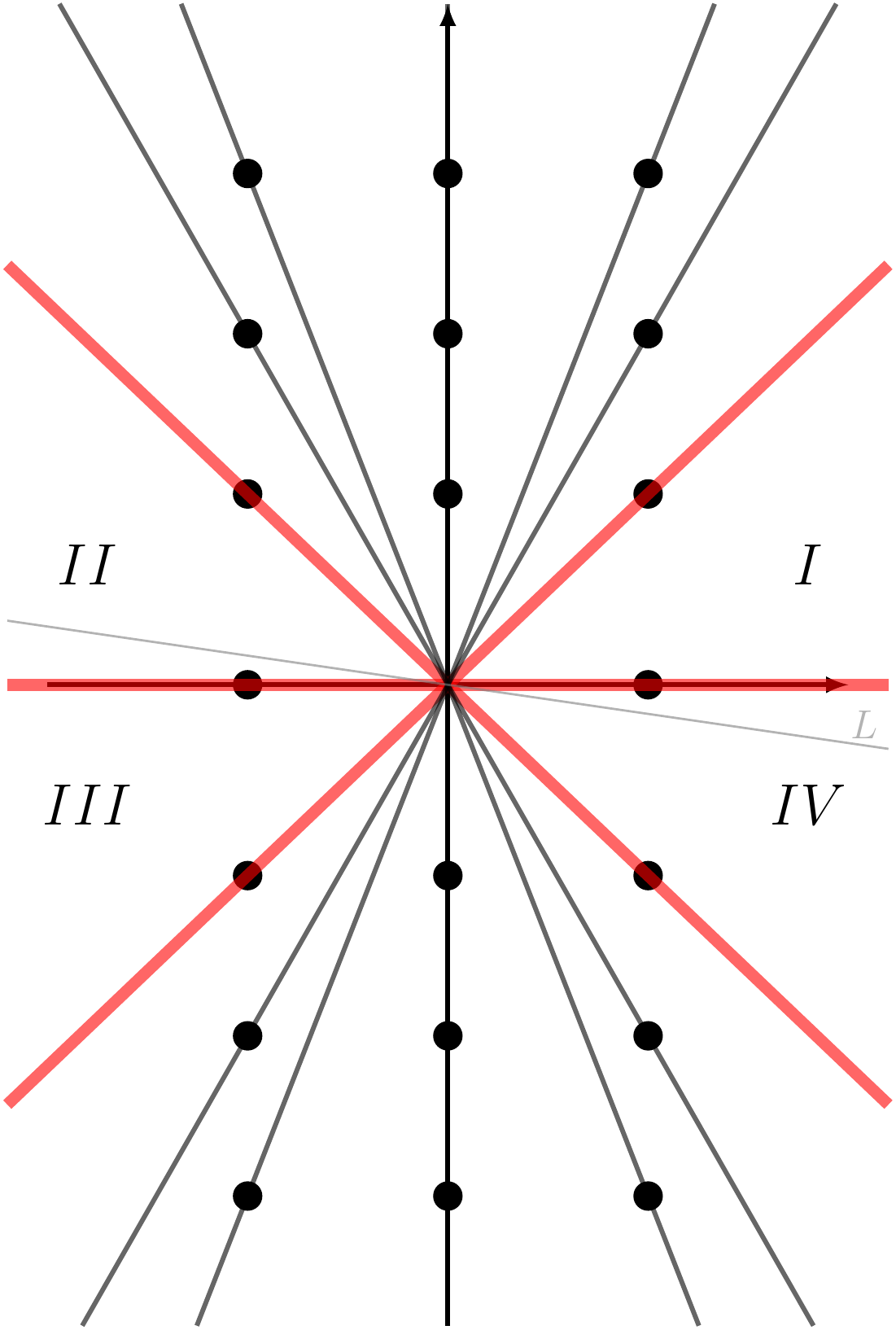}
\end{center}
\caption{Four different sectors in the $\tau$-plane for $\Phi(0;\tau)$
  of knot $\knot{4}_1$.}
\label{fig:secs-u0-41}
\end{figure} 

To identify the matrices $M_R$, we used the third ingredient,
namely the linear $q$-difference equation
\begin{equation}
\label{41qdiff}
y_{m+1}(q) -(2-q^m) y_m(q) + y_{m-1}(q)=0 \qquad (m \in \BZ) \,.
\end{equation}
It has a fundamental solution set given by the columns of the following
matrix 
\begin{equation}
\label{Gfund}
W_m(q) =
\begin{pmatrix} G^0_m(q) & G^1_m(q) \\ G^0_{m+1}(q) & G^1_{m+1}(q)
\end{pmatrix}, \qquad (|q| \neq 1)
\end{equation} 
where $G_m(q)=\begin{psmall} G^0_m(q) \\ G^1_m(q) \end{psmall}$,
and $G^0_m(q)$ and $G^1_m(q)$ are defined by 
\begin{subequations}
\begin{align}
\label{gm}
G^0_m(q) &=\sum_{n=0}^\infty (-1)^n \frac{q^{n(n+1)/2+m n}}{(q;q)_n^2}
\\
\label{Gm}
G^1_m(q) &=\sum_{n=0}^\infty (-1)^n \frac{q^{n(n+1)/2+m n}}{(q;q)_n^2}
\left(2m+ E_1(q) + 2 \sum_{j=1}^n \frac{1+q^j}{1-q^j} \right) \,,
\end{align}
\end{subequations}
for $|q|<1$ and extended to $|q|>1$ by
$G^j_m(q^{-1}) = (-1)^j G^j_m(q)$.  Observe that $G_0(q)=G(q)$, the
vector that appears in the factorization~\eqref{Z410-fac} of the
state-integral $Z_{\knot{4}_1}(0;\tau)$.  The matrix $W_m(q)$
of holomorphic functions in $|q| \neq 1$ satisfies several
properties summarized in the following theorem.

\begin{theorem}
  \label{thm.410}
  $W_m(q)$ is a fundamental solution of the linear $q$-difference
  equation~\eqref{41qdiff} that has constant determinant
  \begin{equation}
\label{det41}
\det(W_m(q))=2 \,,
\end{equation}
satisfies the symmetry 
\begin{equation}
  \label{W41inv}
  W_m(q^{-1}) = W_{-m}(q)
  \begin{pmatrix} 1 & 0 \\ 0 & -1 \end{pmatrix} \,,
\end{equation}
the orthogonality property
\begin{equation}
\label{WWT41b}
\frac{1}{2} W_m(q)
\begin{pmatrix} 0 & 1 \\ -1 & 0 \end{pmatrix}
W_{m}(q)^T =
\begin{pmatrix} 0 & 1 \\ -1 & 0 \end{pmatrix} 
\end{equation}
as well as
\begin{equation}
\label{WWT41}
\frac{1}{2} W_m(q)
\begin{pmatrix} 0 & 1 \\ -1 & 0 \end{pmatrix}
W_{\ell}(q)^T \in \SL(2,\BZ[q^\pm]) 
\end{equation}
for all integers $m, \ell$ and for $|q| \neq 1$. 
\end{theorem}

\begin{conjecture} 
\label{conj.41M0}  
Equation~\eqref{MR} holds where the matrices $M_R(q)$ are
given in terms of $W_{-1}(q)$ as follows
\begin{subequations}
\begin{align}
  \label{M41I}
  M_I(q)
  &= W_{-1}(q)^T \,
    \begin{pmatrix}
      0 & -1 \\
      1 & -1
    \end{pmatrix}, & |q|<1 \,,
  \\
  \label{M41II}
  M_{II}(q)
  &= \begin{pmatrix}
    1 & 0 \\
    0 & -1
  \end{pmatrix} \, W_{-1}(q)^T \,
        \begin{pmatrix}
          1 & 0 \\
          1 &-1
        \end{pmatrix}, & |q| <1 \,,
  \\
  \label{M41III}
  M_{III}(q)
  &=
    W_{-1}(q^{-1})^T \,
    \begin{pmatrix}
      1 & 0 \\
      1 & 1 \end{pmatrix}, & |q| >1 \,,
  \\
  \label{M41IV}
  M_{IV}(q)
  &=\begin{pmatrix}
    1 & 0 \\
    0 & -1
  \end{pmatrix} \,
        W_{-1}(q^{-1})^T \,
        \begin{pmatrix}
          0 & 1 \\
          1 & 1
        \end{pmatrix}, & |q|>1 \,.
\end{align}
\end{subequations}
\end{conjecture}

Assuming the above conjecture, we can now describe completely the
resurgent structure of $\Phi(\tau)$.  The Stokes matrices are given by
\begin{equation}
  \label{Spm1}
  \ms{S}^+(q)= 
  \mathfrak{S}_{I \rightarrow II}(q)
  \mfS_{IV\rightarrow I}, \qquad
  \ms{S}^-(q)= 
  \mathfrak{S}_{III \rightarrow IV}(q)
  \mfS_{II\rightarrow III}\,,
\end{equation}
where 
\begin{subequations}
  \begin{align}
  \label{Spm3}
    \mathfrak{S}_{I \rightarrow II}(q)
    &= M_{II}(q)^{-1} M_I(q)
    & \mathfrak{S}_{III \rightarrow IV}(q)
    &= M_{IV}(q^{-1})^{-1}
  M_{III}(q^{-1}) \\
  \label{Spm4}
    \mfS_{IV\rightarrow I}
    & = M_I(q)^{-1}M_{IV}(q)
    & \mfS_{II\rightarrow III}
    &= M_{III}(q)^{-1}M_{II}(q) \,.
\end{align}
\end{subequations}
(Compare with Equations~\eqref{SG} and~\eqref{4G} after
we set $\tx=1$ and replace $\tq$ by $q$).
Note that since $M_I(q), M_{II}(q)$ and
$M_{III}(q), M_{IV}(q)$ are given respectively as $q$-series and
$q^{-1}$-series in \eqref{M41I},\eqref{M41II} and
\eqref{M41III},\eqref{M41IV}, analytic continuation as discussed below
\eqref{Gm} is needed when one computes
$\mfS_{IV\rightarrow I},\mfS_{II\rightarrow II}$ in \eqref{Spm4}.
Using~\eqref{det41}--\eqref{WWT41} we can express the answer in terms
of $W_{-1}(q)$. Explicitly, the Stokes matrices are given by
\begin{subequations}
\begin{align}
  \label{S41p}
  \ms{S}^+(q)
  &={1\over 2} \begin{pmatrix} 0 & -1 \\ 1 & 1 \end{pmatrix}
  W_{-1}(q) \begin{pmatrix} 0 & 1 \\ 1 & 0 \end{pmatrix}W_{-1}(q)^T
  \begin{pmatrix} 0 & -1 \\ 1 & 2 \end{pmatrix}, & |q|<1\,, \\
  \label{S41m}
  \ms{S}^-(q)
  &= {1\over 2}
  \begin{pmatrix} -1 & -1 \\ 0 & 1 \end{pmatrix} W_{-1}(q)
  \begin{pmatrix} 0 & 1 \\ 1 & 0 \end{pmatrix}W_{-1}(q)^T
  \begin{pmatrix} 1 & 0 \\ -2 & 1 \end{pmatrix}, & |q|<1 \,.
\end{align}
\end{subequations}
In the $q\rightarrow 0$ limit,
\begin{equation}
  \ms{S}^+(0) =
  \begin{pmatrix}
    1&3\\0&1
  \end{pmatrix},\quad
  \ms{S}^-(0) =
  \begin{pmatrix}
    1&0\\-3&1
  \end{pmatrix}
  \label{eq:S0u041}
\end{equation}
whose off-diagonal entries $-3,+3$ are Stokes constants associated to
the singularities $\iota_{2,1}$ and $\iota_{1,2}$ on the negative and
positive real axis respectively, and they agree with the matrix of
integers obtained numerically in~\cite{gh-res,GZ:kashaev}. In
addition, we can assemble the Stokes constants into
the matrix $\calS$ of Equation~\eqref{calS} (after we set $\tx=1$
and replace $\tq$ by $q$). The resulting matrix $\mc{S}^+(q)$
has entries in $q\BZ[[q]]$, and we find
\begin{align}
  \mc{S}^+_{\s_1,\s_1}(q) =
  &\ms{S}^+(q)_{1,1}-1\nn=
  &-8 q - 9 q^2 + 18 q^3 + 46 q^4 + 90 q^5 + 62 q^6+\cO(q^7),\\
  \mc{S}^+_{\s_1,\s_2}(q) =
  &\ms{S}^+(q)_{1,2}/\ms{S}^+(q)_{1,1}-\mc{S}^{(0)}_{\s_1,\s_2}\nn=
  &9 q + 75 q^2 + 642 q^3 + 5580 q^4 + 48558 q^5 + 422865 q^6+\cO(q^7),\\
  \mc{S}^+_{\s_2,\s_1}(q) =
  &\ms{S}^+(q)_{2,1}/\ms{S}^+(q)_{1,1}\nn=
  &-9 q - 75 q^2 - 642 q^3 - 5580 q^4 - 48558 q^5 - 422865 q^6+\cO(q^7),\\
  \mc{S}^+_{\s_2,\s_2}(q) =
  &\ms{S}^+(q)_{2,2} -1-
    \ms{S}^+(q)_{1,2}\ms{S}^+(q)_{2,1}/\ms{S}^+(q)_{1,1}\nn=
  &8 q + 73 q^2 + 638 q^3 + 5571 q^4 + 48538 q^5 + 422819 q^6+\cO(q^7).
\end{align}
We notice the symmetry
\begin{equation}
  \mc{S}^{(k)}_{1,2}=
  -\mc{S}^{(k)}_{2,1},\quad \text{for } k\in\IZ_{>0}\,,
\end{equation}
which is due to the reflection property
$\varphi_{\s_1}(-\tau^*)= \varphi_{\s_2}(\tau)^*$ of the asymptotic
series.  Also experimentally it appears that the entries of the matrix
$\mc{S}^+(q) = (\mc{S}^+_{\s_i,\s_j}(q))$ (except the upper-left one)
are (up to a sign) in $\IN[[q]]$.
Similarly we can extract the Stokes constants
$\mc{S}^{(-k)}_{\s_i,\s_j}$ associated to the singularities in the
lower half planes and collect them in $q^{-1}$-series
$\mc{S}^-_{\s_i,\s_j}(q^{-1})$ accordingly, and we find
\begin{equation}
  \mc{S}_{\s_i,\s_j}^{(-k)} = -\mc{S}^{(+k)}_{\s_j,\s_i},\;i\neq j\,,
  \quad\text{and}\quad \mc{S}_{\s_1,\s_1}^{(-k)}  =
  \mc{S}_{\s_2,\s_2}^{(+k)},\;\; \mc{S}_{\s_2,\s_2}^{(-k)}  =
  \mc{S}_{\s_1,\s_1}^{(+k)},\quad \text{for } k\in \IZ_{>0}.
\end{equation}

A nontrivial consistency check in the above calculation is that the
matrices $\mfS_{IV\rightarrow I}(q)$ and $\mfS_{II\rightarrow III}(q)$
should come out to be independent of $q$, and coincide with
$\ms{S}^-(0)$ and $\ms{S}^+(0)$. That is exactly what we find.

The fourth and last ingredient, which makes a full circle of ideas, is
the descendant state-integral of the $\knot{4}_1$ knot
\begin{equation}
\label{Z410-desc}
Z_{\knot{4}_1,m,\mu}(0;\tau) = 
\int_{\mc{D}} \Phi_\bb(v)^2 \, \re^{-\pi \ri v^2
+ 2\pi(m \bb - \mu \bb^{-1})v} \, \rd v \qquad
(m, \mu \in \BZ) \,.
\end{equation}
The integration contour $\mc{D}$ asymptotes at infinity to the horizontal line
$\text{Im}{v} = v_0$ with $v_0 >|\text{Re}(m\bb-\mu\bb^{-1})|$ but is deformed
near the origin so that all the poles of the quantum dilogarithm
located at
\begin{equation}
  c_\bb + \ri\bb r +\ri \bb^{-1}s,\quad r,s\in\IZ_{\geq 0},
\end{equation}
are above the contour.
The integral $Z_{\knot{4}_1,m,\mu}(0;\tau)$ is a holomorphic function of $\tau \in \BC'$ that coincides with
$Z_{\knot{4}_1}(0;\tau)$ when $m=\mu=0$ and can be expressed bilinearly
in $G_m(q)$ and $G_\mu(\tq)$ as follows
\begin{equation}
  \label{410-desc-fac}
  Z_{\knot{4}_1,m,\mu}(0;\tau) 
  =
  (-1)^{m-\mu+1}\ri q^{\frac{m}{2}}\tq^{\frac{\mu}{2}} 
  \left(\frac{q}{\tq}\right)^{\frac{1}{24}} \frac{1}{2}
  \left(\sqrt{\tau} \,G^0_\mu(\tq) G^1_m(q) - \frac{1}{\sqrt{\tau}}
    G^1_\mu(\tq) G^0_m(q)\right) \,.
\end{equation}
It follows that the matrix-valued function 
\begin{align}
  \label{W410-desc}
  W_{m,\mu}(\tau) = 
    (W_\mu(\tq)^T)^{-1}
    \begin{pmatrix}
      1/\sqrt{\tau} & 0 \\ 0 & \sqrt{\tau}
    \end{pmatrix} W_m(q)^T
\end{align}
defined for $\tau=\BC\setminus\BR$, has entries given by the
descendant state-integrals (up to multiplication by a prefactor
of~\eqref{410-desc-fac})
and hence extends to a holomorphic function of $\tau \in \BC'$ for all
integers $m$ and $\mu$. Using this for $m=-1$ and $\mu=0$ and the
orthogonality relation~\eqref{WWT41b}, it follows that we can express
the Borel sums of $\Phi(\tau)$ in a region $R$ in terms of descendant
state-integrals and hence, as holomorphic functions of $\tau \in \BC'$
as follows. For instance, in the region $I$ we have
\begin{equation}
  s_I(\Phi)(\tau) = M_I(\tq)^{-1}
  \begin{pmatrix}
    \frac{1}{\sqrt{\tau}} g^0_0(\tau) \\
    \sqrt{\tau} g^1_0(\tau)
  \end{pmatrix}
  = \ri\left(\frac{q}{\tq}\right)^{-\frac{1}{24}}
  \begin{pmatrix}
    Z_{\knot{4}_1,0,0}(0;\tau) -\tq^{1/2} Z_{\knot{4}_1,0,-1}(0;\tau) \\ 
    Z_{\knot{4}_1,0,0}(0;\tau)
\end{pmatrix}
\end{equation}
  
This completes the discussion of $u=0$ for the $\knot{4}_1$ knot.

\subsection{The $\knot{5}_2$ knot when $u=0$}
\label{sub.52x=1}

The state integral of the $\knot{5}_2$ at $u=0$ is given by
\begin{equation}
  \label{Z520}
  Z_{\knot{5}_2}(0;\tau) = 
  \int_{\BR+\ri 0} \Phi_\bb(v)^3 \, \re^{-2\pi \ri v^2} \rd v \,.
\end{equation}
The critical points of the integrand are the logarithms of the
solutions $\xi_1 \approx 0.78492 + 1.30714 \ri$,
$\xi_2 \approx 0.78492 - 1.30714 \ri$ and $\xi_3 \approx 0.43016$ of
the polynomial equation
\begin{equation}
  \label{52xi}
  (1-y)^3=y^2 \,.
\end{equation}
The trace field of the $\knot{5}_2$ knot is $\BQ(\xi_1)$, the cubic
field of discriminant $-23$, which has three complex embeddings
labeled by $\s_j$ for $j=1,2,3$ corresponding to the $x_j$, and the
labeling set is $\calP=\{\s_1,\s_2,\s_3\}$, where $\s_1$ corresponds
to the geometric representation of the $\knot{5}_2$ knot, $\s_2$ to
the complex conjugate of the geometric representation and $\s_3$ for
the corresponding real representation.

The first ingredient is a vector of formal power series
\be\Phi(\tau)=
\begin{pmatrix} \Phi_{\s_1}(\tau)\\ \Phi_{\s_2}(\tau) \\
\Phi_{\s_3}(\tau) \end{pmatrix}
\ee
where 
\begin{equation}
  \label{volume-52}
  \Im\,V(\s_1) = -\Im\,V(\s_2) =\text{Vol}(\knot{5}_2)= 2.82812\ldots
\end{equation}
is the hyperbolic volume of the knot $\knot{5}_2$ and the first few
terms of
$\varphi_{\s_j}(\tau/(2\pi \ri)) \in \delta_j^{-1/2}
\BQ(\xi_j)[[\tau]]$ are given by 
{\small
\begin{align}
\label{varphi52}
  \varphi_{\s_j}\left(\frac{\tau}{2\pi \ri}\right) =
  &\left(\frac{-3\xi_j^2+3\xi_j-2}{23}\right)^{1/4}
    \left(1+\frac{-242\xi^2+209\xi-454}{2^2\cdot 23^2}\tau
    +\frac{12643\xi^2-22668\xi+25400}{2^5\cdot 23^3}\tau^2\right.
  \\
  \notag
  &\left.+\frac{-35443870\xi^2+85642761\xi-164659509}
    {2^7\cdot 3\cdot 5\cdot 23^5}\tau^3 + \dots \right) \,.    
\end{align}
}

The second ingredient is two vectors
$H^+(q)= (H^+_0(q), H^+_1(q), H^+_2(q))^T$ and
$H^-(q) = (H^-_0(q),H^-_1(q),H^-_2(q))^T$ of $q$-series defined for
$|q|<1$ by
\begin{subequations}
\begin{align}
\label{Hp0}
  H^+_0(q)
  &=\sum_{n=0}^\infty \frac{q^{n(n+1)}}{(q;q)_n^3}\,,
\\
\label{Hp1}
  H^+_1(q)
  &=\sum_{n=0}^\infty \frac{q^{n(n+1)}}{(q;q)_n^3}
    \left(1+2n-3E_1^{(n)}(q)\right) \,,
  \\
  \label{Hp2}
  H^+_2(q)
  &=\sum_{n=0}^\infty \frac{q^{n(n+1)}}{(q;q)_n^3}
    \left((1+2n-3E_1^{(n)}(q))^2-3E_2^{(n)}(q)-\frac{1}{6}E_2(q)\right) \,.
\end{align}
\end{subequations}
and
\begin{subequations}
\begin{align}
\label{Hn0}
  H^-_0(q)
  &=\sum_{n=0}^\infty (-1)^n\frac{q^{\frac{1}{2}n(n+1)}}{(q;q)_n^3}\,,
\\
\label{Hn1}
  H^-_1(q)
  &=\sum_{n=0}^\infty (-1)^n\frac{q^{\frac{1}{2}n(n+1)}}{(q;q)_n^3}
    \left(\frac{1}{2}+n-3E_1^{(n)}(q)\right) \,,
  \\
  \label{Hn2}
  H^-_2(q)
  &=\sum_{n=0}^\infty (-1)^n\frac{q^{\frac{1}{2}n(n+1)}}{(q;q)_n^3}
  \left(\big(\frac{1}{2}+n-3E_1^{(n)}(q)\big)^2
    -3E_2^{(n)}(q)-\frac{1}{12}E_2(q)\right) \,,
\end{align}
\end{subequations}
where
\begin{equation}
  E_l^{(n)}(q) = \sum_{s=1}^\infty \frac{s^{l-1}q^{s(n+1)}}{1-q^s}.
  \label{eq:Eln}
\end{equation}
The two sets of $q$-series can be extended to $|q|>1$ and are in fact
related by $H^+_k(q^{-1})=(-1)^k H^-_k(q)$, and define a $q$-series
$H_k(q)$ for $|q| \neq 1$ by
\begin{equation}
  H_k(q) = \begin{cases}
    H^+_k(q) & \qquad |q|<1 \\
    (-1)^k H^-_k(q^{-1}) & \qquad |q|>1 \,.
  \end{cases}
\end{equation}
Likewise, we define holomorphic functions $h_k(\tau)$ in $\BC\setminus\BR$ by
\begin{equation}
  h_k(\tau) =
  \begin{cases}
    H^+_k(\re^{2\pi\ri \tau}), & \quad \Im(\tau) > 0\\
    (-1)^k H^-_k(\re^{-2\pi\ri \tau}), & \quad \Im(\tau) <0
  \end{cases}, \quad k=0,1,2\,. 
\end{equation}
These series appear in the factorization of the state-integral of the
$\knot{5}_2$ knot given in~\cite[Cor.1.8]{GK:qseries}
\begin{equation}
  \label{Z520-fac}
  Z_{\knot{5}_2}(0;\tau) =
  -\frac{\re^{\frac{\pi\ri}{4}}}{2} \left(\frac{q}{\tq}\right)^{\frac{1}{8}} 
  \left(
    \tau H_0^-(\tq) H_2^+(q) -2 H^-_1(\tq) H^+_1(q)+ \tau^{-1} H^-_2(\tq) H^+_0(q) 
  \right), \qquad (\tau \in \BC\setminus\BR).
\end{equation}
The above factorization follows by applying the residue theorem to the
integrand of~\eqref{Z520}, a meromorphic function of $v$ with
prescribed zeros and poles. In particular, the integrand
of~\eqref{Z520} determines the $q$-hypergeometric formula for the
vectors $H^+(q)$, $H^-(q)$ of
$q$-series.  

As in the case of the $\knot{4}_1$ knot, multiplying the vector
$h(\tau)^T=(h_0(\tau) , h_1(\tau), h_2(\tau) )^T$ by the automorphy
factors $\diag(\tau^{-1},1,\tau)$ (dictated by~\eqref{Z520-fac}),
and looking at the asymptotics as $\tau$ approaches zero in sectors,
we found that
\begin{equation}
\label{MR52}
\re^{\frac{3\pi\ri}{4}}
\diag(\tau^{-1},1,\tau) h(\tau)
= M_R (\tq) \, s_R(\Phi)(\tau) \,,
\end{equation}
where the right hand side depends on the sectors of Borel resummation.
The Borel plane singularities of the component series of the vector
$\Phi(\tau)$ are similarly located at 
\begin{equation}
  \iota_{i,j} = \frac{V(\s_i) - V(\s_j)}{2\pi\ri},\quad i,j=1,2,3,i\neq j\,,
\end{equation}
as well as
\begin{equation}
  2\pi\ri k,\; \iota_{i,j}+2\pi\ri k,\quad k\in \IZ_{\neq 0},
\end{equation}
which form vertical towers as illustrated in Figure~\ref{fig:sings-u0-52}.
In particular, the two singularities $\iota_{1,2},\iota_{2,1}$ are on
the positive and negative real axis.  We pick out the four sectors
which separate the two singularities on the real axis and all the
others, and label them by $I,II,III,IV$, as illustrated in
Figure~\ref{fig:secs-u0-52}.
%

\begin{figure}[htpb!]
\leavevmode
\begin{center}
\includegraphics[height=7cm]{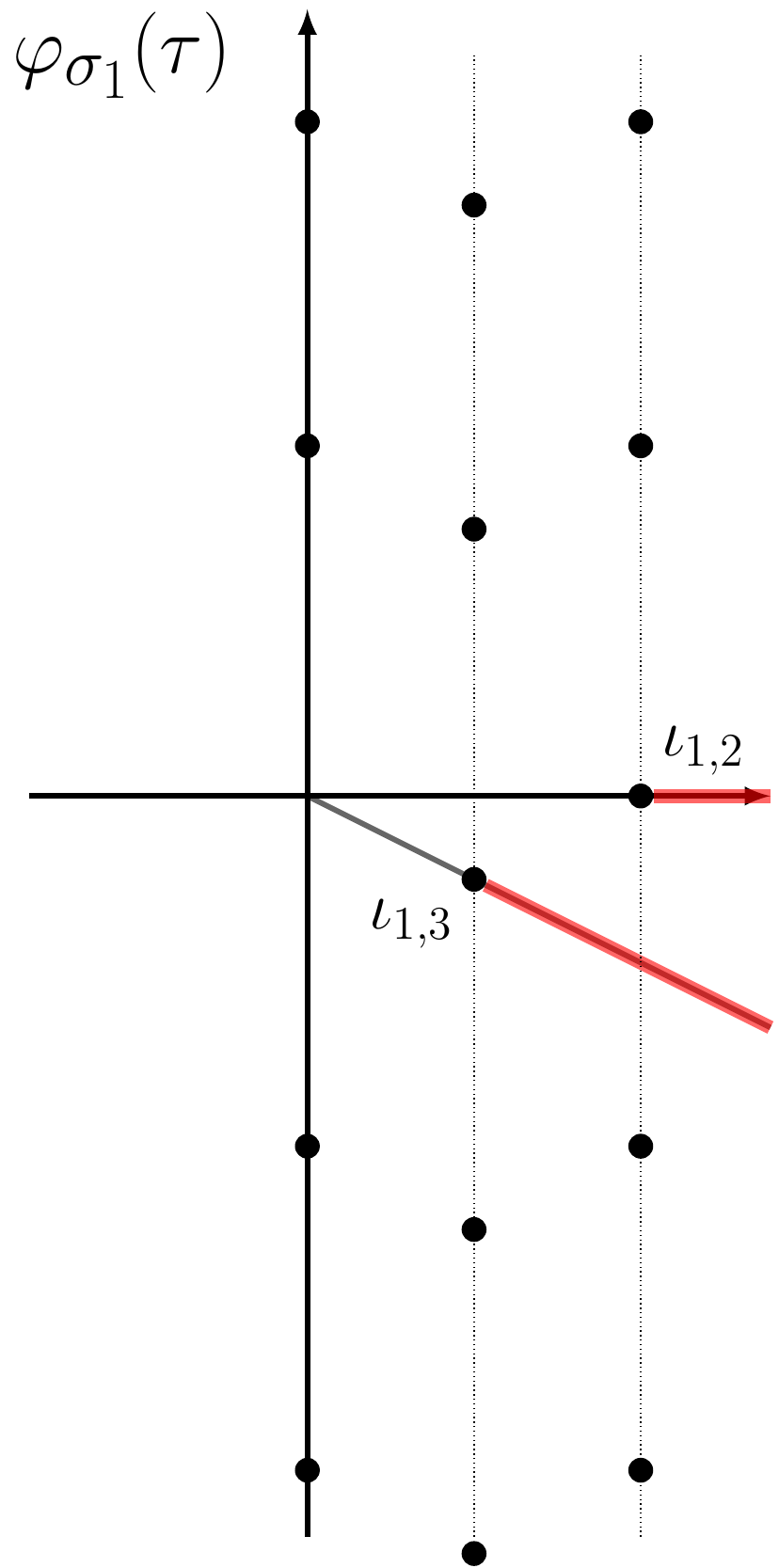}\hspace{3ex}
\includegraphics[height=7cm]{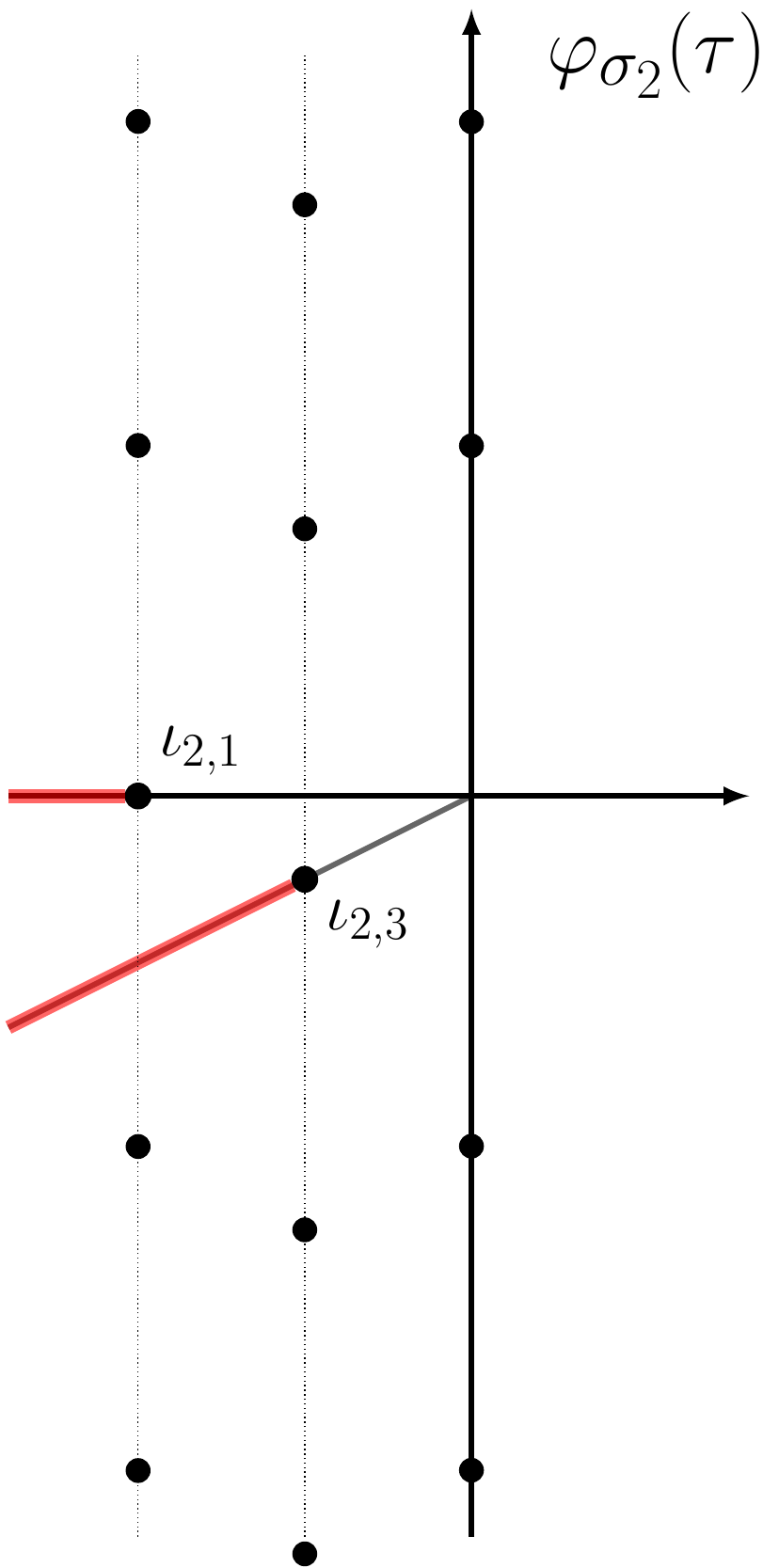}\hspace{3ex}
\includegraphics[height=7cm]{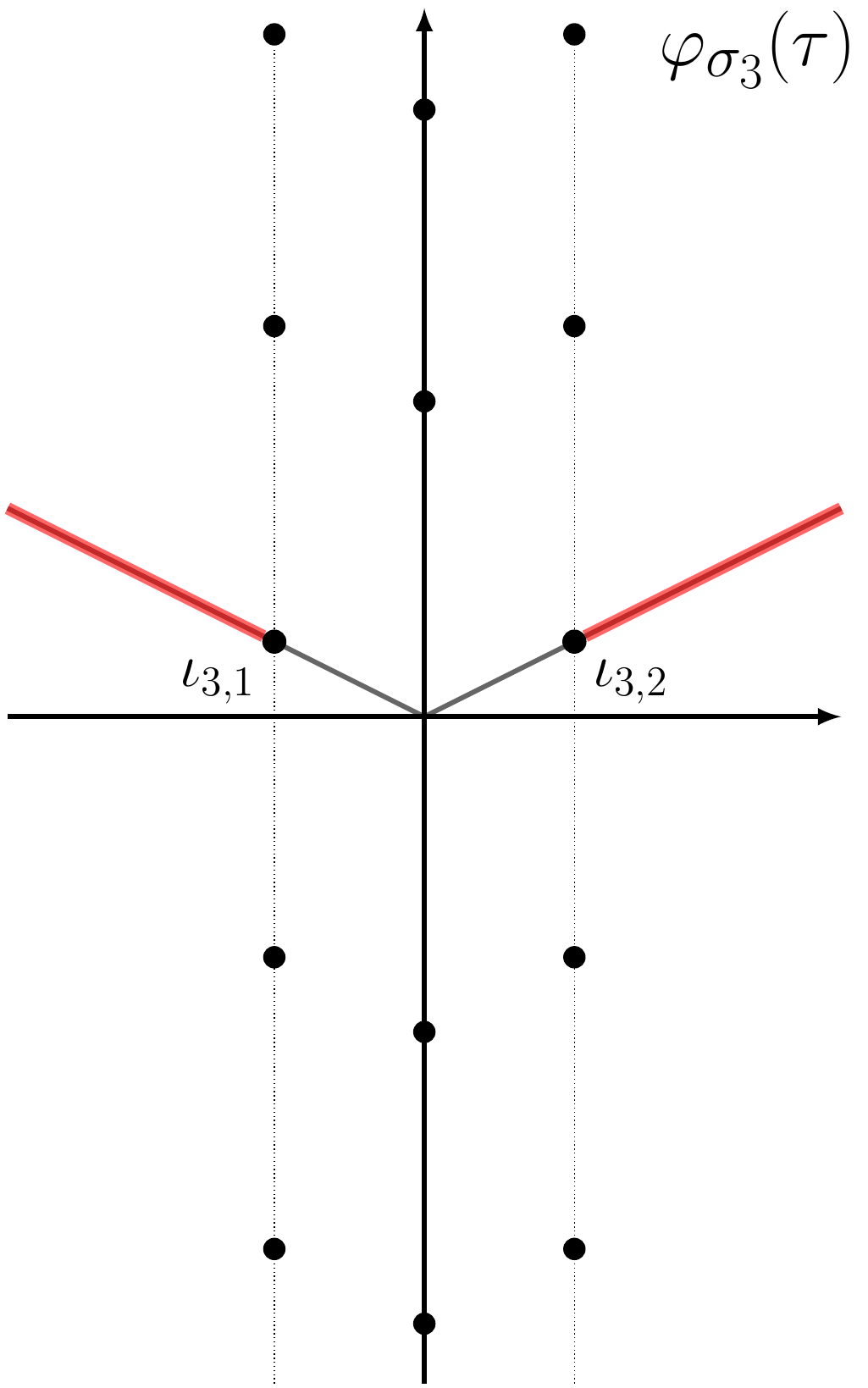}
\end{center}
\caption{The singularities in the Borel plane for the series
  $\varphi_{\sigma_{j}} (0;\tau)$ for $j=1,2,3$ of knot $\knot{5}_2$.}
\label{fig:sings-u0-52}
\end{figure}

\begin{figure}[htpb!]
\leavevmode
\begin{center}
\includegraphics[height=7cm]{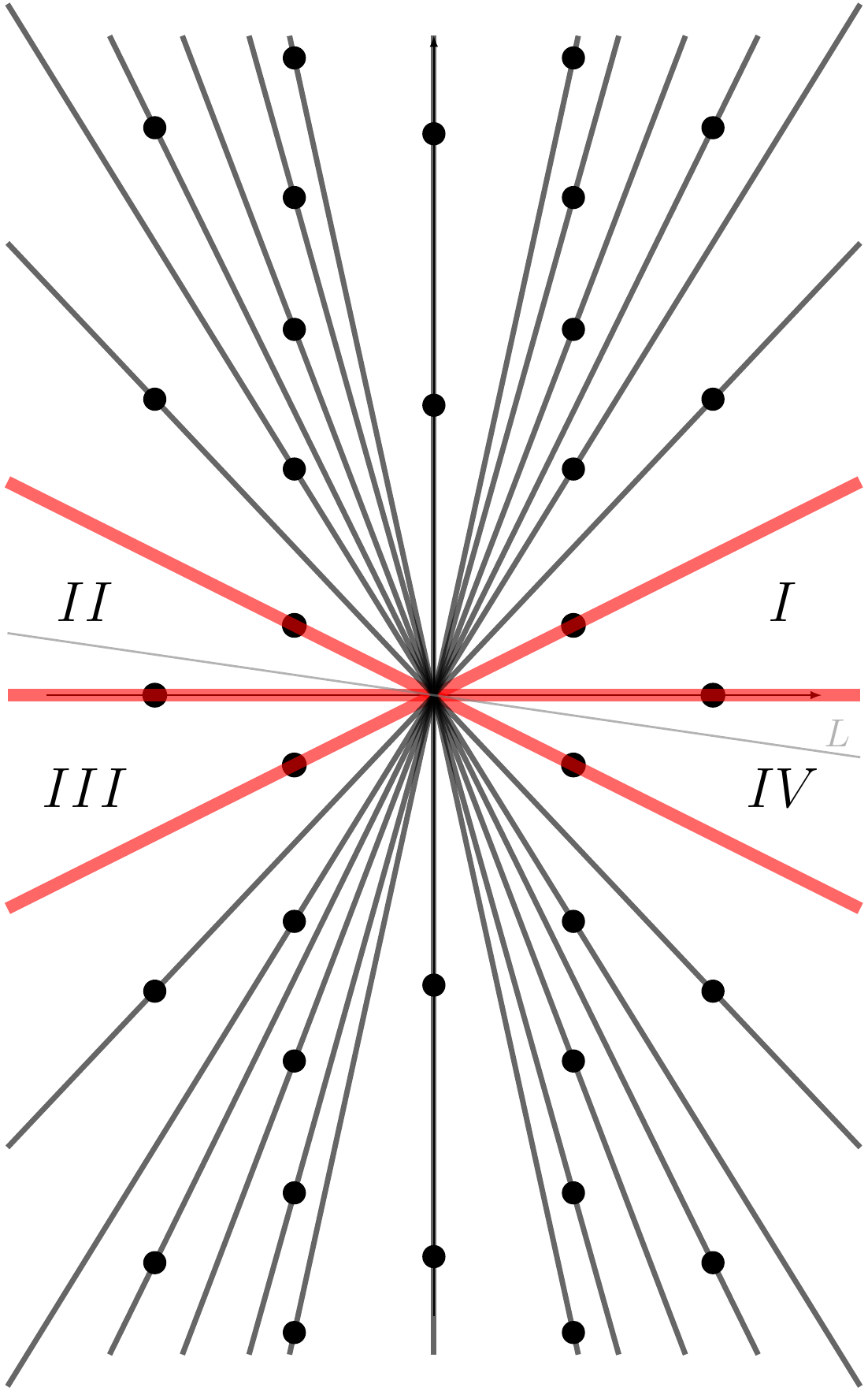}
\end{center}
\caption{Four different sectors in the $\tau$-plane for $\Phi(0;\tau)$
  of knot $\knot{5}_2$ and the line $L$ that divides the $\tau$-plane.}
\label{fig:secs-u0-52}
\end{figure} 


To identify the matrices $M_R$, we consider the third ingredient,
the linear $q$-difference equation 
\begin{equation}
  \label{52qdiffp}
  y_{m}(q) -3y_{m+1}(q)+(3-q^{2+m}) y_{m+2}(q) - y_{m+3}(q)=0 \qquad (m \in \BZ) \,,
\end{equation}
They have fundamental solution sets given by the columns of the
following matrix 
\begin{equation}
  W_m(q) =
  \begin{cases}
    W_m^+(q), & \quad |q|<1\\
    \begin{pmatrix}
      0&0&1\\0&1&0\\1&0&0
    \end{pmatrix}
    W_{-m-2}^-(q^{-1})
    \begin{pmatrix}
      1&0&0\\0&-1&0\\0&0&1
    \end{pmatrix}
    , & \quad |q| >1 \,.
  \end{cases}
\end{equation}
where the matrices $W^\v_m(q)$ with $\v=\pm$ are respectively
\begin{equation}
  \label{Hfund}
  W^{\v}_m(q) =
  \begin{pmatrix}
    H^\v_{0,m}(q) & H^\v_{1,m}(q) & H^\v_{2,m}(q) \\
    H^\v_{0,m+1}(q) & H^\v_{1,m+1}(q) & H^\v_{2,m+1}(q) \\
    H^\v_{0,m+2}(q) & H^\v_{1,m+2}(q) & H^\v_{2,m+2}(q) 
  \end{pmatrix}, \qquad (|q| \neq 1)
\end{equation}
with entries the $q$-series
\begin{subequations}
\begin{align}
\label{Hp0m}
  H^+_{0,m}(q)
  &=\sum_{n=0}^\infty \frac{q^{n(n+1)+nm}}{(q;q)_n^3}\,,
\\
\label{Hp1m}
  H^+_{1,m}(q)
  &=\sum_{n=0}^\infty \frac{q^{n(n+1)+nm}}{(q;q)_n^3}
    \left(1+2n+m-3E_1^{(n)}(q)\right) \,,
  \\
  \label{Hp2m}
  H^+_{2,m}(q)
  &=\sum_{n=0}^\infty \frac{q^{n(n+1)+nm}}{(q;q)_n^3}
    \left((1+2n+m-3E_1^{(n)}(q))^2-3E_2^{(n)}(q)-\frac{1}{6}E_2(q)\right) \,,
\end{align}
\end{subequations}
and
\begin{subequations}
\begin{align}
\label{Hn0m}
  H^-_{0,m}(q)
  &=\sum_{n=0}^\infty (-1)^n\frac{q^{\frac{1}{2}n(n+1)+nm}}{(q;q)_n^3}\,,
\\
\label{Hn1m}
  H^-_{1,m}(q)
  &=\sum_{n=0}^\infty (-1)^n\frac{q^{\frac{1}{2}n(n+1)+nm}}{(q;q)_n^3}
    \left(\frac{1}{2}+n+m-3E_1^{(n)}(q)\right) \,,
  \\
  \label{Hn2m}
  H^-_{2,m}(q)
  &=\sum_{n=0}^\infty (-1)^n\frac{q^{\frac{1}{2}n(n+1)+nm}}{(q;q)_n^3}
  \left(\big(\frac{1}{2}+n+m-3E_1^{(n)}(q)\big)^2-3E_2^{(n)}(q)
    -\frac{1}{12}E_2(q)\right) \,,
\end{align}
\end{subequations}
for $|q|<1$ and extended to $|q|>1$ by the relation
$H^+_{k,m}(q^{-1}) = (-1)^k H^-_{k,-m}(q)$.  Observe that
$H^\v_{k,0}(q)=H^\v_{k}(q)$.
The matrix $W_m(q)$ of holomorphic functions in $|q| \neq 1$ satisfies
several properties summarized in the following theorem.

\begin{theorem}
  \label{thm.520}
  $W_m(q)$ is a fundamental solution of the linear $q$-difference
  equation~\eqref{52qdiffp} that has constant determinant
\begin{equation}
  \label{det52}
  \det(W_m(q))=2 \,,
\end{equation}
satisfies
the orthogonality property 
\begin{equation}
  \label{WWT52b}
  \frac{1}{2} W_{m-1}(q)
  \begin{pmatrix} 0 & 0 & 1 \\ 0&2&0\\ 1 & 0 &0  \end{pmatrix}
  W_{-m-1}(q^{-1})^T =
  \begin{pmatrix} 1 & 0 & 0 \\ 0 & 0 & 1\\0 & 1 & 3-q^{m} \end{pmatrix} \,.
\end{equation}
as well as 
\begin{equation}
  \label{WWT52}
  \frac{1}{2} W_m(q)
  \begin{pmatrix} 0 & 0 & 1 \\ 0&2&0\\ 1 & 0 &0 \end{pmatrix}
  W_{\ell}(q^{-1})^T \in \SL(3,\BZ[q^\pm]) 
\end{equation}
for all integers $m, \ell$ and for $|q| \neq 1$.
\end{theorem}

\begin{conjecture}
  \label{conj.52M0}
  Equation~\eqref{MR52} holds where the matrices $M_R(q)$ are given in
  terms of $W_{-1}(q)$ as follows 
\begin{subequations}
\begin{align}
  \label{M52I}
  M_I(q)
  &= W_{-1}(q)^T \,
    \begin{pmatrix}
      0 & 0 & 1 \\
      -1 & 3 & 0 \\
      0 & -1 & 0
    \end{pmatrix}, & |q|<1 \,,  \\
  \label{M52II}
  M_{II}(q)
  &=\begin{pmatrix}
      1&0&0\\
      0&-1&0\\
      0&0&1
    \end{pmatrix} \,
           W_{-1}(q)^T \,
           \begin{pmatrix}
             0&0&1\\
             3&-1&0\\
             -1&0&0
           \end{pmatrix}, & |q|<1 \,, \\
  \label{M52III}
  M_{III}(q)
  &=
         W_{-1}(q)^T \,
         \begin{pmatrix}
           0&0&1\\
           -1&-1&0\\
           -1&0&0
         \end{pmatrix}, & |q| >1 \,, \\
  \label{M52IV}
  M_{IV}(q)
  &=\begin{pmatrix}
      1&0&0\\
      0&-1&0\\
      0&0&1
    \end{pmatrix} \,
    W_{-1}(q)^T \,
    \begin{pmatrix}
      0&0&1\\
      -1&-1&0\\
      0&-1&0
    \end{pmatrix}, & |q| >1 \,.
\end{align}
\end{subequations}
\end{conjecture}

Assuming the above conjecture, we can now describe completely the
resurgent structure of $\Phi(\tau)$, following the same computation as
in the case of the $\knot{4}_1$ knot. The Stokes matrices are given
by~\eqref{Spm1}--\eqref{Spm4}.
Note that since $M_I(q), M_{II}(q)$ and
$M_{III}(q), M_{IV}(q)$ are given respectively as $q$-series and
$q^{-1}$-series in \eqref{M52I},\eqref{M52II} and
\eqref{M52III},\eqref{M52IV}, analytic continuation as discussed below
\eqref{eq:Eln} is needed when one computes
$\mfS_{IV\rightarrow I},\mfS_{II\rightarrow II}$ in \eqref{Spm4}.
Using~\eqref{det52}--\eqref{WWT52} we
can express the answer in terms of $W_{-1}(q)$. Once again, we find
that the the Stokes matrices $\mfS_{IV\rightarrow I}(q)$ and
$\mfS_{II\rightarrow III}(q)$ are independent of $q$, consistent with
semiclassical asymptotics. The Stokes matrices are given by
\begin{subequations}
\begin{align}
  \label{S52p}
  \ms{S}^+(q)
  &={1\over 2}
    \begin{pmatrix}
      0&1&0\\
      0&1&1\\
      -1&0&0
    \end{pmatrix}W_{-1}(q^{-1})
            \begin{pmatrix}
              0&0&1\\
              0&-2&0\\
              1&0&0
            \end{pmatrix} W_{-1}(q)^T
                   \begin{pmatrix}
                     0&0&-1\\
                     1&1&0\\
                     0&1&0
                   \end{pmatrix},\quad |q|<1\\
  \label{S52m}
  \ms{S}^-(q)
  &= {1\over 2}
    \begin{pmatrix}
      0&3&-1\\
      0&-1&0\\
      1&0&0
    \end{pmatrix} W_{-1}(q)
           \begin{pmatrix}
             0&0&1\\
             0&-2&0\\
             1&0&0
           \end{pmatrix}W_{-1}(q^{-1})^T
                  \begin{pmatrix}
                    0&0&1\\
                    3&-1&0\\
                    -1&0&0
                  \end{pmatrix} \,,\quad |q|<1.
\end{align}
\end{subequations}
These Stokes matrices completely describe the resurgent structure of
$\Phi(\tau)$.  They also satisfy other statements in
Conjectures~\ref{conj.asy} and~\ref{conj.S3D} when $x=1$. The
$q\rightarrow 0$ limit of the Stokes matrices factorizes
\begin{align}
  \ms{S}^+(0) =
  &\mfS_{\s_3,\s_1}\mfS_{\s_3,\s_2}\mfS_{\s_1,\s_2}
    =
     \begin{pmatrix}
       1&0&0\\0&1&0\\-3&0&1
     \end{pmatrix}
     \begin{pmatrix}
       1&0&0\\0&1&0\\0&3&1
     \end{pmatrix}
     \begin{pmatrix}
      1&4&0\\0&1&0\\0&0&1
    \end{pmatrix},  \label{eq:S0pu052}\\
  \ms{S}^-(0) =
  &\mfS_{\s_1,\s_3}\mfS_{\s_2,\s_3}\mfS_{\s_2,\s_1}
    =
     \begin{pmatrix}
       1&0&3\\0&1&0\\0&0&1
     \end{pmatrix}
     \begin{pmatrix}
       1&0&0\\0&1&-3\\0&0&1
     \end{pmatrix}
     \begin{pmatrix}
      1&0&0\\-4&1&0\\0&0&1
    \end{pmatrix}, \label{eq:S0nu052}
\end{align}
where the non-vanishing off-diagonal entry of $\mfS_{\s_i,\s_j}$ is
the Stokes constant associated to the Borel singularity $\iota_{i,j}$.
Assembling these off-diagonal entries in a matrix, we obtain the matrix
\begin{equation}
\label{GZ52}
\begin{pmatrix}
  0 & 4 & 3 \\ -4 & 0 & -3 \\ -3 & 3 & 0
  \end{pmatrix}
\end{equation}
that was found numerically in~\cite[Sec.3.3]{GZ:kashaev}. In addition,
we can assemble the Stokes constants into the matrix $\calS$ of
Equation~\eqref{calS} (after we set $\tx=1$ and replace $\tq$ by $q$).
The resulting matrix $\mc{S}^+(q)$ has entries in $q\BZ[[q]]$, and we
find
\begin{subequations} 
\begin{align}
  \mc{S}^+_{\s_1,\s_1} =
  &\ms{S}^+(q)_{1,1}-1\nn=
  &-12 q+3 q^2+74 q^3+90 q^4+33 q^5+\cO(q^6),\\
  \mc{S}^+_{\s_1,\s_2} =
  &12 q+141 q^2+1520 q^3+17397 q^4+191970 q^5+\cO(q^6),\\
  \mc{S}^+_{\s_1,\s_3} =
  &q+3 q^2+9 q^3+30 q^4+99 q^5+\cO(q^6),\\
  \mc{S}^+_{\s_2,\s_1} =
  &-12 q-141 q^2-1520 q^3-17397 q^4-191970 q^5+\cO(q^6),\\
  \mc{S}^+_{\s_2,\s_2} =
  &12 q+141 q^2+1582 q^3+17583 q^4+194703 q^5+\cO(q^6),\\
  \mc{S}^+_{\s_2,\s_3} =
  &-21 q-235 q^2-2586 q^3-28593 q^4-316104 q^5+\cO(q^6),\\
  \mc{S}^+_{\s_3,\s_1} =
  &-q-3 q^2-9 q^3-30 q^4-99 q^5+\cO(q^6),\\
  \mc{S}^+_{\s_3,\s_2} =
  &21 q+235 q^2+2586 q^3+28593 q^4+316104 q^5+\cO(q^6),\\
  \mc{S}^+_{\s_3,\s_3} =
  &0.
\end{align}
\end{subequations}
The Stokes constants enjoy the symmetry
\begin{equation}
  \mc{S}^{(k)}_{\s_i,\s_j} =
  -\mc{S}^{(k)}_{\s_{\varphi(i)},\s_{\varphi(j)}},\;\; i\neq
  j,\quad \text{for } k\in\IZ_{>0}\,,
\end{equation}
with $\varphi(1)=2,\varphi(2)=1,\varphi(3)=3$.
We notice that the entries of the matrix $\mc{S}^+(q) =
(\mc{S}^+_{\s_i,\s_j}(q))$ (except the upper-left one) are (up to a
sign) in $\IN[[q]]$. Similarly we can extract the Stokes constants
$\mc{S}^{(-k)}_{\s_i,\s_j}$ associated to the singularities in the
lower half plane, and we find
\begin{equation}
  \mc{S}^{(-k)}_{\s_i,\s_j} = - \mc{S}^{(+k)}_{\s_j,\s_i},\;i\neq j,
  \quad\text{and}\quad
  \mc{S}^{(-k)}_{\s_i,\s_i} =
  \mc{S}^{(+k)}_{\s_{\varphi(i)},\s_{\varphi(j)}},\quad
  \text{for } k\in\IZ_{>0}\,.
\end{equation}

The fourth and last ingredient, which makes a full circle of ideas, is
the descendant state-integral of the $\knot{5}_2$ knot 
\begin{equation}
  \label{Z520-desc}
  Z_{\knot{5}_2,m,\mu}(0;\tau) = 
  \int_{\mc{D}} \Phi_\bb(v)^3 \,
  \re^{-2\pi \ri v^2 + 2\pi(m \bb - \mu \bb^{-1})v} \, \rd v \qquad
  (m, \mu \in \BZ) \,.
\end{equation}
Here the same contour $\mc{D}$ as in \eqref{Z410-desc} is used.  It is
a holomorphic function of $\tau \in \BC'$ that coincides with
$Z_{\knot{5}_2}(0;\tau)$ when $m=\mu=0$ and can be expressed
bilinearly in $H^+_{k,m}(q)$ and $H^-_{k,\mu}(\tq)$ as follows
\begin{align}
  \label{520-desc-fac}
    Z_{\knot{5}_2,m,\mu}(0;\tau) 
    =&
    (-1)^{m-\mu+1}\frac{\re^{\frac{\pi\ri}{4}}}{2}
    q^{\frac{m}{2}}\tq^{\frac{\mu}{2}} 
    \left(\frac{q}{\tq}\right)^{\frac{1}{8}} \\ \notag & 
    \left(\tau \, H^-_{0,\mu}(\tq)H^+_{2,m}(q)
      -2H^-_{1,\mu}(\tq)H^+_{1,m}(q) +
      \tau^{-1}\, H^-_{2,\mu}(\tq)H^+_{0,m}(q)\right) \,.
\end{align}
It follows that the matrix-valued function 
\begin{align}
  \label{W520-desc}
  W_{m,\mu}(\tau) = (W_\mu(\tq)^T)^{-1}
  \begin{pmatrix} \tau^{-1} & 0 & 0 \\ 0 & 1 & 0 \\ 0 & 0 & \tau
    \end{pmatrix}  W_m(q)^T 
\end{align}
defined for $\tau=\BC\setminus\BR$, has entries given by the
descendant state-integrals (up to multiplication by a prefactor
of~\eqref{520-desc-fac}) and hence extends to a holomorphic function
of $\tau \in \BC'$ for all integers $m$ and $\mu$. Using this for
$m=-1$ and $\mu=0$ and the orthogonality relation~\eqref{WWT52b}, it
follows that we can express the Borel sums of $\Phi(\tau)$ in a region
$R$ in terms of descendant state-integrals and hence, as holomorphic
functions of $\tau \in \BC'$ as follows. For instance, in the region
$I$ we have
\begin{equation}
  s_I(\Phi)(\tau) = \re^{\frac{3\pi\ri}{4}}M_I(\tq)^{-1}
  \begin{pmatrix}
    \tau^{-1}h_0(\tau) \\
    h_1(\tau) \\
    \tau h_2(\tau)
  \end{pmatrix}
  = \ri\left(\frac{q}{\tq}\right)^{-\frac{1}{8}}
  \begin{pmatrix}
    Z_{\knot{5}_2,0,0}(0;\tau)-\tq^{1/2}Z_{\knot{5}_2,0,-1}(0;\tau) \\
    Z_{\knot{5}_2,0,0}(0;\tau)\\
    \tq^{-1/2}Z_{\knot{5}_2,0,1}(0;\tau)
\end{pmatrix}
\end{equation}
  
This completes the discussion of $u=0$ for the $\knot{5}_2$ knot.


\section{Holonomic and elliptic functions inside and outside
  the unit disk}
\label{sec.elliptic}

An important property for the functions of a complex variable $q$ in
our paper (such as the holomorphic blocks considered below) is that
they can be defined both inside ($|q|<1$) and outside ($|q|>1$) the unit
disk, in such a way that they have the same annihilator ideal. Recall
that if $L$ and $M$ denote the operators that act on functions
$f(x;q)$ by $(L f)(x;q)=f(qx;q)$ and $(M f)(x;q)=x f(x;q)$, then
$LM=q ML$, hence $L^{-1}M=q^{-1}ML^{-1}$. It follows that if
$P(L,M,q)f(x;q)=0$, then $P(L^{-1},M,q^{-1}) f(x;q^{-1})=0$ where
$P(L,M,q)$ denotes a polynomial in $L$ with coefficients polynomials
in $M$ and $q$.

A first example of a function to consider is
$(x;q)_\infty=\prod_{j=0}^\infty (1-q^j x)$, which is well-defined for
$|q|<1$ and $x \in \BC$ and satisfies the linear $q$-difference
equation
\begin{equation}
\label{recxq}
(1-x)(qx;q)_\infty = (x;q)_\infty \qquad (|q|<1) \,.
\end{equation}
We can extend it to a meromorphic function of $x$ when $|q|>1$ (by a
slight abuse of notation) by defining
\begin{equation}
\label{xqout}
(x;q^{-1})_\infty := (qx;q)_\infty^{-1} \qquad (|q|<1)\,,
\end{equation}
so that Equation~\eqref{recxq} holds for $|q| \neq 1$. Our second
example is the theta function
\begin{equation}
\th(x;q)=(-q^{\frac{1}{2}}x;q)_\infty (-q^{\frac{1}{2}}x^{-1};q)_\infty \qquad (|q|<1)  
\end{equation}
which satisfies the linear $q$-difference equation
\begin{equation}
\label{recth}
\th(qx;q)=q^{-\frac{1}{2}} x^{-1}\th(x;q) \qquad (|q|<1) \,
\end{equation}
and can be extended to $\th(x;q^{-1})=\th(x;q)^{-1}$ when $|q|> 1$ so
that Equation~\eqref{recth} holds for $|q| \neq 1$. $\th(x;q)$ is a
meromorphic function of $x \in \BC^*$ with the following (simple)
zeros and (simple) poles

\begin{align}
  \label{thetap}
  |q| < 1 & & \text{zeros}(\th)& =-q^{\frac{1}{2}+\BZ}
  &  \text{poles}(\th) & =\emptyset \\
  \notag
  |q| > 1 & & \text{zeros}(\th)& = \emptyset
  &  \text{poles}(\th) & =-q^{\frac{1}{2}+\BZ} \,.
\end{align}

An important property of the theta functions is that they factorize
the exponentials of a quadratic and linear form of $u$. This
fact is a consequence of the modular invariance of the theta function
and was used extensively in the study of holomorphic
blocks~\cite{Beem}.

\begin{lemma}
  \label{lem.thetaf}
  For integers $r$ and $s$ we have:
  \begin{subequations}
    \begin{align}
      \label{ebx2}
      \re^{\pi \ri (u + r c_\bb)^2}
      & = \re^{-\frac{\pi \ri}{12}(\tau+\tau^{-1})} \th((-q^{\frac{1}{2}})^r x;q)
        \th((-\tq^{-\frac{1}{2}})^r\tx;\tq^{-1}) \\
      \label{ebx3}
      \re^{\pi \ri r u^2 + 2 \pi \ri s c_\bb u}
      &= \ri^s
        \re^{\frac{\pi \ri}{12}(3s-r)(\tau+\tau^{-1})}
        \th(x;q)^{r-s} \th(-q^{\frac{1}{2}} x;q)^s
        \times \th(\tx;\tq^{-1})^{r-s} \th(-\tq^{-\frac{1}{2}}\tx;\tq^{-1})^s
    \end{align}
  \end{subequations}
  for integers $r$ and $s$.
\end{lemma}

Note that we the above factorization formulas are by no means unique,
and this is a reflection of the dependence of the above formulas on a
theta divisor.

\begin{proof}
  When $x=\re^{2\pi \bb u}$, $q=e(\tau)$, $\tq=e(-1/\tau)$ and
  $|q|<1$, then we claim
  \begin{equation}
    \label{ebx}
    \re^{-\frac{1}{4\pi \ri \tau} (\log x)^2} = \re^{\pi \ri u^2}
    = \Phi_\bb(0)^{-2} \Phi_\bb(u) \Phi_\bb(-u)
    = \re^{-\frac{\pi \ri}{12}(\tau+\tau^{-1})} \th(x;q)\th(\tx;\tq^{-1}) \,.
  \end{equation}
  The first equality is easy, the second one follows from the
  inversion formula of Faddeev's quantum dilogarithm, and the third one
  follows from the product expansion of Faddeev's quantum dilogarithm
  or from the modular invariance of the theta function. Note also that
  $\Phi_\bb(0)^2=(q/\tq)^{\frac{1}{24}}=\re^{\frac{\pi
      \ri}{12}(\tau+\tau^{-1})}$.  Equation~\eqref{ebx2} follows
  easily from the above and Equation~\eqref{ebx3} follows from the
  above using for example,
  $ \re^{\pi \ri r u^2 + 2 \pi \ri s c_\bb u} = \re^{(r-s)\pi \ri
    u^2}\re^{s\pi i(u+c_\bb)^2}\re^{-s\pi \ri c_\bb^2}$.
\end{proof}


\section{The $\knot{4}_1$ knot}
\label{sec.41}

\subsection{Asymptotic series}
\label{sub.41asy}

Our starting point will be the state-integral for the $\knot{4}_1$
knot~\cite[Eqn.38]{AK} (after removing a prefactor that depends on $u$
alone)
\begin{equation}
  \label{Z41x}
  Z_{\knot{4}_1}(u;\tau) =
  \re^{-2 \pi \ri u^2} \int_{\BR+\ri 0}  \Phi_\bb(v) \, \Phi_\bb(u+v)
  \, \re^{-\pi \ri (v^2+4uv)} \rd v \,.
\end{equation}
The above state-integral (and all the subsequent ones) is a
holomorphic function of $\tau \in \BC'$ and $u$ when
$|\Im(u)| < |\bb+\bb^{-1}|/2$ and extends to an entire function of
$u$ (see Theorem~\ref{thm.41a} below).

After a change of variables $u \mapsto u/(2\pi \bb)$ (see
Equation~\eqref{ub}) and $v \mapsto v/(2\pi\bb)$ 
the asymptotic expansion of the quantum dilogarithm (see for
instance~\cite[Prop.6]{AK}) implies that the
integrand of $Z_{\knot{4}_1}(\ub;\tau)$ has a leading term given by
$\re^{V(u,v)/(2 \pi \ri \tau)}$ where
\begin{equation}
  \label{V41}
  V(u,v) = \Li_2(-\re^{v}) + \Li_2(-\re^{u+v})
  + \frac{1}{2} (v)^2 + 2 u v \,.
\end{equation}
Taking derivative with respect to $v$ gives the equation for the
critical point
\begin{equation}
  \label{41crit}
  2 u + v - \log(1+\re^{v}) -\log(1+\re^{u+v})=0
\end{equation}
which implies that $(x,y)=(e^u,-e^v)$ is a complex point 
points of the affine curve $S$ given by
\begin{equation}
  \label{41critb}
  S : -x^2 y = (1-y)(1- x y) 
\end{equation}
and $(u, v)$ is a point of the exponentiated curve $S^*$ given by
the above equation where $(x,y)=(\re^{u},-\re^{v})$.  Moreover, we
have
\begin{equation}
  \label{V41b}
  V(u,v) = \Li_2(y)+\Li_2(xy) +\frac{1}{2} (\log (-y))^2 + 2 \log x \log(-y) \,.
\end{equation}
Note that~\eqref{41crit} has more information than~\eqref{41critb}
since it chooses the logarithms of $1+\re^{v}$ and $1+\re^{u+v}$
such that~\eqref{41crit} holds. This ultimately implies that $V$ is a
holomorphic $\BC/2\pi^2\BZ$-valued function on the exponentiated curve
$S^*$. Note that when $u=0$, Equation~\eqref{41critb}
becomes~\eqref{41xi}.

The constant term of the asymptotic expansion is given by the Hessian
of $V(u,v)$ at a critical point $(u,v)$, and it is a rational
function of $x$ and $y$ is given by
\begin{equation}
  \label{41delta}
  \delta(x,y) = -\frac{1-xy^2}{x^2y} \,.
\end{equation}
Note that $\delta(x,y)=0$ on $S$ if and only if $x$ is a root of the
discriminant of $S$ with respect to $y$, i.e.,
\begin{equation}
  (1 - 3 x + x^2) (1 + x + x^2)=0.
  \label{disc41y}
\end{equation}
In other words, $\delta$ vanishes precisely when two branches of
$y=y(x)$ coincide.

Beyond the leading asymptotic expansion and its constant term, the
asymptotic series has the form $\Phi(x,y;\tau)$ where 
\begin{equation}
  \label{Phi41}
  \Phi(x,y;\tau)=\exp\left( { V(u,v) \over 2\pi\ri \tau} \right)
  \varphi(x,y;\tau),
  \qquad 
  \varphi(x,y;\tau) \in \frac{1}{\sqrt{\ri\delta}}
  \BQ[x^\pm,y^\pm,\delta^{-1}] [[2 \pi \ri \tau]]
\end{equation}
where $\delta$ is given in~\eqref{41delta} and
$\sqrt{\ri \delta} \, \varphi(x,y;0)=1$.  In other words, the
coefficient of every power of $2\pi \ri \tau$ in
$\sqrt{\delta} \, \varphi(x,y;\tau)$ is a rational function on
$S$. There is a natural projection $S \to \BC^*$ given by
$(x,y) \rightarrow x$ and we denote by $y_\s(x)$ the choice of a local
section (an algebraic function of $x$), for
$\s \in \calP=\{\s_1,\s_2\}$.  We denote the corresponding series
$\Phi(x,y_\s(x);\tau)$ simply by $\Phi_\s(x;\tau)$.
Note that
\begin{equation}
  \delta(x,y_{1,2}(x)) = \pm\frac{\sqrt{(1-x-x^{-1})^2-4}}{x}
  \label{eq:delta-41}
\end{equation}
and that the two series are related by
\begin{equation}
\label{Phi4112}
\Phi_2(x;\tau) = \ri \Phi_1(x;-\tau) \,.
\end{equation}
The power series $\sqrt{\ri\delta}\varphi_\s(x;\tau)$ can be computed
by applying Gaussian expansion to the state-integral \eqref{Z41x}. One
can compute up to 20 terms in a few minutes, and
the first few terms agree with an independent computation using the
WKB method (see \cite[Eqn.(4.39)]{DGLZ} as well as
\cite{GG:asymptotics}), and given by $\sqrt{\ri\delta}\varphi_\s(x;\tau)$
\begin{align}
  &\sqrt{\ri\delta}\varphi_{1,2}\left(x;\frac{\tau}{2\pi\ri}\right)\nn =
  &1 - \frac{1}{24\gamma_{1,2}^3(x)}\big(x^{-3} - x^{-2} - 2 x^{-1} + 15 -
    2 x - x^2 + x^3\big) \tau\nn+
  &\frac{1}{1152\gamma_{1,2}^6(x)}
    \big(x^{-6} - 2 x^{-5} - 3 x^{-4} + 610 x^{-3} - 
    606 x^{-2} - 1210 x^{-1} + 3117\nn
  &\phantom{=====}- 1210 x - 606 x^2 + 610 x^3 - 
    3 x^4 - 2 x^5 + x^6\big) \tau^2 +\cO(\tau^3),
    \label{eq:vf-41}
\end{align}
where
\begin{equation}
  \gamma_{1,2}(x) = x\delta(x,y_{1,2}(x)) = \pm\sqrt{x^{-2}-2x^{-1}-1-2x+x^2}.
\end{equation}
On the other hand, if one sets $x$ to numerical values, one can
compute 300 terms of this power series.
\subsection{Holomorphic blocks}
\label{sub.41holo}


In this section we give the definition of the holomorphic blocks (and
their descendants) which factorize the state-integral (and its
descendants), and discuss their analytic properties, and their linear
$q$-difference equations.  Note that in this section, as well as in
Section~\ref{sub.41near1}, all the statements are theorems, whose
proofs we provide.

Motivated by the state-integral $Z_{\knot{4}_1}(u;\tau)$ of the
$\knot{4}_1$ knot given in~\eqref{Z41x}, and by the descendant
state-integral $Z_{\knot{4}_1,m,\mu}(0;\tau)$ given
in~\eqref{Z410-desc}, we introduce the descendant state-integral of
the $\knot{4}_1$ knot 
\begin{equation}
  \label{Z41x-desc}
  Z_{\knot{4}_1,m,\mu}(u;\tau) =
  \re^{-2 \pi \ri u^2} 
  \int_{\mc{D}} \Phi_\bb(v) \, \Phi_\bb(u+v)
  \, \re^{-\pi \ri (v^2+4uv) + 2\pi(m \bb - \mu \bb^{-1})v} \, \rd v
\end{equation}
for integers $m$ and $\mu$, which agrees with the Andersen-Kashaev
invariant of the $\knot{4}_1$ knot when $m=\mu=0$. Here the contour
$\mc{D}$ was introduced in \eqref{Z410-desc}. It is expressed in terms
of two descendant holomorphic blocks, which we denote by $A_m$ and
$B_m$ instead of $B^{\s_1}_m$ and $B^{\s_2}_m$, in order to simplify
the notation.  For $|q| \neq 1$, $A_m(x;q)$ and $B_m(x;q)$ are given
by
\begin{subequations}
  \begin{align}
    \label{41Am}
    A_m(x;q) &= \th(-q^{\frac{1}{2}}x;q)^{-2}
               x^{2m} J(q^m x^2, x; q)\,, \\
    \label{41Bm}
    B_m(x;q) &= \th(-q^{-\frac{1}{2}}x;q)
               x^m J(q^m x, x^{-1}; q)\,,
  \end{align}
\end{subequations}
where $J(x,y;q^\ve):=J^\ve(x,y;q)$ for $|q|<1$ and $\ve=\pm$ is the
$q$-Hahn Bessel function
\begin{subequations}
\begin{align}
  \label{J41p}
  J^+(x,y;q) & = (qy;q)_\infty \sum_{n=0}^\infty (-1)^n
               \frac{q^{\frac{1}{2}n(n+1)} x^n}{(q;q)_n (qy;q)_n}\,, \\
  \label{J41n}
  J^-(x,y;q) & = \frac{1}{(y;q)_\infty}
               \sum_{n=0}^\infty (-1)^n \frac{q^{\frac{1}{2}n(n+1)}
               x^n y^{-n}}{ (q;q)_n (qy^{-1};q)_n}
               \,.
\end{align}
\end{subequations}

The next theorem expresses the descendant state-integrals bilinearly
in terms of descendant holomorphic blocks.

\begin{theorem}
  \label{thm.41a}
  \rm{(a)} The descendant state-integral can be expressed in terms of
  the descendant holomorphic blocks by 
  \begin{align}
    Z_{\knot{4}_1,m,\mu}(\ub;\tau) 
    =& (-1)^{m+\mu}q^{m/2}\tq^{\mu/2} \Big(
    \re^{-\frac{3\pi\ri}{4} -\frac{\pi\ri}{6}(\tau+\tau^{-1})}
      A_m(x;q) A_{-\mu}(\tx;\tq^{-1}) \\ & +
    \re^{\frac{3\pi\ri}{4}+\frac{\pi\ri}{3}(\tau+\tau^{-1})}
      B_{m}(x;q) B_{-\mu}(\tx;\tq^{-1})\Big) \,.
    \label{41x-desc-fac}
  \end{align}
  \rm{(b)} The functions $A_m(x;q)$, $B_m(x;q)$ are holomorphic
  functions of $|q| \neq 1$ and meromorphic functions of $x \in \BC^*$
  with poles in $x \in q^\BZ$ of order at most $1$.  \newline
  \rm{(c)} Let
  \begin{equation}
    W_m(x;q) = \begin{pmatrix}
      A_{m}(x;q) & B_{m}(x;q)  \\
      A_{m+1}(x;q) & B_{m+1}(x;q) 
    \end{pmatrix} \qquad (|q| \neq 1) \,.
    \label{eq:Wm}
  \end{equation}
  For all integers $m$ and $\mu$, the state-integral
  $Z_{\knot{4}_1,m,\mu}(u;\tau)$ and the matrix-valued function
  \begin{equation}
    W_{m,\mu}(u;\tau) = W_{-\mu}(\tx;\tq^{-1}) \Delta(\tau) 
    W_m(x;q)^T \,,
    \label{W41x-desc}
  \end{equation}
  where
  \begin{equation}
    \label{D41}
    \Delta(\tau) = \begin{pmatrix}
      \re^{-\frac{3\pi\ri}{4}-\frac{\pi\ri}{6}(\tau+\tau^{-1})}&0\\
      0&\re^{\frac{3\pi\ri}{4}+\frac{\pi\ri}{3}(\tau+\tau^{-1})}
    \end{pmatrix} \,,
  \end{equation}
  are holomorphic functions of $\tau \in \BC'$ and entire
  functions of $u \in \BC$. 
\end{theorem}

\begin{proof}
  Part (a) follows by applying the residue theorem to the
  state-integral~\eqref{Z41x-desc}, along the lines of the proof of
  Theorem 1.1 in~\cite{GK:qseries}. A similar result was stated in
  \cite{dimofte-state}.

  Part (b) follows from the fact that when $|q|<1$, the ratio test
  implies that $J^+(x,y;q)$ is an entire function function of
  $(x,y) \in \BC^2$ and $J^-(x,y;q)$ is a meromorphic function of
  $(x,y)$ with poles in $y \in q^{\BZ}$.

  For part (c), one uses parts (a) and (b) to deduce that
  $W_{m,\mu}(u;\tau)$ is holomorphic of $\tau\in \IC'$ and meromorphic in
  $u$ with potential simple poles at $\ri \bb \BZ + \ri \bb^{-1} \BZ$.
  An expansion at these points, done by the method of
  Section~\ref{sub.41near1}, demonstrates that the function is analytic
  at the points $\ri \bb \BZ + \ri \bb^{-1} \BZ$.
\end{proof}

Note that the summand of $J^+$ (a proper $q$-hypergeometric function)
equals to that of $J^-$ after replacing $q$ by $q^{-1}$. This implies
that $J^\pm$ have a common annihilating ideal $\calI_J$ with respect
to $x,y$ which can be computed (rigorously, along with a provided
certificate) using the creative telescoping method of
Zeilberger~\cite{PWZ} implemented in the \texttt{HolonomicFunctions}
package of Koutschan~\cite{Koutschan,Koutschan:holofunctions}.  Below,
we will abbreviate this package by~\texttt{HF}.

\begin{lemma}  
  \label{lem.annJ}
  The annihilating ideal of $\calI_J$ of $J^\pm$ is given by
  \begin{equation}
    \label{annJ}
    \calI_J = \langle
    (-x + y) + x S_y - y S_x, \,\, 1 + (-1 - x + q y) S_y + x S_y^2
    \rangle
  \end{equation}
  where $S_x$ and $S_y$ are the shifts $x$ to $qx$ and $y$ to $qy$.
\end{lemma}

The next theorem concerns the properties of the linear $q$-difference
equations satisfied by the descendant holomorphic blocks.

\begin{theorem}
  \label{thm.41b}
  \rm{(a)} The pair $A_m(x;q)$ and $B_m(x;q)$ are $q$-holonomic
  functions in the variables $(m,x)$ with a common annihilating ideal
  \begin{equation}
    \label{ann41}
    \calI_{\knot{4}_1} = \langle q^m x^2 + (-q^m + q^{1 + 2 m} x^2) S_m + x^3
    S_x,  \, 
    (1-x^{-1}S_m)(1-x^{-2}S_m) + q^{1+m} S_m \rangle
  \end{equation}
  where $S_m$ is the shift of $m$ to $m+1$ and $S_x$ is the shift of
  $x$ to $q x$.  $\calI_{\knot{4}_1}$ has rank 2 and the two functions form
  a basis of solutions of the corresponding system of linear equations.
  \newline
  \rm{(b)} As functions of an integer $m$, $A_m(x;q)$ and
  $B_m(x;q)$ form a basis of solutions of the linear $q$-difference
  equation $\Bhat_{\knot{4}_1}(S_m,x,q^m,q) f_m(x;q) = 0$ for $|q| \neq 1$
  where 
  \begin{equation}
    \label{rec41m}
    \Bhat_{\knot{4}_1}(S_m,x,q^m,q) =  (1-x^{-1}S_m)(1-x^{-2}S_m) + q^{1+m} S_m \,.
  \end{equation}
  \rm{(c)} The Wronskian $W_m(x;q)$ of~\eqref{rec41m},
  defined in ~\eqref{eq:Wm}, satisfies
  \begin{equation}
    \label{W410}
    \det(W_m(x;q)) = x^{3m+3} 
    \qquad
    (m \in \BZ) \,.
  \end{equation}
  \rm{(d)} The Wronskian satisfies the orthogonality relation
  \begin{equation}
    \label{WW41b}
    W_{-1}(x;q) \,
    \begin{pmatrix}
      1 & 0 \\
      0 & -1
    \end{pmatrix}
    \,  W_{-1}(x;q^{-1})^T =
    \begin{pmatrix}
      x^{-2}+x^{-1}-1 & 1\\
      1 & 0
    \end{pmatrix} \,.
  \end{equation}
  It follows that for all integers $m$ and $\ell$ 
  \begin{equation}
    \label{WW41}
    W_m(x;q) \,
    \begin{pmatrix}
      1 & 0 \\
      0 & -1
    \end{pmatrix} \, W_\ell(x;q^{-1})^T \in \GL(2,\BZ[q^\pm, x^\pm])
  \end{equation}
  \rm{(e)} As functions of $x$, $A_m(x;q)$ and $B_m(x;q)$ form a basis
  of a linear $q$-difference equation
  $\hat A_{\knot{4}_1} (S_x,x,q^m,q) f_m(x;q) =0$
where 
  \begin{equation}
    \label{rec41x}
    \hat A_{\knot{4}_1} (S_x,x,q^m,q) = \sum_{j=0}^2 C_j(x,q^m,q) S_x^j \,,
  \end{equation}
  $S_x$ is the operator that shifts $x$ to $qx$ and
  \begin{subequations}
    \begin{align}
      \label{rec41x0}
      C_0 =
      &q^{2 + 3 m} x^2 (-1 + q^{3 + m} x^2) \\
      \label{rec41x1}
      C_1 =
      & -q^m (-1 + q^{2 + m} x^2) (1 - q x - q^{1 + m} x^2 - q^{3 + m} x^2 - 
        q^{3 + m} x^3 + q^{4 + 2 m} x^4) \\ 
      \label{rec41x2}
      C_2 =
      & q^2 x^2 (-1 + q^{1 + m} x^2) \,.
    \end{align}
  \end{subequations}
  \rm{(f)} The Wronskian of~\eqref{rec41x}
  \begin{equation}
    \label{W41x}
    \calW_m(x;q) = \begin{pmatrix}
      A_{m}(x;q) & B_{m}(x;q)  \\
      A_{m}(qx;q) & B_{m}(qx;q) 
    \end{pmatrix}, \qquad (|q| \neq 1)
  \end{equation}
  satisfies
  \begin{equation}
    \label{detW41x}
    \det(\calW_m(x;q)) = q^m x^{3m}(1-q^{m+1} x^2) \qquad (m \in \BZ) \,.
  \end{equation}
\end{theorem}

\begin{proof}
  Since $A_m(x;q)$ and $B_m(x;q)$ are given in terms of $q$-proper
  hypergeometric multisums, it follows from the fundamental theorem of
  Zeilberger~\cite{Zeil:holo,WZ,PWZ} (see also~\cite{GL:survey}) that
  they are $q$-holonomic functions in both variables $m$ and $x$.
  Part (a) follows from an application of the \texttt{HF} package of
  Koutschan~\cite{Koutschan, Koutschan:holofunctions}.

  Part (b) follows from the \texttt{HF} package. The
  fact that they are a basis follows from (c).

  For part (c), Equation~\eqref{rec41m} implies that the determinant of
  the Wronskian satisfies the first order
  equation $\det(W_{m+1}(x;q))=x^3\det(W_{m}(x;q))$
  (see~\cite[Lem.4.7]{GK:74}). It follows that
  $\det(W_m(x;q))=x^{3m}\det(W_0(x;q))$ with initial condition a
  function of $x$ given by Swarttouw~\cite{MR2714507} 
  \begin{equation}
    \label{W4100}
    \det(W_0(x;q)) = x^3 
    \qquad
    (|q| \neq 1) \,.
  \end{equation}
  We recall the details of the proof which will be useful in the case
  of the $\knot{5}_2$ knot.  When $|q|<1$, the $q$-Hahn Bessel function
  $J(x,y;q)$ satisfies the recursion relation
  \begin{equation}
    y J(qx,y;q) - (1+y-x)J(x,y;q) + J(q^{-1}x,y;q) = 0.
    \label{eq:J-rec1}
  \end{equation}
  This follows from~\cite{MR2714507}, and can also be proved
  using the \texttt{HF} package. It then follows that
  \begin{equation}
    \cJ_{\nu,1}(z;q) := J(z,q^\nu;q), \quad
    \cJ_{\nu,2}(z;q) := z^{-\nu}J(q^{-\nu}z,q^{-\nu};q)
  \end{equation}
  are two independent solutions to
  \begin{equation}
    q^\nu \cJ(qz;q) - (1+q^\nu-x)\cJ(z;q) + \cJ(q^{-1}z;q) = 0.
  \end{equation}
  The corresponding Wronskian
  \begin{equation}
    \cW_\nu(z;q) =
    \begin{pmatrix}
      \cJ_{\nu,1}(z;q) & \cJ_{\nu,2}(z;q)\\
      \cJ_{\nu,1}(qz;q) & \cJ_{\nu,2}(qz;q)
    \end{pmatrix}
  \end{equation}
  satisfies the recursion relation (see~\cite[Lem.4.7]{GK:74})
  \begin{equation}
    \det\cW_\nu(z;q) = q^{-\nu} \det\cW_\nu(q^{-1}z;q)
  \end{equation}
  which implies that the determinant of $U(z;q) = z^\nu \cW_\nu(z;q)$ is
  an elliptic function
  \begin{equation}
    \det U(qz;q) = \det U(z;q).
  \end{equation}
  It can be computed by the following limit
  \begin{align}
    \det U(z;q) =
    &\lim_{k\rightarrow \infty} \det U(q^k z;q) = \lim_{z\rightarrow 0}
      \det U(z;q)\nn  =
    &\lim_{z\rightarrow 0} \Big(
      q^{-\nu}J(z,q^\nu;q)J(q^{1-\nu}z,q^{-\nu};q) -
      J(q^{-\nu}z,q^{-\nu};q)J(qz,q^\nu;q)
      \Big) \nn =
    &(q^{-\nu}-1)(q^{1+\nu};q)_\infty(q^{1-\nu};q)_\infty,
  \end{align}
  where in the last step we just used the $q$-expansion definition of
  the $q$-Hahn Bessel function.  We thus have
  \begin{equation}
    \label{zcw}
    z^\nu \det \cW_\nu(z;q) = -(q q^\nu;q)_\infty(q^{-\nu};q)_\infty.
  \end{equation}
  Using the substitution
  \begin{equation}
    z\mapsto x^2,\quad z^\nu \mapsto x
  \end{equation}
  in the above equation and cancelling with the
  $\theta$-prefactors of $A_m(x;q)$ and $B_m(x;q)$ we
  obtain Equation~\eqref{W4100} for $|q|<1$.
  The case of $|q|>1$ can be obtained by analytic continuation on both
  sides of \eqref{W4100}.

  For part (d), Equation~\eqref{rec41m} implies that
  \begin{equation}
    W_{m+1}(x;q) =
    \begin{pmatrix}
      0&1\\-x^3&x^2+x-q^{1+m}x^3
    \end{pmatrix} W_m(x;q) \,.
    \label{eq:W41rec}
  \end{equation}
  Hence, Equation~\eqref{WW41} follows from~\eqref{WW41b}. The
  latter is a direct consequence of the analytic continuation formula
  \begin{equation}
    J(x,y;q) = \theta(-q^{1/2}y;q)J(y^{-1}x,y^{-1};q^{-1})
  \end{equation}
  which one easily sees by comparing \eqref{J41p} and \eqref{J41n}.

  Part (e) follows from the \texttt{HF} package. The
  fact that they are a basis follows from (f).

  For part (f), since $A_m(x;q)$ (as well as $B_m(x;q)$) are
  annihilated by the first generator of~\eqref{ann41}, it follows that
  \begin{equation}
    q^m x^2 A_m(x;q) -q^m (1 - q^{1 + m} x^2) A_{m+1}(x;q)
    + x^3 A_m(qx;q) =0 \,.
  \end{equation}
  After solving for the above for $A_m(qx;q)$ (and same for
  $B_m(qx;q)$) and substituting into the Wronskian~\eqref{W41x}, it
  follows that the two Wronskians are related by
  \begin{equation}
    \label{W41xm}
    \calW_m(x;q) =
    \begin{pmatrix} 1 & 0 \\ -q^m x^{-1} & q^m x^{-3}(1-q^{1+m} x^2)
    \end{pmatrix}
    W_m(x;q) \,.
  \end{equation}
  After taking determinants, it follows that
  \begin{equation}
    \det(\calW_m(x;q)) = q^m x^{-3} (1 - q^{1 + m} x^2)\det(W_m(x;q))
    \label{detW41xm}
  \end{equation}
  This, together with~\eqref{W410} concludes the proof of~\eqref{detW41x}.
\end{proof}

We now come to Conjecture~\ref{conj.ann} concerning a refinement of
the $\Ahat$-polynomial.  
Combining Theorems~\ref{thm.41a} and~\ref{thm.41b} we obtain explicit
linear $q$-difference equations for the descendant integrals with respect
to the $u$ and the $m$ variables. To simplify our presentation
keeping an eye on Equation~\eqref{41x-desc-fac}, let us define
a normalized version of the descendant state-integral by
\begin{equation}
  \label{Z41r}
  z_{\knot{4}_1,m,\mu}(u;\tau) =
  (-1)^{m+\mu}q^{-m/2}\tq^{-\mu/2}
  Z_{\knot{4}_1,m,\mu}(u;\tau) \,.
\end{equation}

\begin{theorem}
  \label{thm.41c}\rm
  $z_{\knot{4}_1,m,k}(\ub;\tau)$ is a $q$-holonomic function of $(m,u)$
  with annihilator ideal $\calI_{\knot{4}_1}$ given in~\eqref{ann41}.
  As a function of $u$ (resp., $m$) it is annihilated by the operators
  $\Ahat_{\knot{4}_1}(S_x,x,q^m,q)$ and
  $\Bhat_{\knot{4}_1}(S_m,x,q^m,q)$ (given respectively
  by~\eqref{rec41x} and~\eqref{rec41m}), whose classical limit is
   \begin{equation}
   \begin{aligned}
     &  \Ahat_{\knot{4}_1}(S_x,x,q^m,1)\\
     & = (-1 + q^{m} x^2) (x^2 S_x^2 -q^m (1 - x -2 q^{m} x^2 - q^{m}
     x^3 + q^{2 m} x^4) S_x + q^{3 m} x^2)
 \end{aligned}
  \end{equation}
  and
  \begin{equation}
    \Bhat_{\knot{4}_1}(S_m,x,q^m,1) = (1-x^{-1}S_m)(1-x^{-2}S_m) + q^m S_m.
  \end{equation}
  $\Ahat_{\knot{4}_1}(S_x,x,1,1)$ is the $A$-polynomial of the knot,
  $\Ahat_{\knot{4}_1}(S_x,x,1,q)$ is the (homogeneous part) of the
  $\Ahat$-polynomial of the knot and $\Bhat_{\knot{4}_1}(x^2y,x,1,1)$
  is the defining equation of the curve~\eqref{41critb}.
\end{theorem}

Note that although the two equations~\eqref{rec41m} and~\eqref{rec41x}
look quite different, they come from the common annihilating
ideal~\eqref{ann41} of rank $2$.  This explains their common order,
assuming that the ideal is generic. The annihilating ideal is easier
to describe than the $S_m$-free element~\eqref{rec41x} of it. In fact,
the first generator of $\calI_{\knot{4}_1}$ expresses $S_x$ as a
polynomial in $S_m$, and eliminating $S_m$, one obtains
equation~\eqref{rec41x} from~\eqref{rec41m}.  The characteristic
variety of $\calI_{\knot{4}_1}$ is a complex is a 2-dimensional
complex surface in $(\BC^*)^4$ and its intersection with a complex
3-torus contains two special curves, namely the $A$-polynomial and the
$B$-polynomial of the $\knot{4}_1$ knot.

\subsection{Taylor series expansion at $u=0$}
\label{sub.41near1}

The descendant state-integral is a meromophic function of $u$ which is
analytic at $u=0$ and factorizes in terms of descendant holomorphic
blocks~\eqref{41x-desc-fac}. In this section we compute the Taylor
series of the holomorphic blocks and of the state-integral at $u=0$
and show how the factorization of the descendant state-integral
\eqref{41x-desc-fac} reproduces~\eqref{410-desc-fac}.

We begin with some general comments valid for descendant holomorphic
blocks and state-integrals.  Since the descendant holomorphic blocks
are products of theta functions times $q$-hypergeometric sums, we need
to compute the Taylor expansion of each piece. For Taylor expansion of
the $q$-hypergeometric sums, we
\begin{subequations}
\begin{align}
  &\phi_n(u):=\frac{(q^{1+n}\re^u;q)_\infty}{(q^{1+n};q)_\infty}
    =\exp\left(-\sum_{l=1}^\infty \frac{1}{l!}E_l^{(n)}(q) u^l \right)
    \label{eq:phin}\\
  &\wt{\phi}_n(u):=\frac{(\tq;\tq)_\infty}{(\tq\re^u;\tq)_\infty}
    \frac{(\tq^{-1};\tq^{-1})_n}{(\tq^{-1}\re^u;\tq^{-1})_n}
    =\exp\left(\sum_{l=1}^\infty
    \frac{1}{l!}\wt{E}_l^{(n)}(\tq) u^l\right)
    \label{eq:tphin}
  \end{align}
\end{subequations}
from~\cite[Prop.2.2]{GK:qseries}, where
\begin{equation}
\label{En}
E_l^{(n)}(q)=\sum_{s=1}^\infty \frac{s^{l-1} q^{s(n+1)}}{1-q^s} 
\end{equation}
and
\begin{align}
  \wt{E}_l^{(n)}(\tq) =
  \begin{cases}
    -n+E_l^{(n)}(\tq) \quad &l=1\\
    E_l^{(n)}(\tq)\quad &l>1\;\text{odd}\\
    2E_l^{(0)}(\tq) - E_l^{(n)}(\tq)\quad &l>1\;\text{even} \,.
  \end{cases} 
\end{align}
For the Taylor series of the theta functions, we use the well-known
identity that expresses them in terms of quasi-modular forms (see,
eg.~\cite[Sec.8, Eqn(76)]{Za:quasimod}, or alternatively observe that
the theta functions that appear in the bilinear expressions of the
holomorphic blocks are exponentials of quadratic and linear forms in
$u$; see for instance~\eqref{ebx2}--\eqref{ebx3}.  Yet alternatively
(and this is the method that we will use below), when $r$ and $s$ are
nonzero integers with $r$ odd and positive, we can use the identity
\begin{equation}
  \label{thrs}
  \th(-q^{\frac{r}{2}} x^s;q)
  =(-1)^{\frac{r+1}{2}}q^{\frac{r^2-1}{8}}x^{-\frac{s(r+1)}{2}}(1-x^s)
  \phi_0(su)\phi_0(-su)  
\end{equation}
whereas when $r$ is odd and negative, we can use the $q$-difference
equation~\eqref{recth} to bring it to the cass of $r$ odd and positive.

One last comment is that the descendant holomorphic blocks are in
general meromorphic functions of $u$. However, their bilinear
combination that appears in the descendant state-integrals is regular
at $u=0$.

We now give the details of the Taylor series expansion of
$Z_{\knot{4}_1,m,\mu}(u;\tau)$ and of the descendant holomorphic
blocks for the $\knot{4}_1$ knot.
Using the definition of $J(x,y;q)$ and~\eqref{eq:phin}--\eqref{eq:tphin},
we find
\begin{align}
  A_m(\re^u;q)
  &= (1-\re^{-u})^{-2}(q;q)_\infty^{-3}\phi_0(u)^{-2}\phi_0(-u)^{-2}
  \sum_{n=0}^\infty \frac{q^{nm}}{(q;q)_n(q^{-1};q^{-1})_n}
    \phi_n(u) \re^{(2n+2m)u} \nn
  &= u^{-2}(q;q)_\infty^{-3}
    \left(\alpha^{(m)}_0(q)+u\alpha^{(m)}_1(q)+\cO(u^2)\right)
\label{eq:A-uexp-41}
\end{align}
where
\begin{subequations}
\begin{align}
  &\alpha^{(m)}_0(q) = G_m^0(q),\label{eq:a0-41}\\
  &\alpha^{(m)}_1(q) =
    \sum_{n=0}^\infty \frac{q^{nm}}{(q;q)_n(q^{-1};q^{-1})_n}
    (1+2n+2m-E_1^{(n)}(q)),\label{eq:a1-41}
\end{align}
\end{subequations}
and
\begin{align}
  B_m(\re^u;q)
  &= (1-\re^u)(q;q)_\infty^3\phi_0(u)\phi_0(-u)
    \sum_{n=0}^\infty \frac{q^{nm}}
    {(q;q)_n(q^{-1};q^{-1})_n}
    \phi_n(-u) \re^{(n+m)u} \nn
  &= -u(q;q)_\infty^3
    \left(\beta^{(m)}_0(q)+u\beta^{(m)}_1(q)+\cO(u^2)\right)
    \label{eq:B-uexp-41}
\end{align}
where
\begin{subequations}
\begin{align}
  &\beta^{(m)}_0(q) = G_m^0(q),\label{eq:b0-41}\\
  &\beta^{(m)}_1(q) =
    \sum_{n=0}^\infty \frac{q^{nm}}{(q;q)_n(q^{-1};q^{-1})_n}
    \left(\frac{1}{2}+n+m+E_1^{(n)}(q)\right).\label{eq:b1-41}
\end{align}
\end{subequations}
We notice that
\begin{equation}
  \beta^{(m)}_1(q) -\alpha^{(m)}_1(q) =
  -\frac{1}{2} G_m^1(q).
\end{equation}
Similarly, using the definition of $J(x,y;q^{-1})$, we find
\begin{align}
  A_m(\re^u;q^{-1})
  &=
    (1-\re^u)(q;q)_\infty^3\phi_0(u)^2\phi_0(-u)^2
    \sum_{n=0}^\infty\frac{q^{-nm}}{(q;q)_n(q^{-1};q^{-1})_n}
    \wt{\phi}_n(u)\re^{(2n+2m)u} \nn
  &=
    -u(q;q)_\infty^3\left(\wt{\alpha}^{(m)}_0(q)+u
    \wt{\alpha}^{(m)}_1(q) + \cO(u^2)\right)
    \label{eq:tA-uexp-41}
\end{align}
where
\begin{subequations}
\begin{align}
  &\wt{\alpha}^{(m)}_0(q) = G_{-m}^0(q),\label{eq:ta0-41}\\
  &\wt{\alpha}^{(m)}_1(q) = \sum_{n=0}^\infty
    \frac{q^{-nm}}{(q;q)_n(q^{-1};q^{-1})_n}
    \left(\frac{1}{2}+n+2m+E_1^{(n)}(q)\right),\label{eq:ta1-41}
\end{align}
\end{subequations}
and
\begin{align}
  B_m(\re^u;q^{-1})
  &=
  \frac{\phi_0(u)^{-1}\phi_0(-u)^{-1}}{(1-\re^{-u})^2(q;q)_\infty^3}
    \sum_{n=0}^\infty\frac{q^{-nm}}{(q;q)_n(q^{-1};q^{-1})_n}
    \wt{\phi}_n(-u)\re^{(n+m)u} \nn
  &=
  u^{-2}(q;q)_\infty^{-3}\left(\wt{\beta}^{(m)}_0(q)+u
    \wt{\beta}^{(m)}_1(q) + \cO(u^2)\right)
    \label{eq:tB-uexp-41}
\end{align}
where
\begin{subequations}
\begin{align}
  &\wt{\beta}^{(m)}_0(q) = G_{-m}^0(q),\label{eq:tb0-41}\\
  &\wt{\beta}^{(m)}_1(q) = \sum_{n=0}^\infty
    \frac{q^{-nm}}{(q;q)_n(q^{-1};q^{-1})_n}
    (1+2n+m-E_1^{(n)}(q)).\label{eq:tb1-41}
\end{align}
\end{subequations}
We also notice
\begin{equation}
  \wt{\beta}_1^{(m)}(q) - \wt{\alpha}_1^{(m)}(q) =
  \frac{1}{2} G_{-m}^1(q).
\end{equation}
Applying these results to the right hand side of \eqref{41x-desc-fac},
we find the $O(1/u)$ contributions from~\eqref{eq:tA-uexp-41}
and~\eqref{eq:tB-uexp-41} cancel, and the $O(u^0)$ contributions
reproduce exactly~\eqref{410-desc-fac}. Notice that the $\sqrt{\tau}$
terms that appear in~\eqref{410-desc-fac} come from expanding in terms
of $2\pi\bb\,u$ and $2\pi\bb^{-1}\,u$ in~\eqref{41x-desc-fac}.
    
As an application of the above computations, we obtain proof of a
simplified formula for the $q$-series $G^1(q)$ from~\eqref{Gm0} which
was found experimentally in~\cite{GZ:qseries}.

\begin{proposition}
  \label{prop.G}
  For $|q|<1$, we have:
  \begin{equation}
  \label{GGformula}
  G^1(q) = \sum_{n=0}^\infty (-1)^n \frac{q^{n(n+1)/2}}{(q;q)_n^2}
  (6n+1) \,.
  \end{equation}
\end{proposition}

\begin{proof}
  We first show that the definition \eqref{Gm0} can equally be written
  as
  \begin{equation}
    G^1(q) = \sum_{n=0}^\infty (-1)^n \frac{q^{n(n+1)/2}}{(q;q)_n^2}
    \left(1+2n-4E_1^{(n)}(q)\right).\label{eq:G1-im}
  \end{equation}
  By definition
  \begin{equation}
    E_1^{(n)}(q) = \sum_{s=1}^\infty \frac{q^{s(n+1)}}{1-q^s},
  \end{equation}
  they satisfy the recursion relation
  \begin{equation}
    E_1^{(n)}(q) - E_1^{(n-1)}(q) = -\frac{q^n}{1-q^n},
  \end{equation}
  and therefore
  \begin{equation}
    E_1^{(n)}(q) = E_1^{(0)}(q) -\sum_{j=1}^n\frac{q^j}{1-q^j}.
  \end{equation}
  Using the identification $E_1(q) = 1-4E_1^{(0)}(q)$, one can then
  easily show that \eqref{eq:G1-im} is the same
  as \eqref{Gm0}.

  \eqref{GGformula} follows from \eqref{eq:G1-im} thanks to the
  non-trivial identity
  \begin{equation}
    \sum_{n=0}^\infty \frac{n+E_1^{(n)}(q)}{(q;q)_n(q^{-1};q^{-1})_n}
    = 0,\quad |q|<1,
    \label{eq:E1n-id}
  \end{equation}
  which now we prove.
  The crucial fact we use is that when $|q|<1$, $J(x,y;q)$ is
  symmetric between $x$ and $y$ (see a proof in \cite{Beem}).
  Let us consider the following expansion in small $u$
  \begin{align}
    J(\re^{-u},\re^u;q) =
    &\sum_{n=0}^\infty (q^{1+n}\re^u;q)_\infty\frac{\re^{-nu}}{(q^{-1};q^{-1})_n} =
    (q;q)_\infty\sum_{n=0}^\infty \frac{\phi_n(u)\re^{-nu}}{(q;q)_n(q^{-1};q^{-1})_n}\nn =
    &(q;q)_\infty\sum_{n=0}^\infty \frac{1-(n+E_1^{(n)}(q))u}{(q;q)_n(q^{-1};q^{-1})_n} +
    O(u^2). 
  \end{align}
  Since $J(\re^{-u},\re^u;q) = J(\re^{u},\re^{-u};q)$, the coefficient
  of $u$ (and in fact, of any odd power of $u$) in the expansion above
  vanishes, which leads to \eqref{eq:E1n-id}.
\end{proof}

As a second application, we demonstrate that Theorem~\ref{thm.410},
especially the identities \eqref{det41}, \eqref{WWT41}, as well as the
recursion relation \eqref{41qdiff}, can be proved by taking the $u=0$
limit of the analogue identities in Theorem~\ref{thm.41b}.

Using the expansion formulas of holomorphic blocks \eqref{eq:A-uexp-41},
\eqref{eq:B-uexp-41}, \eqref{eq:tA-uexp-41}, \eqref{eq:tB-uexp-41},
the Wronskians can be expanded as
\begin{align}
  W_m(\re^u;q) =
  &\begin{pmatrix}
    G_m^0(q)+u\alpha_1^{(m)}(q)&G_m^0(q)+u\beta_1^{(m)}(q)\\
    G_{m+1}^0(q)+u\alpha_1^{(m+1)}(q)&G_{m+1}^0(q)+u\beta_1^{(m+1)}(q)
  \end{pmatrix}
  \begin{pmatrix}
    u^{-2}(q;q)_\infty^{-3}&0\\
    0&-u(q;q)_\infty^3
  \end{pmatrix},
  \label{eq:Wp-uexp-41}\\
  W_m(\re^u;q^{-1}) =
  &\begin{pmatrix}
    G_{-m}^0(q)+u\wt{\alpha}_1^{(m)}(q)
    &G_{-m}^0(q)+u\wt{\beta}_1^{(m)}(q)\\
    G_{-m-1}^0(q)+u\wt{\alpha}_1^{(m+1)}(q)
    &G_{-m-1}^0(q)+u\wt{\beta}_1^{(m+1)}(q)
  \end{pmatrix}
  \begin{pmatrix}
    -u(q;q)_\infty^{3}&0\\
    0&u^{-2}(q;q)_\infty^{-3}
  \end{pmatrix}
  \label{eq:Wn-uexp-41}
\end{align}
Taking the determinant of \eqref{eq:Wp-uexp-41}, we find
\begin{equation}
  \det W_m(\re^u;q) = \frac{1}{2} \det W_m(q) + O(u)
\end{equation}
which together with the $u$-expansion of the right hand side of
\eqref{W410} leads to the determinant identity \eqref{det41}.
Furthermore, by substituting \eqref{eq:Wp-uexp-41},
\eqref{eq:Wn-uexp-41} into the Wronskian relation \eqref{WW41b}, the
latter also reduces in the leading order to the determinant identity
\eqref{det41}.  On the other hand, the Wronskian relation
\eqref{WWT41b} is equivalent to
\begin{equation}
  W_m(q) = 2
  \begin{pmatrix}
    0&-1\\1&0
  \end{pmatrix}\left(W_m(q)^{-1}\right)^T
  \begin{pmatrix}
    0&1\\-1&0
  \end{pmatrix}
\end{equation}
which can be proved directly by expressing the inverse matrix on the
right hand side by minors and determinant, using the explicit value of
the determinant given by the identity \eqref{det41}.

Finally, from the expression \eqref{eq:a0-41},\eqref{eq:b0-41} of the
leading order coefficients $\alpha_0^{(m)}(q)$, $\beta_0^{(m)}(q)$ of
$A_m(\re^u;q)$, $B_m(\re^u;q)$ in the expansion of $u$, one concludes
that the recursion relation \eqref{41qdiff} should be the $u=0$ limit
of the recursion relation \eqref{rec41m} in $m$, and one can easily
check it is indeed the case.

\subsection{Stokes matrices near $u=0$}
\label{sub.41stokes}

In this section we conjecture a formula for the Stokes matrices of the
asymptotic series $\varphi(x;\tau)$.  When we turn on the non-vanising
deformation parameter $u$, the resurgent structure discussed in
Section~\ref{sub.41x=1} undergoes significant changes.  Compared to
Figure~\ref{fig:sings-u0-41} there are many more singular points in
the Borel plane whose positions depend on $u$ in addition to $\tau$,
and the Stokes matrices also become $u$-dependent.  However, if we
focus on the case when $u$ is not far away from zero, equivalent to
$x$ not far away from 1, the resurgent data is holomorphic in $u$ and
reduces to those in Section~\ref{sub.41x=1} in the $u=0$ limit.  For
instance, each singular point in Figure~\ref{fig:sings-u0-41} splits
to a cluster of neighoring singular points separated with distance
$\log x$ as shown in Figure~\ref{fig:sings-m-41}. In particular, each
singular point $\iota_{i,j}$ on the real axis splits to a cluster of
three, in accord with the off-diagonal entries $\pm 3$ in
\eqref{eq:S0u041}, and if we choose real $x$, the split singularities
still lie on the real axis.  As in Section~\ref{sub.41x=1}, we label
the four regions separating singularities on real axis and all the
others by $I,II,III,IV$ (see Figure~\ref{fig:secs-m-41}).  In each of
the four regions, we have the following the results.

\begin{figure}
\leavevmode
\begin{center}
\includegraphics[height=7cm]{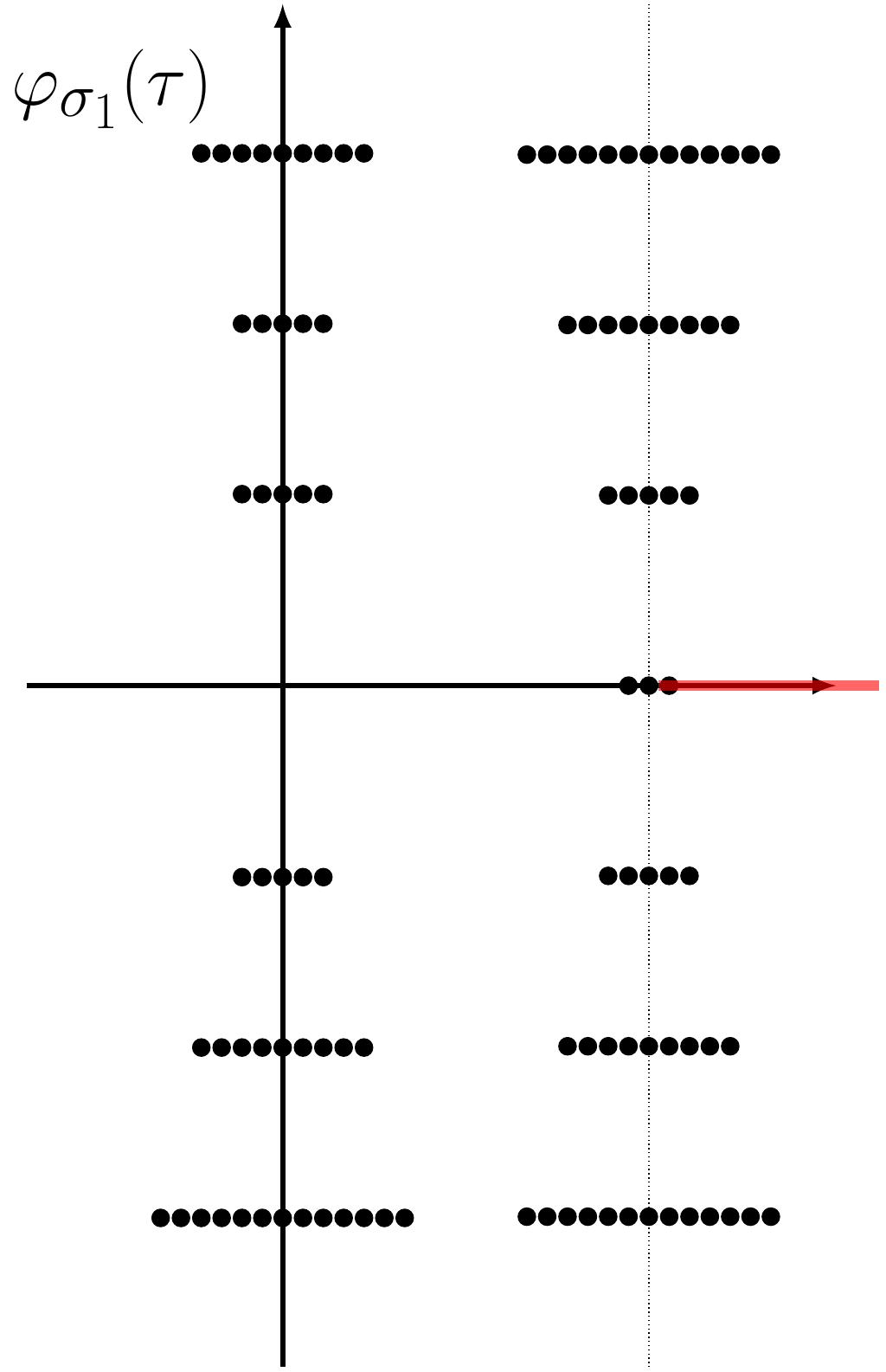}\hspace{4ex}
\includegraphics[height=7cm]{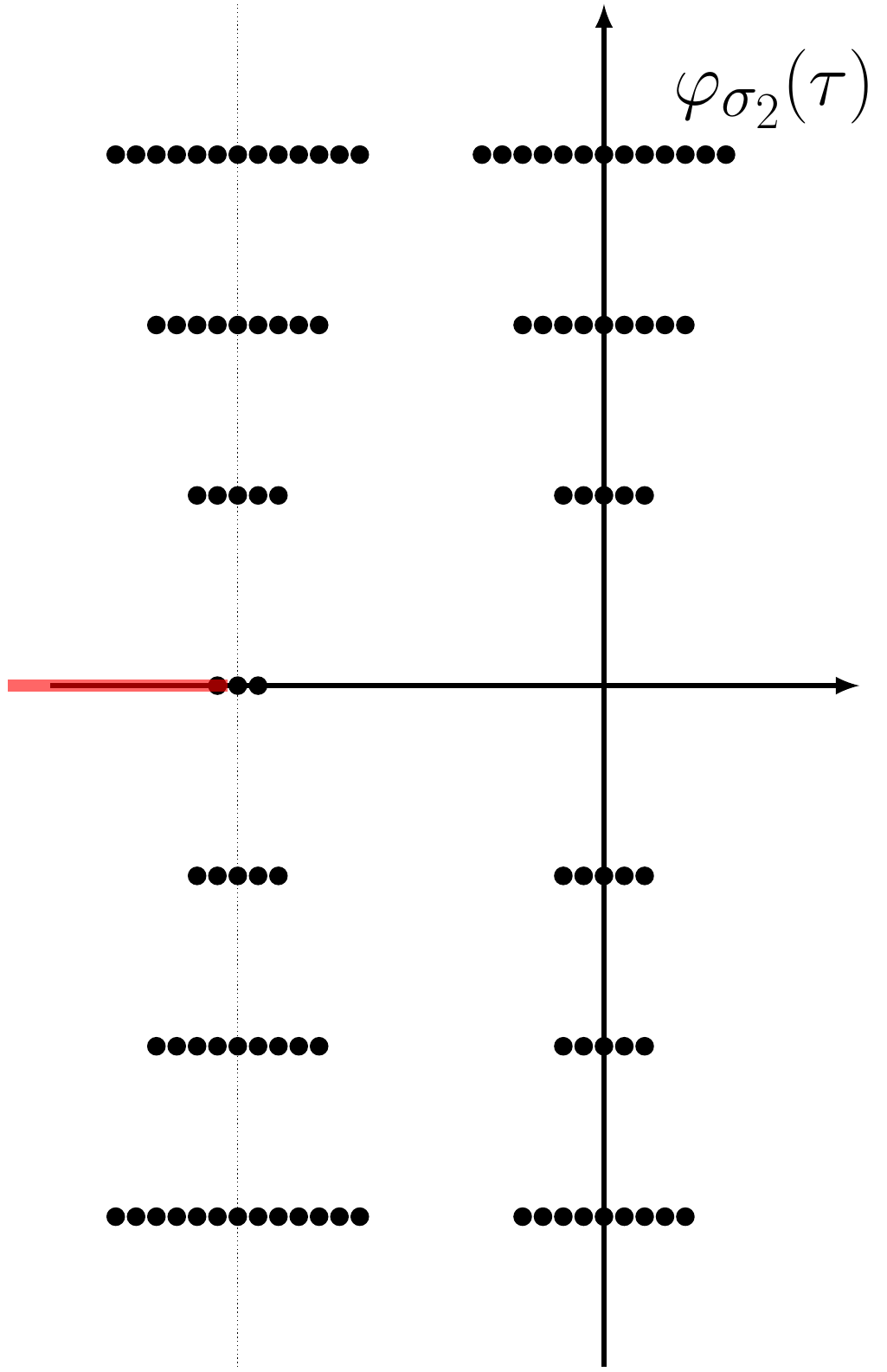}
\end{center}
\caption{The singularities in the Borel plane for the series
  $\varphi_{\sigma_{j}} (x;\tau)$ with $j=1,2$ of knot $\knot{4}_1$
  where $x$ is close to 1. The shortest vertical spacing between
  singularities is $2\pi\ri$, and the horizontal spacing between
  neighboring singularities in each cluster is $\log x$.}
\label{fig:sings-m-41}
\end{figure}

\begin{figure}
\leavevmode
\begin{center}
\includegraphics[height=7cm]{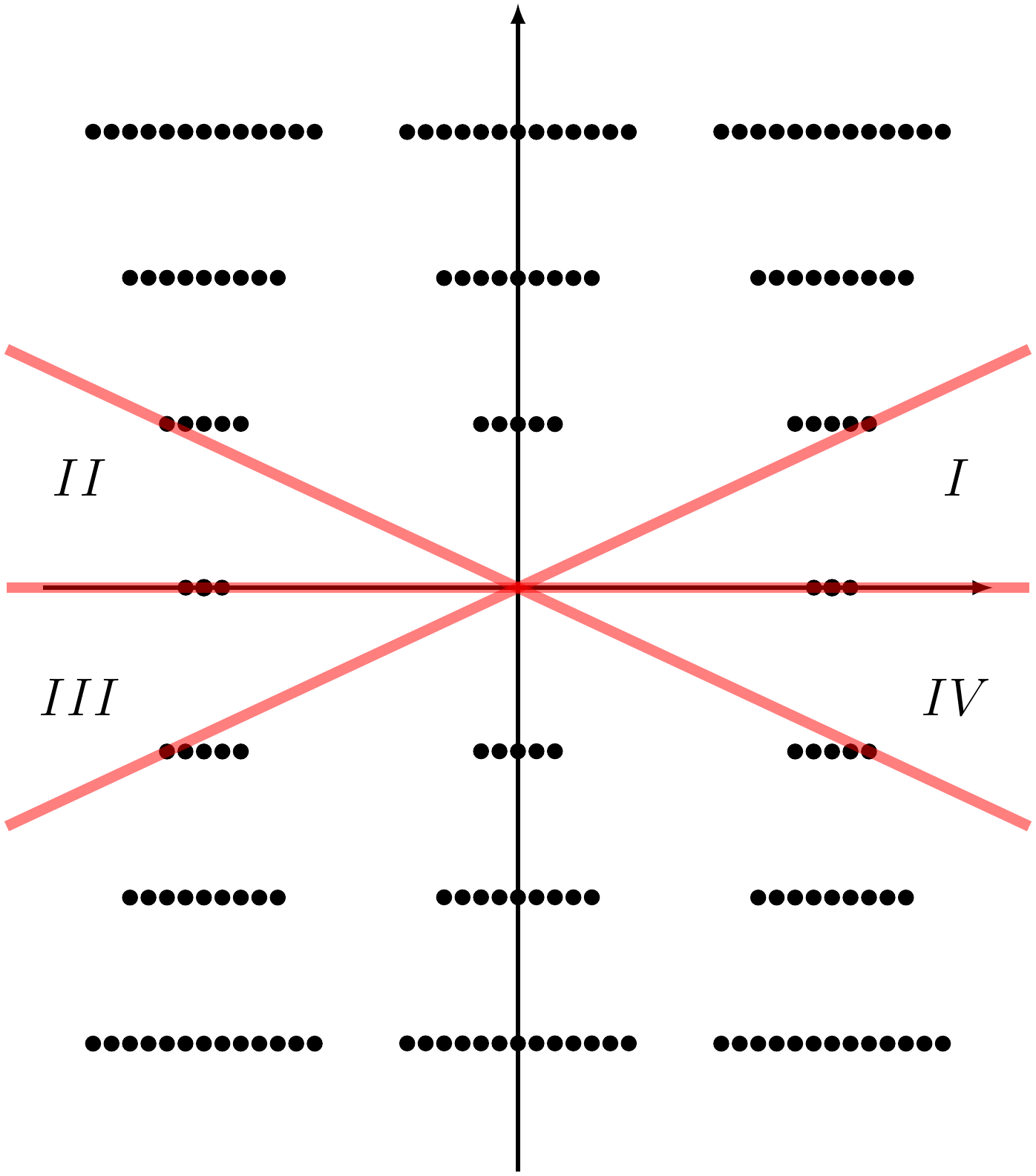}
\end{center}
\caption{Four different sectors in the $\tau$-plane for $\Phi(u;\tau)$
  of knot $\knot{4}_1$ with $u$ close to zero.}
\label{fig:secs-m-41}
\end{figure} 

\begin{conjecture}
  \label{conj.41M}
  The asymptotic series and the holomorphic blocks are related
  by~\eqref{abrelx} with the diagonal matrix $\Delta(\tau)$ as
  in~\eqref{D41}, where matrices $M_R(x;q)$ are given in terms of the
  matrices $W_{-1}(x;q)$ as follows 
\begin{subequations}
\begin{align}
  M_{I}(x;q) =
  &\begin{pmatrix}
    1&0\\0&-1
  \end{pmatrix} W_{-1}(x;q)^T
            \begin{pmatrix}
              0&1\\-x^{-1}&1
            \end{pmatrix}, & |q|<1\,,
  \label{eq:M1-41}\\
  M_{II}(x;q) =
  &W_{-1}(x^{-1};q)^T
    \begin{pmatrix}
      -x^{-1}&0\\-x^{-1}&1
    \end{pmatrix}, & |q|<1\,,
  \label{eq:M2-41}\\
  M_{III}(x;q) =
  &W_{-1}(x^{-1};q)^T
    \begin{pmatrix}
      -x^{-1}&0\\1+x&1
    \end{pmatrix}, & |q|>1\,,
  \label{eq:M3-41}\\
  M_{IV}(x;q) =
  &\begin{pmatrix}
    1&0\\0&-1
  \end{pmatrix} W_{-1}(x;q)^T
            \begin{pmatrix}
              0&1\\-x^{-1}&-x^{-1}-x^{-2}
            \end{pmatrix}, & |q|>1\,.
                            \label{eq:M4-41}
\end{align}
\end{subequations}
\end{conjecture} 
The above conjecture completely determines the resurgent structure of
$\Phi(\tau)$. Indeed, it implies that the Stokes matrices, defined in 
Equations~\eqref{4G} and~\eqref{SG}, are explicitly given by:

\begin{subequations}
\begin{align}
    \ms{S}^+(x;q) =
    &\begin{pmatrix}
      0&-1\\
      x^{-1}&-1-x
    \end{pmatrix}
              \cdot W_{-1}(x^{-1};q^{-1})\cdot W_{-1}(x;q)^T\cdot
              \begin{pmatrix}
                0&x\\
                -1&-1-x^{-1}
              \end{pmatrix} ,\quad |q|<1\,,
  \label{eq:Sp-41}\\
    \ms{S}^-(x;q)=
    &\begin{pmatrix}
      x&x\\
      0&-1
    \end{pmatrix}
              \cdot W_{-1}(x;q)\cdot W_{-1}(x^{-1};q^{-1})^T\cdot
              \begin{pmatrix}
                x^{-1}&0\\
                x^{-1}&-1
              \end{pmatrix},\quad |q|<1\,.
                                \label{eq:Sn-41}
\end{align}
\end{subequations}
We remark that since $s(\Phi_\s)(x;\tau)$ for $\s=1,2$ transform under
the reflection $\pi: x\mapsto x^{-1}$ uniformly by \eqref{eq:-xPhi} (see
the comment below), the Stokes matrices should be invariant under
$\pi$, and we have checked that \eqref{eq:Sp-41},\eqref{eq:Sn-41}
indeed satisfy this consistency condition.

In the $q\mapsto 0$ limit,
\begin{equation}
  \ms{S}^+(x;0) =
  \begin{pmatrix}
    1&x^{-1}+1+x\\0&1
  \end{pmatrix},\quad
  \ms{S}^-(x;0) =
  \begin{pmatrix}
    1&0\\-x^{-1}-1-x&1
  \end{pmatrix}
  \label{eq:S0x41}
\end{equation}
A curious corollary of our computation is that the matrices of
integers \eqref{eq:S0u041} from~\cite{gh-res,GZ:kashaev} which relates
the asymptotics of the coefficients of $\varphi(\tau)$ to the
coefficients themselves, spreads out to the matrices \eqref{eq:S0x41}
with entries in $\IZ[x^{\pm 1}]$.

Using the unique factorization Lemma~\ref{lem.fac} and the Stokes
matrix $\sf{S}$ from above, we can compute the Stokes constants and
the corresponding matrix $\calS$ of Equation~\eqref{calS} to arbitrary
order in $q$, and we find that
\begin{align}
  \mc{S}^+_{\s_1,\s_1}(x;q) =
  &\ms{S}^+(x;q)_{1,1}-1\nn=
  &(-2 - x^{-2} - 2x^{-1} - 2 x - x^2)q
    +(-3 - x^{-2} - 2x^{-1} - 2 x - x^2)q^2+\cO(q^3),\\
  \mc{S}^+_{\s_1,\s_2}(x;q) =
  &\ms{S}^+(x;q)_{1,2}/\ms{S}^+(x;q)_{1,1}
  -(\mc{S}^{(1,0)}_{\s_1,\s_2}x+\mc{S}^{(0,0)}_{\s_1,\s_2}
  +\mc{S}^{(-1,0)}_{\s_1,\s_2}x^{-1})\nn=
  &(3 + x^{-2} + 2x^{-1} + 2 x + x^2)q\nn
  &+(17 + x^{-4} + 4x^{-3} + 9x^{-2} + 15x^{-1} + 15 x + 9 x^2 + 4 x^3 + x^4)
  q^2+\cO(q^3),\\
  \mc{S}^+_{\s_2,\s_1}(x;q) =
  &\ms{S}^+(x;q)_{2,1}/\ms{S}^+(x;q)_{1,1}\nn=
  &(-3 - x^{-2} - 2x^{-1} - 2 x - x^2)q\nn
  &+(-17 - x^{-4} - 4x^{-3} - 9x^{-2} - 15 x^{-1} - 15 x - 9 x^2 - 4 x^3 - x^4)
  q^2+\cO(q^3),\\
  \mc{S}^+_{\s_2,\s_2}(x;q) =
  &\ms{S}^+(x;q)_{2,2} -1-
    \ms{S}^+(x;q)_{1,2}\ms{S}^+(x;q)_{2,1}/\ms{S}^+(x;q)_{1,1}\nn=
  &(2 + x^{-2} + 2x^{-1} + 2 x + x^2) q \nn
  &+ (17 + x^{-4} + 4x^{-3} + 9x^{-2} + 14x^{-1} + 14 x + 9 x^2 + 4 x^3 + x^4)
  q^2+\cO(q^3).
\end{align}
They enjoy the symmetry
\begin{equation}
  \mc{S}^{+}_{\s_1,\s_2}(x;q) =
  -\mc{S}^{+}_{\s_2,\s_1}(x;q)\,,
\end{equation}
and experimentally, it appears that the entries of the matrix
$\mc{S}^+(x;q) = (\mc{S}^+_{\s_i,\s_j}(x;q))$ (except the upper-left
one) are (up to a sign) in $\BN[x^{\pm 1}][[q]]$. 
Similarly we can extract the Stokes constants
$\mc{S}^{(\ell,-k)}_{\s_i,\s_j}$ associated to the singularities in
the lower half plane, and assemble into $q^{-1}$-series
$\mc{S}^{-}_{\s_i,\s_j}(x;q^{-1})$.  We find they are related to
$\mc{S}^+_{\s_i,\s_j}(x;q)$ by
\begin{gather}
  \mc{S}^{-}_{\s_i,\s_j}(x;q) = -\mc{S}^{+}_{\s_j,\s_i}(x;q),\;\;i\neq j\,
  \nn
  \mc{S}^{-}_{\s_1,\s_1}(x;q) = \mc{S}^{+}_{\s_2,\s_2}(x;q),\;\;
  \mc{S}^{-}_{\s_2,\s_2}(x;q) = \mc{S}^{+}_{\s_1,\s_1}(x;q)\,.
\end{gather}

Let us now verify Conjecture~\ref{conj.S3D}.
From \eqref{eq:Sp-41} we find that indeed
\begin{equation}
  \ms{S}^+(x;q) \stackrel{\cdot}{=} W_{-1}(x^{-1};q^{-1})\cdot W_{-1}(x;q)^T.
\end{equation}
Using the recursion relation \eqref{eq:W41rec} and the relation
between two Wronskians \eqref{W41xm}, we further find 
\begin{equation}
  \ms{S}^+(x;q) \stackrel{\cdot\cdot}{=} \cW_{0}(x^{-1};q^{-1})\cdot
  \cW_{0}(x;q)^T.
  \label{eq:SW41}
\end{equation}
If we use the uniform notation for all holomorphic blocks
\begin{equation}
  (B_{K}^\alpha(x;q))_{\alpha=1,2} = (A_0(x;q),B_0(x;q)),
\end{equation}
the right hand side of \eqref{eq:SW41} reads 
\begin{equation}
  \cW_{0}(x^{-1};q^{-1})\cdot \cW_{0}(x;q)^T =
  \left(
    \sum_{\alpha}B^{\alpha}_{\knot{4}_1}(q^jx;q)
    B^{\alpha}_{\knot{4}_1}(q^{-i}x^{-1};q^{-1})
  \right)_{i,j=0,1}.
  \label{eq:WW41}
\end{equation}
which is precisely the right hand side of \eqref{eq:SInd} in
Conjecture~\ref{conj.S3D} following \eqref{eq:IndB}.\footnote{Note
  that because the form of the state integral in \cite{Beem} is
  slightly different from that in \cite{AK}, which we adopt, our
  convention for holomorphic blocks is also different from
  \cite{Beem}.  As a result, the entries of $\eqref{eq:WW41}$ equate
  the DGG indices computed in \cite{Beem} up to a prefactor
  \begin{equation}
    \left(\cW_{0}(x^{-1};q^{-1})\cdot \cW_{0}(x;q)^T\right)_{i+1,j+1} =
    (-q^{1/2})^{j-i}\text{Ind}_{\knot{4}_1}^{\text{rot}}
    (j-i,q^{\frac{j+i}{2}}x;q),\quad i,j=0,1.
  \end{equation}
  If we take this into account, Conjecture~\ref{conj.S3D} should be
  modified slightly by stating the accompanying matrices on the left
  and on the right are in $GL(2,\IZ(x,q^{1/2}))$.}
In addition, the forms of the accompanying matrices on the left and on the
right are such that the $(1,1)$ entry of $\ms{S}^+(x;q)$ equates
exactly the DGG index with no magnetic flux.
By explicit calculation,
\begin{equation}
\label{S11v}
\begin{aligned}
  \ms{S}^+(x;q)_{1,1}& =
  J(x, x^{-1};q) J(x, x^{-1}; q^{-1}) + J(x^{-2}, x^{-1};q) J(x^2, x; q^{-1}) \\
  &=1 -(2x^{-2}+ x^{-1}+2+x + 2x^2)q - (x^{-2}+2x^{-1} +3+ 2x + 
  x^2) q^2 + O(q^3),
  \end{aligned}
\end{equation}
which is the $\text{Ind}^{\text{rot}}_{\knot{4}_1}(0,x;q)$ given in
\cite{Beem}.

\subsection{The Borel resummations of the asymptotic series $\Phi$}
\label{sub.borel41}

In this section we explain how Conjecture~\ref{conj.41M} identifies the
Borel resummations of the factorially divergent series $\Phi(x;\tau)$
with the descendant state-integrals, thus lifting the Borel resummation
to holomorphic functions on the cut-plane $\BC'$. This is
interesting theoretically, but also practically in the numerical
computation of Borel resummations.

After multiplying the inverse of $M_R(\tx,\tq)$ from the left
on both sides of \eqref{abrelx}, we can also express the Borel sums
$s_R(\Phi)(x;\tau)$ in each region in terms of holomorphic functions
of $\tau \in \IC\backslash \IR$ as follows

\begin{corollary}(of Conjecture~\ref{conj.41M})
  \label{cor.41borel}
  We have
\begin{subequations}
\begin{align}
  s_{I}(\Phi)(x;\tau) =
  &\begin{pmatrix}
    -\tx&1+\tx^{-1}\\0&1
  \end{pmatrix} W_{-1}(\tx;\tq^{-1})\Delta(\tau)B(x;q)\,,
                        \label{eq:HU1-41}\\
  s_{II}(\Phi)(x;\tau) =
  &\begin{pmatrix}
    0&-\tx\\1&-\tx-\tx^2
  \end{pmatrix} W_{-1}(\tx^{-1};\tq^{-1})
               \begin{pmatrix}
                 1&0\\0&-1
               \end{pmatrix}\Delta(\tau) B(x;q)\,,
                         \label{eq:HU2-41}\\
  s_{III}(\Phi)(x;\tau) =
  &\begin{pmatrix}
    0&-\tx\\1&1   
  \end{pmatrix} W_{-1}(\tx^{-1};\tq^{-1})
               \begin{pmatrix}
                 1&0\\0&-1
               \end{pmatrix}\Delta(\tau) B(x;q)\,,
                         \label{eq:HU3-41}\\
  s_{IV}(\Phi)(x;\tau)=
  &\begin{pmatrix}
    -\tx&-\tx\\0&1   
  \end{pmatrix}W_{-1}(\tx,\tq^{-1})\Delta(\tau) B(x;q)\,.
                  \label{eq:HU4-41}
\end{align}
\end{subequations}
where the right hand side of~\eqref{eq:HU1-41}--\eqref{eq:HU4-41} are
holomorphic functions of $\tau \in \IC'$, as they are linear
combinations of the descendants \eqref{41x-desc-fac}.
\end{corollary}

The asymptotic series of the $\knot{4}_1$ knot have the
symmetry~\eqref{Phi4112} due to the fact that it is an amphichiral
knot. This gives a symmetry of the state-integral.

\begin{proposition}
  \label{prop.Z41uu}
  (Assuming Conjecture~\ref{conj.41M}) We have:
  \begin{equation}
  \label{Z41uu}
  Z_{\knot{4}_1}(u;\tau) = e^{2\pi (\bb+\bb^{-1})u}
  Z_{\knot{4}_1}(-u;\tau) \,.
  \end{equation}
\end{proposition}

\begin{proof}
  The second line of~\eqref{eq:HU1-41} indicates that in region $I$
\begin{equation}
  Z_{\knot{4}_1}(u;\tau) = s_I(\Phi_2)(x;\tau).
\end{equation}
Recall the structure of $\Phi_2(x;\tau)$ from \eqref{Phi41}
\begin{equation}
  \Phi_2(x;\tau) = \exp\left(\frac{V(x,y_2(x))}{2\pi\ri\tau}\right)
  \frac{1}{\sqrt{\ri\delta(x,y_2(x))}}
  \sqrt{\ri\delta}\varphi(x,y_2(x);\tau).
\end{equation}
Here
$\sqrt{\ri\delta}\varphi_2(x;\tau):=\sqrt{\ri\delta}\varphi(x,y_2(x);\tau)$
is an asymptotic series in $\tau$ with
$\sqrt{\ri\delta}\varphi_2(x;0)=1$.  The coefficients of the series
$\sqrt{\ri\delta}\varphi_2(x;\tau)$ are invariant under the
transformation $x\mapsto 1/x$ (see for instance \eqref{eq:vf-41}; this
is also true for
$\sqrt{\ri\delta}\varphi_1(x;\tau) =
\sqrt{\ri\delta}\varphi(x,y_1(x);\tau)$), and thus
\begin{equation}
  s(\sqrt{\ri\delta} \varphi_2)(1/x;\tau) = s(\sqrt{\ri\delta}
  \varphi_2)(x;\tau).
  \label{eq:-xphi}
\end{equation}
On the other hand, from definition \eqref{eq:delta-41} of
$\delta_2(x):=\delta(x,y_2(x))$, it is clear that
\begin{equation}
  \delta_2(1/x) = x^2\delta_2(x).
  \label{eq:-xdelta}
\end{equation}
Finally, to study the behavior of $V_2(x):=V(x,y_2(x))$ under the
transformation $x\mapsto 1/x$, it is convenient to do the change of
variables $y = x^{-1}+\ty$, so that the Equation~\eqref{41critb}
satisfied by $y$ becomes
\begin{equation}
  1-(1-x-x^{-1})\ty+\ty^2 = 0
\end{equation}
which is manifestly invariant under this transformation, and thus
$\ty(1/x) = \ty(x)$.
Expressed in terms of this variable
\begin{align}
  V_2(x) =
  &-\Li_2(-x^{-1}\ty_2)-\Li_2(-x \ty_2) -\frac{\pi}{3}
    -\frac{1}{2}\log^2(x\ty_2^{-1})\nn
  &+\log(1+x \ty_2^{-1})\log(-x\ty_2^{-1})
    -\log(1+x\ty_2)\log(-x\ty_2)\nn
  &-\frac{1}{2}\log^2(-1-x\ty_2^{-1})
    +\frac{1}{2}\log^2(-x^{-1}-\ty_2)
    +2\log x \log(-x^{-1}-\ty_2),
\end{align}
and it has the property that
\begin{equation}
  V_2(x^{-1}) = V_2(x) - 2\pi\ri x
  \label{eq:-xV}
\end{equation}
This can be proven by differentiating both sides with respect to $x$,
and reducing it to an identity of rational functions on the curve $S$.
Combining \eqref{eq:-xphi},\eqref{eq:-xdelta},\eqref{eq:-xV}, we have
\begin{equation}
  s_I(\Phi_2)(x^{-1};\tau) = x^{-1}\tx^{-1} s_I(\Phi_2)(x;\tau)
  \label{eq:-xPhi}
\end{equation}
which implies~\eqref{Z41uu}. We comment in the passing that the
identity~\eqref{eq:-xPhi} is true for both
$s(\Phi_{1,2})(x;\tau)$ for any $\tau\in\IC$ whenever the asymptotic
series is Borel summable.

Once we have established the identity \eqref{Z41uu} in region $I$, it
can be extended to $\tau\in\IC'$ by the holomorphicity of
\eqref{41x-desc-fac}.
\end{proof}

\subsection{Stokes matrices for real $u$}
\label{sub.41stokesp}

In Section~\ref{sub.41stokes} we only considered the resurgent
structure for $x$ near $1$, or equivalently, $u=\log x $ near $0$.
When $x$ is arbitrary, the resurgent structure of the vector $\Phi_\s(x;\tau)$
could be very different.  According to the Picard--Lefschetz theory
(for review, see for instance \cite{Witten:2010cx}), when a set of
asymptotic series originates from a (path) integral, the Borel sum of
each asymptotic series is the evaluation of the integral along a
Lefschetz thimble anchored to a critical point.  In the $x$-plane,
there are walls of marginal stability which start
from the roots to the discriminant \eqref{disc41y} and which end at
infinity.  When we cross such a Stokes line, Lefschetz thimbles jump
leading to linear transformations of the Borel summed asymptotic
series.  In this section we extend slightly the discussion of
Section~\ref{sub.41stokes} by considering the resurgent structure of
$\Phi_\s(x;\tau)$ for generic positive $x$ (see \cite{Beem} for a
similar discussion in complex Chern-Simons theory).  The positive real
axis is divided by the two real solutions to~\eqref{disc41y} 

\begin{equation}
  \label{z35}
  x_{\pm} = \frac{1}{2}(3\pm \sqrt{5})
\end{equation}
to three intervals
\begin{equation}
  (0,x_{-}),\;(x_-,x_+),\;(x_+,\infty).
\end{equation}
The middle interval is covered by Section~\ref{sub.41stokes}, while
the first (labeled by $<$) and third (labeled by $>$) intervals are
discussed below.

\begin{figure}[htpb!]
\leavevmode
\begin{center}
\includegraphics[height=7cm]{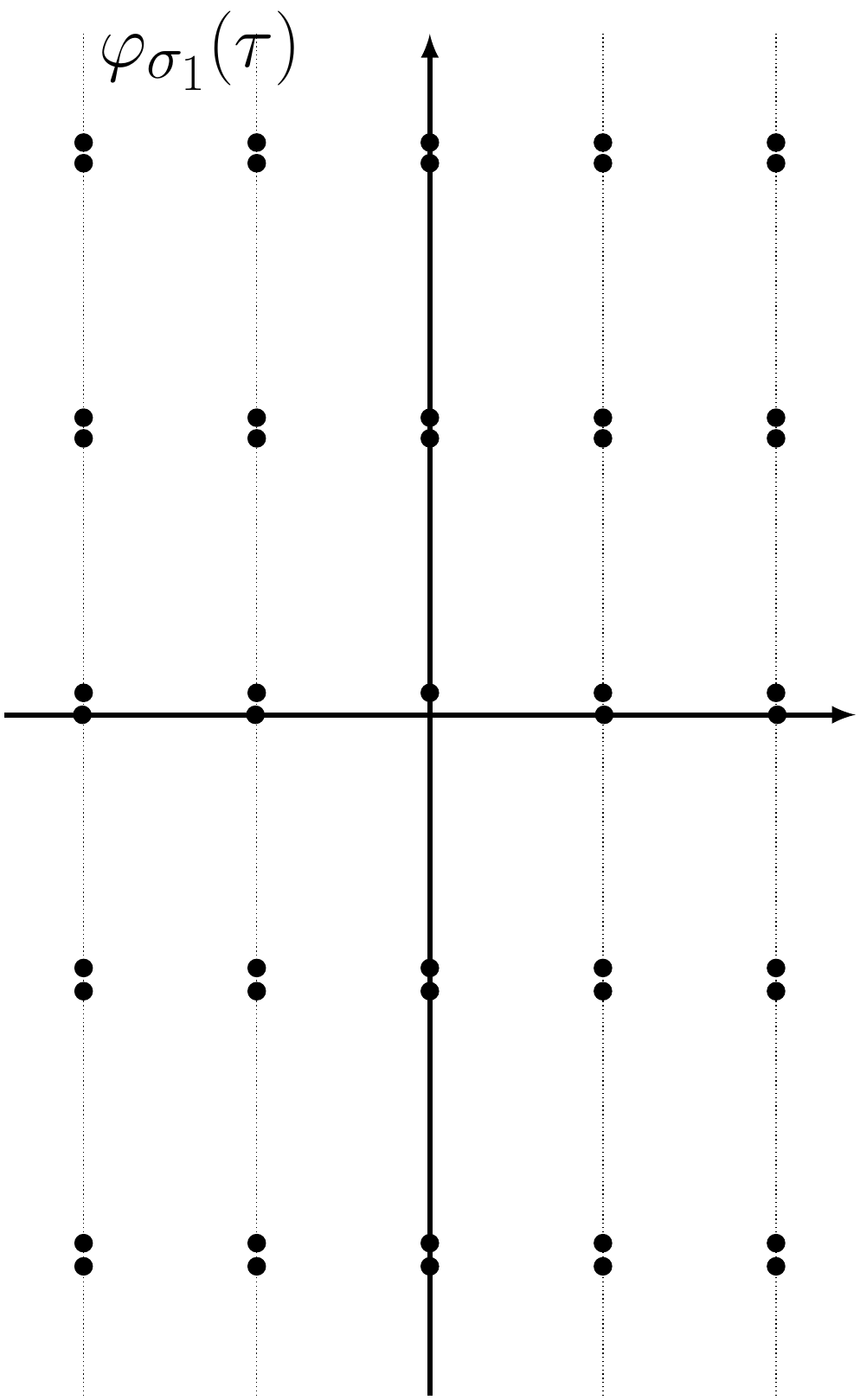}\hspace{4ex}
\includegraphics[height=7cm]{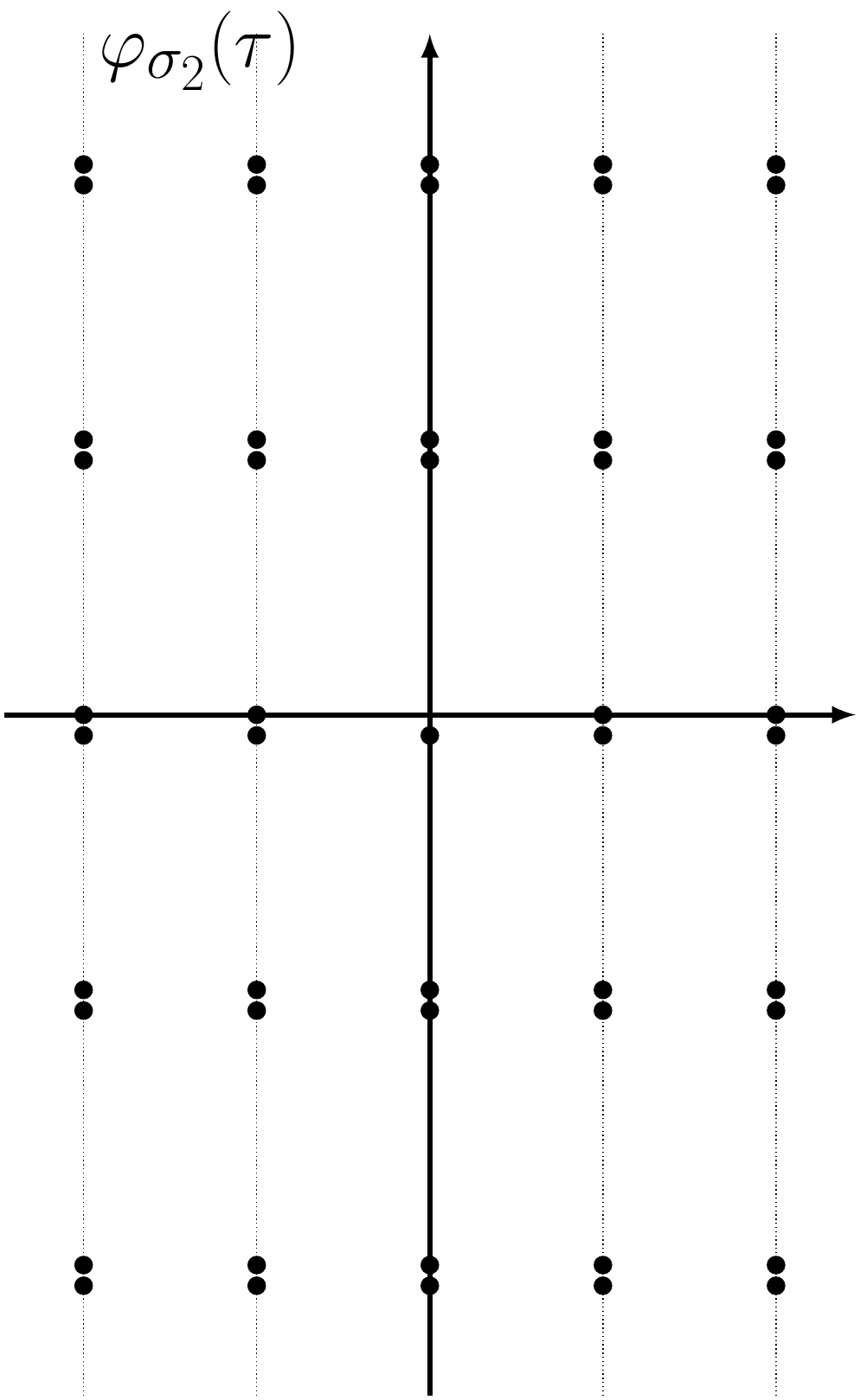}
\end{center}
\caption{The singularities in the Borel plane for the series
  $\varphi_{\sigma_{j}} (u;\tau)$ with $j=1,2$ of knot $\knot{4}_1$
  for $x=\re^{2\pi\bb u}$ in the first or the third interval of the
  positive axis.}
\label{fig:sings-l-41}
\end{figure}

\begin{figure}[htpb!]
\leavevmode
\begin{center}
\includegraphics[height=7cm]{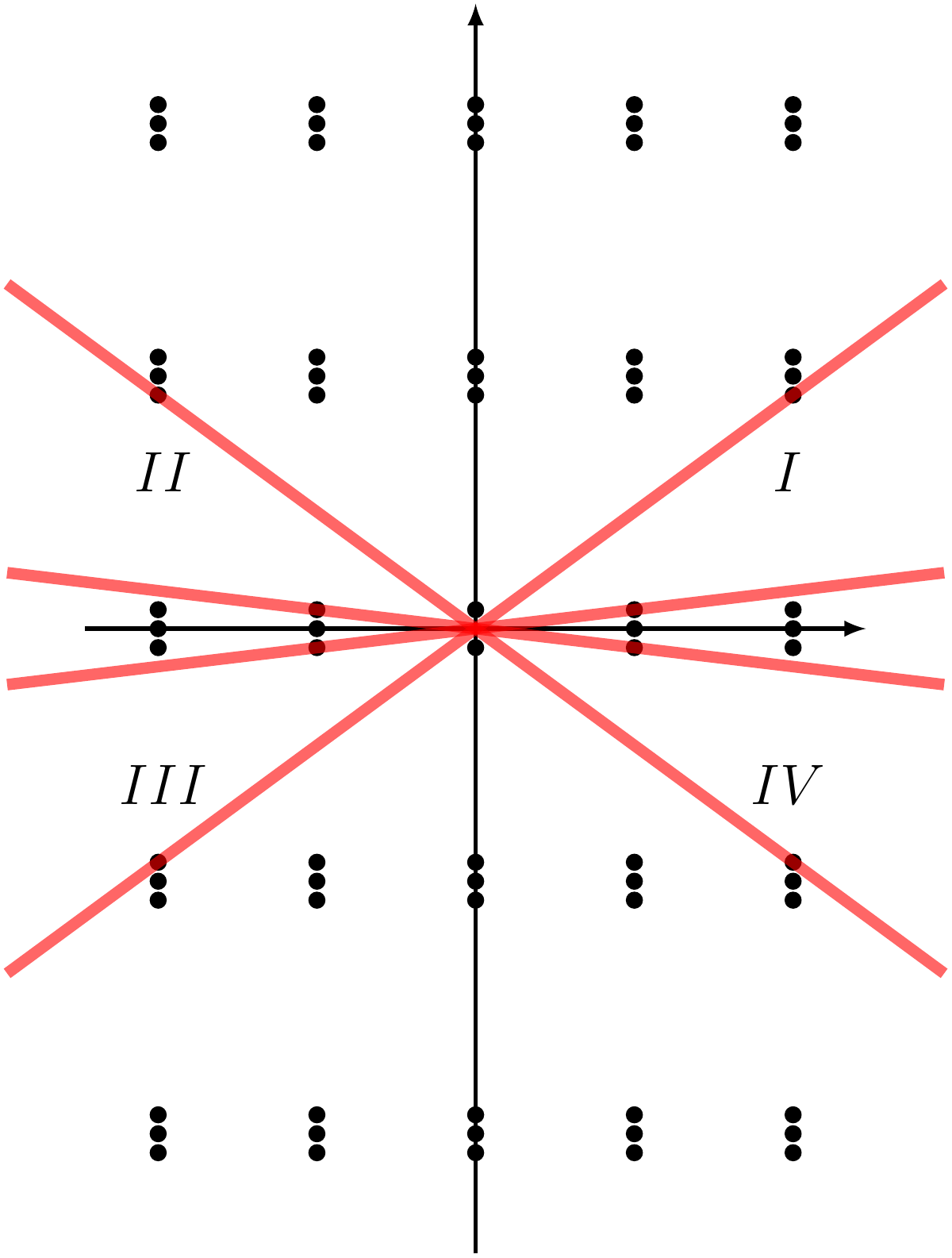}
\end{center}
\caption{Four different sectors in the $\tau$-plane for $\Phi(u;\tau)$
  of knot $\knot{4}_1$ with $x=\re^{2\pi\bb u}$ in the first or the
  third interval of the positive axis.}
\label{fig:secs-l-41}
\end{figure} 

First of all, we notice that the first and third intervals are related
by the reflection $x\mapsto x^{-1}$.  In fact, due to the property
\eqref{eq:-xphi} of the asymptotic series, the Borel plane
singularities for $\Phi_\s(x;\tau)$ and $\Phi_\s(x^{-1};\tau)$ are
identical, and we illustrate them uniformly in
Figure~\ref{fig:sings-l-41}. The positions of singularities are still
described by \eqref{iota}.  However, the difference of action
$(V(\s_1)-V(\s_2))/(2\pi\ri)$ is now imaginary and it describes the
vertical spacing between neighoring singularities.  The shortest
horizontal spacing is $\log x$. Finally all singularities are repeated
vertically by the spacing $2\pi\ri$.  Similar to the discussion in
Section~\ref{sub.41stokes}, we label in the $\tau$-plane by
$I,II,III,IV$ the four sectors which separate the 12 singularities
close to the real axis and other singularities along the imaginary
axis or away from the real axis, as in Figure~\ref{fig:secs-l-41}.  In
each of the four sectors, the Borel summed vector
$s_R^{\gtrless}(\Phi)(x;\tau)$ is a linear transformation of the Borel
summed vector $s_R(\Phi)(x;\tau)$ in the middle interval, as per the
Picard-Lefschetz theory
\begin{equation}
  s_R^{\gtrless}(\Phi)(x;\tau) = T_R^{\gtrless}(\tx)\cdot s_R(\Phi)(x;\tau).
\end{equation}
It is most convenient to compute the transformation matrix
$T_R^{\gtrless}$ by comparing the left hand side with the holomorphic
lifts of $s_R(\Phi)(x;\tau)$ summarized in
Corollary~\ref{cor.41borel}.  By doing so, we find in the first
interval
\begin{align}
  T^{<}_{I}(\tx) &=
    \begin{pmatrix}
      \tx&0\\-\tx&1
    \end{pmatrix}, &
         T^{<}_{II}(\tx) &=
           \begin{pmatrix}
             0&1\\-\tx&1
           \end{pmatrix}\\
  T^{<}_{III}(\tx) &=
    \begin{pmatrix}
      \tx&1\\-\tx&0
    \end{pmatrix}, &
         T^{<}_{IV}(\tx) &=
           \begin{pmatrix}
             \tx&1\\0&1
           \end{pmatrix},
\end{align}
while in the third interval
\begin{align}
  T^{>}_{I}(\tx) &=
    \begin{pmatrix}
      \tx^{-1}&0\\
      -\tx^{-1}&1
    \end{pmatrix},  &
              T^{>}_{II}(\tx) &=
                \begin{pmatrix}
                  0&1\\-\tx^{-1}&1
                \end{pmatrix}\\
  T^{>}_{III}(\tx) &=
    \begin{pmatrix}
      \tx^{-1}&1\\
      -\tx^{-1}&0
    \end{pmatrix}, &
              T^{>}_{IV}(\tx) &=
                \begin{pmatrix}
                  \tx^{-1}&1\\0&1
                \end{pmatrix}.
\end{align}
They are indeed related by
\begin{equation}
  T^{<}_R(\tx) = T^{>}_R(\tx^{-1}).
  \label{eq:Tgtrless}
\end{equation}

Once the linear combinations are known, the Stokes matrices can be
computed using the Stokes matrices in the middle interval given in
Section~\ref{sub.41stokes}
\begin{equation}
  \mfS^{\gtrless}_{R\rightarrow R'}(x;q) =
  T_{R'}^{\gtrless}(x)\cdot
  \mfS_{R\rightarrow R'}(x;q)\cdot \left(T_{R}^{\gtrless}(x)\right)^{-1}.
\end{equation}
We find
\begin{align}
  \mfS^>_{I\rightarrow II}(x;q) =
    &\begin{pmatrix}
      0&1\\
      -x^{-1}&1
    \end{pmatrix}\cdot
               \mfS_{I\rightarrow II}(x;q)\cdot
               \begin{pmatrix}
                 x&0\\1&1
               \end{pmatrix}\\
  \mfS^>_{II\rightarrow III}(x) =
    &\begin{pmatrix}
      1-x-x^2 && x+x^2\\-x-x^2 && 1+x+x^2
    \end{pmatrix}\\
  \mfS^>_{III\rightarrow IV}(x;q) =
    &\begin{pmatrix}
      x^{-1}&1\\-x^{-1}&0
    \end{pmatrix}\cdot \mfS_{III\rightarrow IV}(x;q)\cdot
    \begin{pmatrix}
      x&-x\\0&1
    \end{pmatrix}\\
  \mfS^>_{IV\rightarrow I}(x)=
    &\begin{pmatrix}
      1+x^{-1}+x^{-2} && x^{-1}+x^{-2}\\-x^{-1}-x^{-2} &&
      1-x^{-1}-x^{-2}
    \end{pmatrix},
\end{align}
and thanks to \eqref{eq:Tgtrless}, 
\begin{equation}
  \mfS^{<}_{R\rightarrow R'}(x;q) =
  \mfS^{>}_{R\rightarrow R'}(x^{-1};q).
\end{equation}

Note that $\mfS^>_{III\rightarrow I}(x)$ and
$\mfS^>_{II\rightarrow IV}(x)$ encode all the 12 singularities near the
real axis illustrated in Figure~\ref{fig:secs-l-41}, as well as the
Stokes constants associated to them.  Furthermore, the Stokes matrices
also have the property
\begin{equation}
  \mfS^>_{IV\rightarrow III}(x^{-1};q)^{-1} = \mfS^>_{I\rightarrow II}(x;q)^T,
\end{equation}
in accord with Conjecture~\ref{conj.ann}.

\subsection{Numerical verification}
\label{sub.41num}

In this section we explain the numerical verification of
Conjecture~\ref{conj.41M}. This involves, on the one hand, a numerical
computation of the asymptotics of the holomorphic blocks and on the
other hand, a numerical computation of the Borel resummation by the
Laplace integral of a Pad\'e approximation. Taking the two
computations into account, we found out numerically, integers
appearing at two exponentially small scales, namely $\tq$ and $\tx$,
and guessing these integers eventually led to
Conjecture~\ref{conj.41M}.

\begin{table}
  \centering%
  \subfloat[Region $I$: $\tau = \frac{1}{20}\re^{\frac{\pi\ri}{5}}$]
  {\begin{tabular}{*{5}{>{$}c<{$}}}\toprule
     &|\frac{s_I(\Phi)(x;\tau)}{P_I(x;\tau)}- 1|
     &|\frac{s_I(\Phi)(x;\tau)}{s_I'(\Phi)(x;\tau)}- 1|
     &|\tq(\tau)|& |\tx(x,\tau)|\\\midrule
     \s_1& 3.2\times 10^{-66} & 9.7\times10^{-66}
     & \multirow{2}{*}{$8.3\times 10^{-33}$}
      & \multirow{2}{*}{$0.05$}\\
     \s_2& 1.9\times 10^{-94} & 5.2\times 10^{-94}&&\\\bottomrule
   \end{tabular}}\\
 \subfloat[Region $II$: $\tau = \frac{1}{20}\re^{\frac{4\pi\ri}{5}}$]
 {\begin{tabular}{*{5}{>{$}c<{$}}}\toprule
    &|\frac{s_{II}(\Phi)(x;\tau)}{P_{II}(x;\tau)}- 1|
    &|\frac{s_{II}(\Phi)(x;\tau)}{s_{II}'(\Phi)(x;\tau)}- 1|
    &|\tq(\tau)|& |\tx(x,\tau)^{-1}|\\\midrule
    \s_1& 1.9\times 10^{-94} & 5.2\times 10^{-94}
    & \multirow{2}{*}{$8.3\times 10^{-33}$}
      & \multirow{2}{*}{$0.05$}\\
     \s_2& 3.2\times 10^{-66} & 9.7\times 10^{-66}&&\\\bottomrule
   \end{tabular}}\\
 \subfloat[Region $III$: $\tau = \frac{1}{20}\re^{-\frac{4\pi\ri}{5}}$]
 {\begin{tabular}{*{5}{>{$}c<{$}}}\toprule
    &|\frac{s_{IV}(\Phi)(x;\tau)}{P_{IV}(x;\tau)}- 1|
    &|\frac{s_{IV}(\Phi)(x;\tau)}{s_{IV}'(\Phi)(x;\tau)}- 1|
    & |\tq(\tau)^{-1}|& |\tx(x,\tau)|\\\midrule
    \s_1& 1.9\times 10^{-94} &5.2\times 10^{-94}
    & \multirow{2}{*}{$8.3\times 10^{-33}$}
      & \multirow{2}{*}{$0.05$}\\
     \s_2&3.2\times 10^{-66}  & 9.7\times 10^{-66}&&\\\bottomrule
  \end{tabular}}\\
 \subfloat[Region $IV$: $\tau = \frac{1}{20}\re^{-\frac{\pi\ri}{5}}$]
 {\begin{tabular}{*{5}{>{$}c<{$}}}\toprule
    &|\frac{s_{III}(\Phi)(x;\tau)}{P_{III}(x;\tau)}- 1|
    &|\frac{s_{III}(\Phi)(x;\tau)}{s_{III}'(\Phi)(x;\tau)}- 1|
    & |\tq(\tau)^{-1}|& |\tx(x,\tau)^{-1}|\\\midrule
    \s_1& 3.2\times 10^{-66} & 9.7\times 10^{-66}
    & \multirow{2}{*}{$8.3\times 10^{-33}$}
      & \multirow{2}{*}{$0.05$}\\
     \s_2& 1.9\times 10^{-94} & 5.2\times 10^{-94}&&\\\bottomrule
   \end{tabular}}
 \caption{Numerical tests of holomorphic lifts of Borel sums of
   asymptotic series for knot $\knot{4}_1$.  We perform the
   Borel-Pad\'e resummation on $\Phi(x;\tau)$ with 280 terms at
   $x=6/5$ and $\tau$ in four different regions, and compute the
   relative difference between them and the right hand side of
   \eqref{eq:HU1-41}--\eqref{eq:HU4-41}, which we denote by
   $P_R(x;\tau)$.  They are within the error margins of Borel-Pad\'e
   resummation, which are estimated by redo the resummation with 276
   terms, denoted by $s'_R(\bullet)$ in the tables. The relative
   errors are much smaller than $|\tq^{\pm 1}|, |\tx^{\pm 1}|$,
   possible sources of additional corrections.}
  \label{tab:num-41}
\end{table}

We found ample numerical evidence for the resurgent data
\eqref{eq:M1-41}--\eqref{eq:M4-41}.  First of all, due to the symmetry
$\Phi_{1}(x;-\tau) = -\ri \Phi_{2}(x;\tau)$,
$\Phi_{2}(x;-\tau) = \ri \Phi_{1}(x;\tau)$, the resurgent behavior of
$s(\Phi)(x;\tau)$ for $\tau$ in the lower half-plane can be deduced
from that for $\tau$ in the upper half-plane.  We only have to
numerically test the resurgent data for $\tau$ in regions $I,II$ in
the upper-half plane.

The first piece of evidence comes from analysing the radial
asymptotics of the left hand side of \eqref{abrelx}.  Note that the
matrix $\Delta(\tau) = \diag(\Delta_\s(\tau))$ is always diagonal, and
each row of \eqref{abrelx} is
\begin{equation}
  \Delta_{\s}(\tau)B^\s(x;q) =
  \sum_{\s'} M_R(\tx,\tq)_{\s,\s'}s_R(\Phi_{\s'})(x;\tau).
  \label{eq:abrelx-row}
\end{equation}
If we take $\tau = \re^{\ri\alpha}/k$ with the argument $\alpha$
depending on a ray in the region $R$ and $k$ a very large integer, the
difference between $\exp\left(\frac{V(\s')}{2\pi\ri\tau}\right)$
associated to different critical points is greatly magnified, and the
right hand side of \eqref{eq:abrelx-row} is dominated by a single
series.  Furthermore, when $\tau$ is in the upper (lower) half plane,
$\tq$ ($1/\tq$) is exponentially suppressed and the correction
$M_R(\tx,\tq)_{\s,\s'}$ as a series in $\tq$ ($1/\tq$) is dominated by
the leading term. \eqref{eq:abrelx-row} thus becomes
\begin{equation}
  \Delta_\s(\tau) B^\s(x;q) \sim
  \exp\left(\frac{V(\hat{\s}) + \omega_{\hat{\s}}}{2\pi\ri\re^{\ri\alpha}}k
    - \log (\ri \delta(x,y_{\hat{\s}}))
    + \sum_{n=1}^\infty S_n(x,y_{\hat{\s}})\re^{n \ri
      \alpha}k^{-n}\right),\quad k\gg 1,
\end{equation}
where $\omega_{\hat{\s}}$ is possible contribution from the leading
term of $M_R(\tx,\tq)_{\s,\hat{s}}$, and the series in $k^{-1}$ is
$\log\varphi(x,y_{\hat{\s}};\tau)$.  As pointed out in
\cite{GZ:qseries}, this equation can be tested numerically with the
help of Richardson transformations (see for instance
\cite{bender-orszag}).

Next, we can test \eqref{eq:abrelx-row} directly.  One way of doing
this is to compute Borel-Pad\'e resummation $s_R(\Phi_{\s'})(x;\tau)$
for various values of $x\in \IR$ and $\tau$ in the same region $R$,
and by comparing with the left hand side extract terms of
$M_R(\tx,\tq)_{\s,\s'}$ order by order.  To facilitate this operation,
instead of $M_R(\tx;\tq)$ we consider
\begin{equation}
  \wt{M}_R(\tx;\tq) =
  \begin{pmatrix}
    \theta(-\tq^{-1/2}\tx;\tq)^2&0\\
    0&\theta(-\tq^{1/2}\tx;\tq)^{-1}
  \end{pmatrix}M_R(\tx;\tq)
\end{equation}
whose entries are $\tq$-series with coefficients in $\IZ[\tx^{\pm 1}]$
instead of in $\IZ(\tx)$.
Using $280$ terms of $\Phi_{\s}(x;\tau)$, we find the results for
$\tau$ in the upper half plane
\begin{subequations}
\begin{align}
  &\wt{M}_I(x;q) =
    \begin{pmatrix}
      -x + (x^2+x^3)q + (x^2+x^3)q^2 & 1-(x+x^2+x^3)q-(x+x^2)q^2\\
      -1+(x^{-1}+x)q+(x^{-1}+x)q^2 & 1-(x^{-1}+1+x)q-(x^{-1}+x)q^2
    \end{pmatrix}+O(q^3)\\
  &\wt{M}_{II}(x;q) =
    \begin{pmatrix}
      -x+(1+x^{-1}+x^{-2})q+(1+x^{-1})q^2 &
      1-(x^{-1}+x^{-2})q-(x^{-1}+x^{-2})q^2\\
      -1+(x+1+x^{-1})q+(x+x^{-1})q^2&
      1-(x+x^{-1})q-(x+x^{-1})q^2
    \end{pmatrix}+O(q^3)
\end{align}
\end{subequations}
and they agree with \eqref{eq:M1-41}, \eqref{eq:M2-41}.  Another more
decisive way is to compute both sides of \eqref{eq:abrelx-row}
numerically assuming \eqref{eq:M1-41}, \eqref{eq:M2-41} and compare
them.  Alternatively, we can compare two sides of the equations of
holomoprhic lift \eqref{eq:HU1-41}, \eqref{eq:HU2-41}.  We find that
the difference between the two sides is always within the error margin
of Borel-Pa\'e resummation, and much smaller than
$\tq^{\pm 1},\tx^{\pm 1}$, i.e.~possible additional corrections. We
illustrate this comparison by one example with $x=6/5$ and
$\tau = \frac{1}{20}\re^{\pm\frac{\pi\ri}{5}},
\frac{1}{20}\re^{\pm\frac{4\pi\ri}{5}}$ in four regions in
Table~\ref{tab:num-41}.

A final way to test these results is to see that in the $x\mapsto 1$
limit, the resurgent data as well as the Stokes
matrices~\eqref{eq:Sp-41}--\eqref{eq:Sn-41} are compatible with the
results in section~\ref{sub.41x=1} where $x=1$. This is a non-trivial
test since the matrix $W_{-1}(x^{-1};q^{-1})$ ($|q|<1$)
in~\eqref{eq:Sp-41}--\eqref{eq:Sn-41} is divergent in the limit $x\mapsto 1$.


\section{The $\knot{5}_2$ knot}
\label{sec.52}

\subsection{Asymptotic series}
\label{sub.52asy}

Our second example that we discuss in detail will be the case of the
$\knot{5}_2$ knot. The state-integral for the $\knot{5}_2$
knot~\cite[Eqn.(39)]{AK} (after removing a prefactor that depends on $u$ alone) 
\begin{equation}
  \label{Z52x}
  Z_{\knot{5}_2}(u;\tau) =
  \int_{\BR+\ri 0}  \Phi_\bb(v) \, \Phi_\bb(v+u)
  \, \Phi_\bb(v-u) \, \re^{-2\pi \ri v^2} \rd v \,.
\end{equation}
After a change of variables $u \mapsto u/(2\pi \bb)$ (see
Equation~\eqref{ub}) and $v \mapsto v/(2\pi\bb)$, it follows
that the integrand of $Z_{\knot{5}_2}(\ub;\tau)$ has a leading term
given by $\re^{V(u,v)/(2 \pi \ri \tau)}$ where
\begin{equation}
  \label{V52}
  V(u,v) = \Li_2(-\re^{v}) + \Li_2(-\re^{u+v}) + \Li_2(-\re^{-u+v})
  + (v)^2 \,.
\end{equation}
Taking derivative with respect to $v$ gives the equation for the
critical point
\begin{equation}
  \label{52crit}
  2 v - \log(1+\re^{v}) -\log(1+\re^{u+v}) -\log(1+\re^{-u+v})=0
\end{equation}
which implies that $x=\re^{2\pi \bb u}$ and $y=-\re^{2\pi \bb v}$ are
points of the affine curve $S$ given by
\begin{equation}
  \label{52critb}
  S : y^2 = (1-y)(1- x y)(1-x^{-1}y) 
\end{equation}
and $(u, v)$ are points of the exponentiated curve $S^*$ given
by the above equation with $(x,y)=(\re^{u},-\re^{v})$.
Moreover, 
\begin{equation}
  \label{V52b}
  V(u,v) = \Li_2(y)+\Li_2(xy) +\Li_2(x^{-1}y) + (\log (-y))^2 
\end{equation}
is a holomorphic $\BC/2\pi^2\BZ$-valued function on the exponentiated
curve $S^*$. Note that when $u=0$, Equation~\eqref{52critb}
becomes~\eqref{52xi}.

The constant term of the asymptotic expansion is given by the Hessian
of $V(u,v)$ at a critical point $(u,v)$, and it is a rational
function of $x$ and $y$ is given by 
\begin{equation}
  \label{52delta}
  \delta(x,y) = y-(1+x+x^{-1})y^{-1}+2y^{-2} \, .
\end{equation}
Note that $\delta(x,y)=0$ on $S$ if and only if $x$ is a root of the
discriminant of $S$ with respect to $y$, i.e., 
\begin{equation}
\label{disc52y}
1 - 6 x + 11 x^2 - 12 x^3 - 11 x^4 - 12 x^5 + 11 x^6 - 6 x^7 + x^8 =
0 \,.
\end{equation}
This happens at two points in the real line given approximately by
$ x \approx 0.235344$ and $x \approx 4.24909$. Moreover, when $x$
is a root of~\eqref{disc52y}, exactly two out of the three branches
of $y=y(x)$ collide, and the corresponding branch point is simple.

Beyond the leading asymptotic expansion and its constant term, the
asymptotic series has the form $\Phi(x,y;\tau)$ where 
\begin{equation}
  \label{Phi52}
  \Phi(x,y;\tau)=\exp\left( { V(u,v) \over 2\pi\ri \tau} \right)
  \varphi(x,y;\tau),
  \qquad 
  \varphi(x,y;\tau) \in \frac{1}{\sqrt{\ri\delta}}
  \BQ[x^\pm,y^\pm,\delta^{-1}] [[2 \pi \ri \tau]]
\end{equation}
where $\delta$ is given in~\eqref{52delta} and
$\sqrt{\ri \delta} \, \varphi(x,y;0)=1$.  In other words, the
coefficient of every power of $2\pi \ri \tau$ in
$\sqrt{\delta} \, \varphi(x,y;\tau)$ is a rational function on
$S$. There is a natural projection $S \to \BC^*$ given by
$(x,y) \mapsto x$ and we denote by $y_\s(x)$ the choice of a local
section (an algebraic function of $x$), for
$\s \in \calP=\{\s_1,\s_2,\s_3\}$.  We denote the corresponding series
$\varphi(x,y_\s(x);\tau)$ simply by $\varphi_\s(x;\tau)$.  When $x$ is
close to 1, we order $\mc{P}$ so that $\s_1,\s_2,\s_3$ correspond to
small deformations away from geometric, conjugate, and real
connections at $x=1$.  Note for $\s_3$, we only keep the real part of
$V$. The power series $\sqrt{\ri\delta}\varphi_\s(x;\tau)$ can be
computed by applying Gaussian expansion on the state-integral
\eqref{Z52x}, and one can compute up to 15 terms in a few minutes.
Let us write down the first few terms of
$\varphi_\s(x;\tau)$
\begin{align}
  &\varphi_\s(x;\frac{\tau}{2\pi\ri})=
    1  +\frac{\tau}{24\delta_\s^3}\Big(
    81+112s-78s^2-70s^3+94s^4-38s^5+5s^6\nn
  &\phantom{========}+\big(138-254s+127s^2+44s^3-89s^4+38s^5-5s^6
    \big)y_\s\nn
  &\phantom{========}+\big(
    135-101s-11s^2+61s^3-33s^4+5s^5
    \big)y_\s^2 \Big)
    +O(\tau^2)\,,
    \label{eq:vf-52}
\end{align}
where
\begin{equation}
  s = s(x) = x^{-1}+1 +x\,.
\end{equation}
On the other hand, if one sets $x$ to numerical values, the power
series can be computed to 200 terms.

\subsection{Holomorphic blocks}
\label{sub.52holo}

Motivated by the case of the $\knot{4}_1$ knot, we define the descendant
state-integral of the $\knot{5}_2$ knot by 
\begin{equation}
  \label{Z52x-desc}
  Z_{\knot{5}_2,m,\mu}(u;\tau) =
  \int_{\mc{D}}  \Phi_\bb(v) \, \Phi_\bb(v+u)
  \, \Phi_\bb(v-u) \, \re^{-2\pi \ri v^2+ 2\pi(m \bb - \mu
    \bb^{-1})v} \rd v 
\end{equation}
for integers $m$ and $\mu$, which agrees with the
Andersen-Kashaev invariant of the $\knot{5}_2$ knot when
$m=\mu=0$. Here the contour
$\mc{D}$ was introduced in \eqref{Z410-desc}. It is expressed in terms of three descendant
holomorphic blocks, which we denote by $A_m$, $B_m$ and $C_m$ instead
of $B^{\s_j}_m$ for $j=1,2,3$.  For $|q| \neq 1$, $A_m(x;q)$,
$B_m(x;q)$ and $C_m(x;q)$ are given by
\begin{subequations}
\begin{align}
  \label{Am}
  A_m(x;q) &=  H(x,x^{-1},q^m;q) \\
  \label{Bm}
  B_m(x;q) &= \th(-q^{1/2}x;q)^{-2}x^m H(x,x^{2},q^m x^2;q) \\
  \label{Cm}
  C_m(x;q) &= \th(-q^{-1/2}x;q)^{-2} x^{-m} H(x^{-1},x^{-2},q^m x^{-2};q)
\end{align}
\end{subequations}
where $H(x,y,z;q^\ve):=H^\ve(x,y,z;q)$ for $|q|<1$ and $\ve=\pm$ and
\begin{subequations}
  \begin{align}
\label{H52p}
H^+(x,y,z;q) & = (qx;q)_\infty (qy;q)_\infty \sum_{n=0}^\infty
\frac{q^{n(n+1)} z^n}{(q;q)_n (qx;q)_n (qy;q)_n} \\
\label{H52n}
H^-(x,y,z;q) & = \frac{1}{(x;q)_\infty (y;q)_\infty}
\sum_{n=0}^\infty (-1)^n \frac{q^{\frac{1}{2}n(n+1)}
  x^{-n} y^{-n}z^n}{ (q;q)_n (qx^{-1};q)_n (qy^{-1};q)_n}
\,.
\end{align}
\end{subequations}
Note that the summand of $H^+$ (a proper
$q$-hypergeometric function) equals to that of $H^-$ after replacing
$q$ by $q^{-1}$. This implies that $H^\pm$ have a common annihilating
ideal $\calI_H$ with respect to $x,y,z$ which can be computed as in
the case of Lemma~\ref{lem.annJ}.

The next theorem expresses the descendant state-integrals bilinearly
in terms of descendant holomorphic blocks.

\begin{theorem}
  \label{thm.52a}
  \rm{(a)} The descendant state-integral can be expressed in terms of
  the descendant holomorphic blocks by
  \begin{align}
    Z_{\knot{5}_2,m,\mu}(\ub;\tau) 
    =
    &(-1)^{m+\mu}
      q^{m/2}\tq^{\mu/2}\Big(
      \re^{\frac{3\pi\ri}{4}+\frac{5\pi\ri}{12}(\tau+\tau^{-1})}
      A_m(x;q) A_{-\mu}(\tx;\tq^{-1}) \nn + 
    &\re^{-\frac{\pi\ri}{4}+\frac{\pi\ri}{12}(\tau+\tau^{-1})}
      B_{m}(x;q) B_{-\mu}(\tx;\tq^{-1})
      + \re^{-\frac{\pi\ri}{4}+\frac{\pi\ri}{12}(\tau+\tau^{-1})}
      C_{m}(x;q) C_{-\mu}(\tx;\tq^{-1})\Big)\,.
    \label{52x-desc-fac}
  \end{align}
  \rm{(b)} The functions $A_m(x;q)$, $B_m(x;q)$ and $C_m(x;q)$ are
  holomorphic functions of $|q| \neq 1$ and meromorphic functions of
  $x \in \BC^*$ with poles in $x \in q^\BZ$ of order at most $2$.
  \newline
  \rm{(c)} Let 
  \begin{equation}
    W_m(x;q) = \begin{pmatrix}
      A_{m}(x;q) & B_{m}(x;q) & C_{m}(x;q) \\
      A_{m+1}(x;q) & B_{m+1}(x;q) & C_{m+1}(x;q) \\
      A_{m+2}(x;q) & B_{m+2}(x;q) & C_{m+2}(x;q) 
    \end{pmatrix} \qquad (|q| \neq 1) \,.
    \label{eq:Wm52}
  \end{equation}
  For all integers $m$ and $\mu$, state-integral $Z_{\knot{5}_2,m,\mu}(u;\tau)$
  and the matrix-valued function
  \begin{equation}
    W_{m,\mu}(u;\tau) = W_{-\mu}(\tx;\tq^{-1}) \Delta(\tau) 
    W_m(x;q)^T \,,
    \label{W52x-desc}
  \end{equation}
  where 
  \begin{equation}
    \label{D52}
    \Delta(\tau) = \begin{pmatrix}
      \re^{\frac{3\pi\ri}{4}+\frac{5\pi\ri}{12}(\tau+\tau^{-1})}&0&0\\
      0&\re^{-\frac{\pi\ri}{4}+\frac{\pi\ri}{12}(\tau+\tau^{-1})}&0\\
      0&0&\re^{-\frac{\pi\ri}{4}+\frac{\pi\ri}{12}(\tau+\tau^{-1})}\\
    \end{pmatrix} \,,
  \end{equation}
  are holomorphic functions of $\tau \in \BC'$ and entire
  functions of $u \in \BC$.
\end{theorem}

\begin{proof}
  Part (a) follows by applying the residue theorem, just as in the
  proof of part (a) of Theorem~\ref{thm.41a}. A similar result was stated
  in \cite{dimofte-state}.

  Part (b) follows from the fact that when $|q|<1$, the ratio test
  implies that $H^+(x,y,z;q)$ is an entire function of
  $(x,y,z) \in \BC^3$ and $J^-(x,y,z;q)$ is a meromorphic function of
  $(x,y,z) \in \BC^2 \times \BC$ with poles at $x=q^\BZ$ and
  $y \in q^\BZ$.

  For part (c), one uses parts (a) and (b) to deduce that $W_{m,\mu}(u;\tau)$
  is holomorphic of $\tau\in \IC'$ and meromorphic in $u$ with possible
  poles of second order at $\ri \bb \BZ + \ri \bb^{-1} \BZ$.
    An expansion at these points, done by the method of
  Section~\ref{sub.52near1}, demonstrates that the function is analytic
  at the points $\ri \bb \BZ + \ri \bb^{-1} \BZ$.
\end{proof}

Note that the holomorphic blocks have the symmetry
\begin{equation}
  A_m(x^{-1};q) = A_m(x;q),\quad B_m(x^{-1};q) = C_m(x;q),
\end{equation}
which implies the symmetry of the matrix $W_m(x;q)$
\begin{equation}
  W_m(x^{-1};q) = W_m(x;q)
  \begin{pmatrix}
    1&0&0\\0&0&1\\0&1&0
  \end{pmatrix} \,.
  \label{eq:-xWm}
\end{equation}
Consequently $W_{m,\mu}(u;\tau)$ is invariant under the
reflection $u\mapsto -u$.

\begin{lemma}  
  \label{lem.annH}
  \rm{(a)}
  The annihilating ideal of $\calI_H$ of $H^\pm$ is given by
  \begin{align}
  \label{annH}
  \calI_J =& \langle
  y z S_x - x z S_y + (x^2 y - x y^2) S_z + (-x^2 y + x y^2 + x z - y z),
  \\ \notag & -x y^2 S_z^2 + z S_y + (y^2 + x y^2 - q y z) S_z + (-y^2 - z),
  -q y S_y S_z + S_y - 1,
  \\ \notag & z S_y^2  + (-x + q y + 
   q x y - q^2 y^2 - z - q z) S_y + q x y S_z + (x - q y - q x y + q z)
  \rangle
\end{align}
where $S_x$, $S_y$ and $S_z$ are the shifts $x$ to $qx$, $y$ to $qy$
and $z$ to $qz$, respectively.
\newline
\rm{(b)} When $|q|<1$, we have
  \begin{equation}
    H(x,y,z;q^{-1}) = \frac{
      \det
      \begin{pmatrix}
        H(x^{-1},x^{-1}y,x^{-2}z;q) &H(y^{-1},y^{-1}x,y^{-2}z;q)\\
        x H(x^{-1},x^{-1}y,q^{-1}x^{-2}z;q) & y
        H(y^{-1},y^{-1}x,q^{-1}y^{-2}z;q)
      \end{pmatrix}
    }{y\th(-q^{-\frac{1}{2}}x;q)\th(-q^{-\frac{1}{2}}y;q)\th(-q^{-\frac{1}{2}}xy^{-1};q)}.
    \label{eq:H-con}
  \end{equation}
\end{lemma}

\begin{proof}
Part (a) follows as in the proof of Lemma~\ref{lem.annJ}.

For part (b), observe that both sides of the equation are 
power series in $z$ and $q$-holonomic functions of $z$. Using the
\texttt{HF} package, we find that the $(i,j)$-entry
of the determinant is annihilated by the operator $r_{ij}$ given by
\begin{align*}
  r_{11} &=
  -y S_z^3 + (x + y + x y - q^2 z) S_z^2  + (-x - x^2 - x y) S_z + x^2
\\
r_{12} &=
-x S_z^3 + (x + y + x y - q^2 z) S_z^2 + (-y - x y - y^2) S_z + y^2
\\
r_{21} &=
-y S_z^3 + (x + y + x y - q z)  S_z^2 + (-x - x^2 - x y) S_z + x^2
\\
r_{22} &=
-x S_z^3 + (x + y + x y - q z) S_z^2 + (-y - x y - y^2) S_z + y^2 \,,
 \end{align*}
 whereas the left hand side of~\eqref{eq:H-con}, after being multiplied
 by the denominator of the right hand side, is annihilated by the
 operator
 \begin{equation}
 r= S_z^3 + (-1 - x - y)  S_z^2 + (x + y + x y - q z) S_z - x y \,.
 \end{equation}
 Using the commands \texttt{DFiniteTimes} and \texttt{DFiniteTimes},
 we computed a 9th order operator $R$ (which is too long to type here)
 that annihilates the determinant, and using the command
 \texttt{OreReduce}, we proved that it is a left multiple of $r$. It
 follows that both sides of~\eqref{eq:H-con} satisfy the same 9th
 order recursion with respect to $z$, with nonvanishing leading
 term. Thus, the identity follows once we prove that the coefficient
 of $z^k$ in both sides agree, for $k=0,\dots,8$. When $|q|<1$, the
 coefficient of $z^k$ in $H(x,y,z;q)$ (resp., $H(x,y,z;q^{-1})$) is in
 $(qx;q)_\infty(qy;q)_\infty \BQ(x,y,q)$ (resp.,
 $(x;q)_\infty^{-1}(y;q)_\infty^{-1} \BQ(x,y,q)$), and this implies
 that the equality of the coefficient of $z^k$ in the above identity
 reduces to an equality on the field $\BQ(x,y,q)$ of rational
 functions in three variables. The latter is easy to check for
 $k=0,\dots,8$. This completes the proof of~\eqref{eq:H-con}.
\end{proof}

The next theorem concerns the properties of the linear $q$-difference
equations satisfied by the descendant holomorphic blocks.

\begin{theorem}
  \label{thm.52b}
  \rm{(a)} They are $q$-holonomic functions in the variables $(m,x)$
  with a common annihilating ideal
  \begin{equation}
    \label{ann52}
    \calI_{\knot{5}_2} = \langle P_1, P_2, P_3 \rangle
  \end{equation}
  where
  {\tiny
    \begin{subequations}
      \begin{align}
        \label{52P1}
        P_1 =
        & x (1 - q^3 x^2)
          (1 - q x^2 - q^2 x^2 - q^{3 + m} x^3 + q^3 x^4)
          -(1 - q x) (1 + q x) (1 - q x^2) (1 - q^3 x^2)
          S_m \\ \notag
        & -x(1-qx)(1+qx)
          (1-qx-qx^2-q^3x^2+q^2x^3+q^4x^3-q^{3+m}x^3-q^{4+m}x^3
          +q^4x^4-q^5x^5 ) S_x\\ \notag
        & - q^{4+m} x^4(1 - q x^2) S_x^2 
        \\
      \label{52P2}
        P_2 =
        & x  -  S_m - x S_x + q x^2 S_x S_m 
        \\
      \label{52P3}
        P_3 =
        & x (1 -q^{1 + m} - q x^2)  - 
          (1 - q^{1 + m} +  x - q x^2 + 
          q^{2 + m} x^2 - q x^3) S_m \\ \notag
        &+ (1 - q x^2) S_m^2 +  q^{1+m} x S_x \,.
      \end{align}    
    \end{subequations}
  }
  $\calI_{\knot{5}_2}$ has rank 3 and the three functions form a basis of
  solutions of the corresponding system of linear equations.  \newline
  \rm{(b)} As functions of an integer $m$, $A_m(x;q)$, $B_m(x;q)$ and
  $C_m(x;q)$ form a basis of solutions of the linear $q$-difference
  equation $\Bhat_{\knot{5}_2}(S_m,x,q^m,q) f_m(x;q) = 0$ for $|q| \neq 1$ where 
  \begin{equation}
    \label{rec52m}
    \Bhat_{\knot{5}_2}(S_m,x,q^m,q) = (1-S_m)(1-xS_m)(1-x^{-1}S_m) - q^{2+m} S_m^2 \,.
  \end{equation}
  \rm{(c)} The Wronskian $W_m(x;q)$ of~\eqref{rec52m}, 
  defined in \eqref{eq:Wm52}, satisfies 
  \begin{equation}
    \label{W520}
    \det(W_m(x;q)) =
    -\th(-q^{-\frac{1}{2}}x; q)^{-2} \th(-q^{-\frac{1}{2}}x^2; q)
     \qquad
    (|q| \neq 1) \,.
  \end{equation}
  \rm{(d)} The Wronskian satisfies the orthogonality relation
  \begin{equation}
    \label{WW52b}
    W_{-1}(x;q) \, W_{-1}(x;q^{-1})^T =
    \begin{pmatrix}
      1 & 0 & 0\\
      0 & 0 & 1\\
      0 & 1 & x+x^{-1}
    \end{pmatrix} \,.
  \end{equation}
  It follows that for all integers $m,\ell$ 
  \begin{equation}
    \label{WW52}
    W_m(x;q) \,  W_\ell(x;q^{-1})^T \in
    \PSL(3,\BZ[q^\pm, x^\pm])
  \end{equation}
  \rm{(e)} As functions of $x$, they form a basis of a linear
  $q$-difference equation
  $\hat A_{\knot{5}_2}(S_x,x,q^m,q) f_m(x;q) =0$
where 
  \begin{equation}
    \label{rec52x}
    \hat A_{\knot{5}_2} (S_x,x,q^m,q) = \sum_{j=0}^3 C_j(x,q^m,q) S_x^j \,,
  \end{equation}
  $S_x$ is the operator that shifts $x$ to $qx$ and
  {\tiny
    \begin{subequations}
      \begin{align}        
        C_0 =
        & -q^{2 + m} x^2 (1 - q^2 x) (1 + q^2 x) (1 - q^5 x^2)
          \label{rec52x0}\\
        C_1 =
        &(1 - q x) (1 + q x) (1 - q^5 x^2) \times
        \nn
        & (1 - q x - q x^2 - q^4 x^2 + q^{2 + m} x^2 + q^{3 + m} x^2 +
          q^2 x^3
          + q^5 x^3 + q^5 x^4 + q^{5 + m} x^4 - q^6 x^5)
          \label{rec52x1}\\
        C_2 =
        & qx(1-q^2x)(1+q^2x)(1-qx^2)\times \nn
        &(1-q^2x-q^{2+m}x-q^2x^2-q^5x^2+q^4x^3+q^7x^3
          -q^{5+m}x^3-q^{6+m}x^3+q^7x^4-q^9x^5)
          \label{rec52x2}\\
        C_3 =
        & q^{8+m} x^4 (1 - q x) (1 + q x) (1 - q x^2)
          \label{rec52x3}
      \end{align}
    \end{subequations}
  }
  \rm{(f)} The Wronskian of Equation~\eqref{rec52x}
  \begin{equation}
    \label{W52x}
    \calW_m(x;q) = \begin{pmatrix}
      A_{m}(x;q) & B_{m}(x;q) & C_{m}(x;q) \\
      A_{m}(qx;q) & B_{m}(qx;q) & C_{m}(qx;q) \\
      A_{m}(q^2x;q) & B_{m}(q^2x;q) & C_{m}(q^2x;q) 
    \end{pmatrix}, \qquad (|q| \neq 1)
  \end{equation}
  satisfies 
  \begin{equation}
    \det \calW_m(x;q) = q^{-5-2m}x^{-5}
    (1-q^2x^2)(1-qx^2)(1-q^3x^2)
    \th(-q^{-\frac{1}{2}}x;q)^{-2}\th(-q^{-\frac{1}{2}}x^2;q).
    \label{detW52x}
  \end{equation}
\end{theorem}

\begin{proof} Since $A_m(x;q)$, $B_m(x;q)$, and
  $C_m(x;q)$ are given in terms of $q$-proper hypergeometric
  multisums, it follows from the fundamental theorem of
  Zeilberger~\cite{Zeil:holo,WZ,PWZ} (see also~\cite{GL:survey}) that
  they are $q$-holonomic functions in both variables $m$ and $x$.
  Part (a) follows from an application of the
  \texttt{HF} package of Koutschan~\cite{Koutschan,
    Koutschan:holofunctions}.

  Part (b) follows from the \texttt{HF} package. The
  fact that they are a basis follows from (c).

  For part (c), Equation~\eqref{rec52m} implies that the determinant of
  the Wronskian satisfies the first order
  equation $\det(W_{m+1}(x;q))=\det(W_{m}(x;q))$
  (see~\cite[Lem.4.7]{GK:74}). It follows that
  $\det(W_m(x;q))=\det(W_0(x;q))$ with initial condition a
  function of $x$ given by
  \begin{equation}
    \label{W5200}
    \det(W_0(x;q)) = -\th(-q^{-\frac{1}{2}}x;q)^{-2}
    \th(-q^{-\frac{1}{2}}x^2;q)
    \qquad
    (|q| \neq 1) \,,
  \end{equation}
  which can be proved in a manner similar to section~\ref{sub.41holo}.
  Using Lemma~\ref{lem.annH} and the \texttt{HF} package we find the
  following recursion relation for the $q$-function $H(x,y,z;q)$ when
  $|q|<1$
  \begin{equation}
    x y H(x,y,q z;q) - (x+y+x y-z)H(x,y,z;q)
    +(1+x+y)H(x,y,q^{-1}z;q) + H(x,y,q^{-2}z;q) = 0.
    \label{eq:H-rec1}
  \end{equation}
  It then follows that
  \begin{align}
    &\cH_{\mu,\nu,1}(z;q) = H(q^{\mu},q^{\nu},z;q),\\
    &\cH_{\mu,\nu,2}(z;q) = z^{-\mu}H(q^{-\mu},q^{\nu-\mu},q^{-2\mu}z;q),\\
    &\cH_{\mu,\nu,3}(z;q) = z^{-\nu}H(q^{-\nu},q^{\mu-\nu},q^{-2\nu}z;q),
  \end{align}
  are three independent solutions to
  \begin{equation}
    q^{\mu+\nu} H(q z;q) - (q^\mu+q^\nu+q^{\mu+\nu}-z)H(z;q)
    +(1+q^\mu+q^\nu)H(q^{-1}z;q) + H(q^{-2}z;q) = 0.
  \end{equation}
  The corresponding Wronskian
  \begin{equation}
    \cW_{\mu,\nu}(z;q) =
    \begin{pmatrix}
      \cH_{\mu,\nu,1}(q^{-1}z;q) & \cH_{\mu,\nu,2}(q^{-1}z;q)
      & \cH_{\mu,\nu,3}(q^{-1}z;q)\\
      \cH_{\mu,\nu,1}(z;q) & \cH_{\mu,\nu,2}(z;q)
      & \cH_{\mu,\nu,3}(z;q)\\
      \cH_{\mu,\nu,1}(q z;q) & \cH_{\mu,\nu,2}(q z;q)
      & \cH_{\mu,\nu,3}(q z;q)\\
    \end{pmatrix}
  \end{equation}
  satisfies the recursion relation (see~\cite[Lem.4.7]{GK:74})
  \begin{equation}
    \det \cW_{\mu,\nu}(z;q) = q^{-\mu-\nu} \det\cW_{\mu,\nu}(q^{-1}z;q)
  \end{equation}
  which implies that the determinant of
  $U(z;q) = z^{\mu+\nu} \cW_{\mu,\nu}(z;q)$ is invariant under the shift
  $z\mapsto q z$.  We can thus identify it with the limit
  $z\mapsto 0$, which is easy to compute.
  Since
  \begin{equation}
    \lim_{z\mapsto 0} H(x,y,z;q) = (qx;q)_\infty(qy;q)_\infty,
  \end{equation}
  we have
  \begin{align}
    &\lim_{z\mapsto 0} \det U(z;q) \nn=
    &\lim_{z\mapsto 0}
      \det \begin{pmatrix}
        H(q^{\mu},q^{\nu},q^{-1}z;q) &
        q^{\mu}H(q^{-\mu},q^{\nu-\mu},q^{-2\mu}q^{-1}z;q) &
        q^{\nu}H(q^{-\nu},q^{\mu-\nu},q^{-2\nu}q z^{-1};q)\\
        H(q^{\mu},q^{\nu},z;q) &
        H(q^{-\mu},q^{\nu-\mu},q^{-2\mu}z;q) &
        H(q^{-\nu},q^{\mu-\nu},q^{-2\nu}z;q)\\
        H(q^{\mu},\re^{\nu},q z;q) &
        q^{-\mu}H(q^{-\mu},q^{\nu-\mu},q^{-2\mu}q z;q) &
        q^{-\nu}H(q^{-\nu},q^{\mu-\nu},q^{-2\nu}q z;q)
    \end{pmatrix}\nn=
    &(q q^{\mu};q)_\infty(q q^{\nu};q)_\infty
     (q q^{-\mu};q)_\infty(q q^{\nu-\mu};q)_\infty
     (q q^{-\nu};q)_\infty(q q^{\mu-\nu};q)_\infty
      \det
      \begin{pmatrix}
        1&q^\mu&q^\nu\\
        1&1&1\\
        1&q^{-\mu}&q^{-\nu}\\
      \end{pmatrix}\nn=
    &q^{-\mu-\nu}(1-q^\mu)(1-q^\nu)(q^\mu-q^\nu)\nn
    &\times(q q^{\mu};q)_\infty(q q^{\nu};q)_\infty
     (q q^{-\mu};q)_\infty(q q^{\nu-\mu};q)_\infty
     (q q^{-\nu};q)_\infty(q q^{\mu-\nu};q)_\infty.
    \label{eq:detU52}
  \end{align}
  We thus have
  \begin{equation}
    \label{zcw52}
    z^{\mu+\nu} \det \cW_{\mu,\nu}(z;q) =
    -(q^\mu;q)_\infty(q q^{-\mu};q)_\infty
    (q q^\nu;q)_\infty(q^{-\nu};q)_\infty
    (q q^{\mu-\nu};q)_\infty (q^{\nu-\mu};q)_\infty.
  \end{equation}
  Using the substitution
  \begin{equation}
    q^\mu = x^{-1},\quad q^{\nu} = x,\quad z =1
  \end{equation}
  in the above equation and cancelling with the $\theta$-prefactors of
  $B_m(x;q)$ and $C_m(x;q)$ we obtain Equation~\eqref{W5200} for
  $|q|<1$.  The case for $|q|>1$ can be obtained by 
  analytic continuation on both sides of \eqref{W5200}.

  For part (d), Equation~\eqref{WW52b} follows from~\eqref{eq:H-con}.
  To see this, let us introduce
  \begin{equation}
    \wt{W}_{m}(x;q) =
    \begin{pmatrix}
      H(x,x^{-1},q^m;q)&x^mH(x,x^2,q^mx^2;q)&x^{-m}H(x^{-1},x^{-2},q^m
      x^{-2};q)\\
      H(x,x^{-1},q^{m+1};q)
      &x^{m+1}H(x,x^2,q^mx^2;q)&x^{-m-1}H(x^{-1},x^{-2},q^{m+1}
      x^{-2};q)\\
      H(x,x^{-1},q^{m+2};q)
      &x^{m+2}H(x,x^2,q^{m+2}x^2;q)&x^{-m-2}H(x^{-1},x^{-2},q^{m+2}
      x^{-2};q)
    \end{pmatrix}.
  \end{equation}
  Then Equation~\eqref{WW52b} can equally be written as
  \begin{equation}
    \wt{W}_{-1}(x;q^{-1}) =
    \begin{pmatrix}
      1&0&0\\
      0&0&1\\
      0&1&x+x^{-1}
    \end{pmatrix}\left(\wt{W}_{-1}(x;q)^{-1}\right)^T
    \begin{pmatrix}
      1&0&0\\0&x^{-2}&0\\0&0&x^2
    \end{pmatrix},
    \label{eq:tW52-id}
  \end{equation}
  consisting of 9 scalar equations.  Each of these equations is a
  specialization of \eqref{eq:H-con}, sometimes after applying the
  recursion relation \eqref{rec52m}.
  
  Observe that Equation~\eqref{rec52m} written in matrix form,
  implies that
  \begin{equation}
    \label{Wmatrec}
    W_{m+1}(x;q) =
    \begin{pmatrix}
      0&1&0\\0&0&1\\1&-s(x)&s(x)-q^{2+m}
    \end{pmatrix} W_m(x;q) 
  \end{equation}
  where 
  \begin{equation}
    \label{sx}
    s(x) = x+x^{-1}+1 \,.
  \end{equation}
  Equation~\eqref{WW52} follows from~\eqref{WW52b} together
  with~\eqref{Wmatrec}.

  Part (e) follows from the \texttt{HF} package. The
  fact that they are a basis follows from (f).

  For part (f), the first and third generators of the annihilating
  ideal \eqref{ann52}, which annihilate $A_m(x;q)$, $B_m(x;q)$,
  $C_m(x;q)$ allow expressing $A_m(q x;q)$, $A_m(q^2x;q)$ in terms of
  $A_m(x;q), A_{m+1}(x;q), A_{m+2}(x;q)$, and similarly for
  $B_m(q^j x;q)$, $C_m(q^j x;q)$ ($j=1,2$).  It follows that the
  Wronskian~\eqref{W52x} and the Wronskian~\eqref{eq:Wm52} are related
  \begin{equation}
    \cW_m(x;q) = M_m(x;q) W_m(x;q)
    \label{eq:W52xm}
  \end{equation}
  where $M_m(x;q)$ is a $3\times 3$ matrix with entries
  \begin{align}
    (M_m(x;q))_{1,1} =
    &1,\quad (M_m(x;q))_{1,2} =
      0,\quad(M_m(x;q))_{1,3} = 0,\nn
      (M_m(x;q))_{2,1} =
    &-q^{-1-m}(1-q^{1+m}-q x^2)\nn
      (M_m(x;q))_{2,2} =
    &q^{-1-m}x^{-1}(1-q^{1+m}+x+(-q+q^{2+m}) x^2
      -qx^3)\nn
      (M_m(x;q))_{2,3} =
    &-q^{-1-m}x^{-1}(1-qx^2)\nn
      (M_m(x;q))_{3,1} =
    &q^{-5-2m}x^{-3}
      \big(1 + (-q + q^{2+m}) x -(q + q^2 + q^3) x^2 \nn +
    &(q^2 + q^3 + q^4 - q^{3+m} - 2 q^{4+m} - q^{5+m} + q^{5+2m}) x^3 \nn +
    &(q^3 + q^4 + 
      q^5) x^4 + (-q^4 - q^5 - q^6 + q^{5+m} + q^{6+m} + q^{7+m}) x^5 - 
      q^6 x^6 + q^7 x^7\big)\nn
      (M_m(x;q))_{3,2} =
    &-q^{-5-2m}x^{-4}(1-qx)(1+qx)\big(
      1 + (1-q + q^{2+m}) x + (-2 q - q^3 + q^{2+m}) x^2 \nn+
    &(-q + q^2 - q^3 +
      q^4 - 2 q^{3+m} - q^{4+m} - q^{5+m} + q^{4+2m} + q^{5+2m}) x^3\nn+
    &(q^2 + 2 q^4 - q^{3+m} - q^{4+m} - q^{5+m}) x^4
      + (q^4 - q^5 + q^{6+m}) x^5 - q^5 x^6   \big)\nn
      (M_m(x;q))_{3,3} =
    &q^{-5-2m}x^{-4}(1-qx) (1+qx)\big(1-qx-(q+q^3) x^2 \nn+
    &(q^2+q^4-q^{3+m}-q^{4+m})x^3 +q^4 x^4 - q^5 x^5\big)
  \end{align}
  After taking determinants on both sides, one finds that
  \begin{equation}
    \det(\calW_m(x;q)) = -q^{-5-2m} x^{-5}
    (1-qx)(1+qx)(1-qx^2)(1-q^3x^2)
    \det(W_m(x;q))
    \label{detW52xm}
  \end{equation}
  This, together with~\eqref{W520} concludes the proof of~\eqref{detW52x}.
\end{proof}

We now come to Conjecture~\ref{conj.ann} concerning a refinement of
the $\Ahat$-polynomial.  
As in Section~\ref{sub.41holo}, we can use Theorems~\ref{thm.52a}
and~\ref{thm.52b} to obtain explicit linear $q$-difference equations
for the descendant integrals with respect to the $u$ and the $m$
variables, and in doing so, we will obtain a refinement of the
$\Ahat$-polynomial. To simplify Equation~\eqref{41x-desc-fac}, let us define
a normalized version of the descendant state-integral 
\begin{equation}
\label{Z52r}
z_{\knot{5}_2,m,\mu}(u;\tau) =
(-1)^{m+\mu}q^{-m/2}\tq^{-\mu/2}
Z_{\knot{5}_2,m,\mu}(u;\tau) \,.
\end{equation}

Our next theorem confirms Conjecture~\ref{conj.ann} for the $\knot{5}_2$
knot.

\begin{theorem}
  \label{thm.52c}
  \rm
  $z_{\knot{5}_2,m,k}(\ub;\tau)$ is a $q$-holonomic function of $(m,u)$
  with annihilator ideal $\calI_{\knot{5}_2}$ given in~\eqref{ann52}.
  As a function of $u$ (resp., $m$) it is annihilated by the operators
  $\Ahat_{\knot{5}_2}(S_x,x,q^m,q)$ (resp.,
  $\Bhat_{\knot{5}_2}(S_m,x,q^m,q)$) given by~\eqref{rec52m}
  and~\eqref{rec52x}, whose classical limit is
\begin{equation}
\begin{aligned}
  & \Ahat_{\knot{5}_2}(S_x,x,q^m,1) =\\
  &-(1-x)^2(1+x)^2
  \Big(q^mx^2-(1-x-2(1-q^m)x^2+2x^3+(1+q^m)x^4-x^5)S_x\\
  &-x(1-(1+q^m)x-2x^2+2(1-q^m)x^3+x^4-x^5)S_x^2-q^mx^4S_x^3\Big)\,,
\end{aligned}
\end{equation}
and
\begin{equation}
  \Bhat_{\knot{5}_2}(S_m,x,q^m,1) =
  (1-S_m)(1-xS_m)(1-x^{-1}S_m)-q^m S_m^2\,.
\end{equation}
$\Ahat_{\knot{5}_2}(S_x,x,1,1)$ is the $A$-polynomial of the knot,
$\Ahat_{\knot{5}_2}(S_x,x,1,q)$ is the (homogeneous part) of the
$\Ahat$-polynomial of the knot and $\Bhat_{\knot{5}_2}(y,x,1,1)$ is
the defining equation of the curve~\eqref{52critb}.
\end{theorem}

\subsection{Taylor series expansion at $u=0$}
\label{sub.52near1}

In this section we discuss the Taylor expansion of the descendant
holomorphic blocks and descendant state-integral of the $\knot{5}_2$
knots at $u=0$.  
Using the definition of $H(x,y,z;q)$
and~\eqref{eq:phin}--\eqref{eq:tphin}, we find
\begin{align}
  A_m(\re^u;q)
  &= (q;q)_\infty^2
    \sum_{n=0}^\infty \frac{q^{n(n+1)+nm}}{(q;q)_n^3}
    \phi_n(u) \phi_n(-u)
    \nn
  &= (q;q)_\infty^2
    \left(\alpha^{(m)}_0(q)+u^2\alpha^{(m)}_2(q)+O(u^3)\right)
    \label{eq:A-uexp-52}
\end{align}
where
\begin{align}
  &\alpha^{(m)}_0(q) = H^{+}_{0,m}(q),\label{eq:a0-52}\\
  &\alpha^{(m)}_2(q) =
    \sum_{n=0}^\infty \frac{q^{n(n+1)+nm}}{(q;q)_n^3}
    (-E_2^{(n)}(q)),\label{eq:a2-52}
\end{align}
and
\begin{align}
 C_m(\re^{-u};q)& = B_m(\re^u;q)\nn
  &= (1-\re^{-u})^{-2}(q;q)_\infty^{-2}
    \phi_0(u)^{-2}\phi_0(-u)^{-2}
    \sum_{n=0}^\infty
    \frac{q^{n(n+1)+nm}}{(q;q)_n^3}
    \phi_n(u)\phi_n(2u) \re^{(m+2n)u} \nn
  &= u^{-2}(q;q)_\infty^{-2}
    \left(\beta^{(m)}_0(q)+u\beta^{(m)}_1(q)
    +u^2\beta^{(m)}_2(q)+O(u^3)\right)
    \label{eq:B-uexp-52}
\end{align}
where the coefficients satisfy
\begin{align}
  &\beta^{(m)}_0(q) = H^{+}_{0,m}(q),
    \label{eq:b0-52}\\
  &\beta^{(m)}_1(q) = H^{+}_{1,m}(q),
    \label{eq:b1-52}\\
  &\beta^{(m)}_2(q) =
    \frac{1}{2}H^+_{2,m}(q)+\alpha^{(m)}_2(q).
    \label{eq:b2-52}
\end{align}
Similarly, using the definition of $H(x,y,z;q^{-1})$, we find
\begin{align}
  A_m(\re^u;q^{-1})
  &=
    \frac{1}{(1-\re^u)(1-\re^{-u})(q;q)_\infty^2}
    \sum_{n=0}^\infty\frac{(-1)^nq^{\frac{1}{2}n(n+1)-nm}}
    {(q;q)_n^3}
    \wt{\phi}_n(u) \wt{\phi}_n(-u) \nn
  &= -u^{-2}(q;q)_\infty^{-2}
    \left(\wt{\alpha}^{(m)}_0(q)+u^2
    \wt{\alpha}^{(m)}_2(q) + O(u^3)\right)
    \label{eq:tA-uexp-52}
\end{align}
where
\begin{align}
  &\wt{\alpha}^{(m)}_0(q) = H^{-}_{0,-m}(q),
  \label{eq:ta0-52}\\
  &\wt{\alpha}^{(m)}_2(q) = \sum_{n=0}^\infty
    \frac{(-1)^nq^{\frac{1}{2}n(n+1)-nm}}
    {(q;q)_n^3}\left(-\frac{1}{12}+2E_2^{(0)}(q)-E_2^{(n)}(q)\right),
    \label{eq:ta2-52}
\end{align}
and
\begin{align}
  C_m(\re^{-u};q^{-1})
  &=B_m(\re^u;q^{-1})\nn
  &=
    (1-\re^{u})(1-\re^{2u})^{-1}(q;q)_\infty^2\phi_0(u)^2\phi_0(-u)^2
    \sum_{n=0}^\infty\frac{(-1)^nq^{\frac{1}{2}n(n+1)-nm}}
    {(q;q)_n^3}
    \wt{\phi}_n(u)\wt{\phi}_n(2u)\re^{(2n+m)u} \nn
  &= \frac{1}{2}(q;q)_\infty^2\left(\wt{\beta}^{(m)}_0(q)+u
    \wt{\beta}^{(m)}_1(q)+u^2
    \wt{\beta}^{(m)}_2(q) + O(u^3)\right)
    \label{eq:tB-uexp-52}
\end{align}
where the coefficients satisfy
\begin{align}
  &\wt{\beta}^{(m)}_0(q) = H^{-}_{0,-m}(q),
  \label{eq:tb0-52}\\
  &\wt{\beta}^{(m)}_1(q) = -H^{-}_{1,-m}(q),
  \label{eq:tb1-52}\\
  &\wt{\beta}^{(m)}_2(q) =
    \frac{1}{2}H^{-}_{2,-m}(q) + \wt{\alpha}_2^{(m)}(q)
    \label{eq:tb2-52}
\end{align}
Applying these results to the right hand side of \eqref{52x-desc-fac},
as well as using the trick
\begin{equation}
  \frac{1}{2\pi\ri} = -\frac{1}{12}(\tau E_2(q) - \tau^{-1}E_2(\tq)),
\end{equation}
we find the $O(1/u^2)$ and $O(1/u)$
contributions from~\eqref{eq:tA-uexp-52} and~\eqref{eq:tB-uexp-52}
cancel, and the $O(u^0)$ contributions reproduce~\eqref{520-desc-fac}.

As an application of the above computations, we demonstrate that
Theorem~\ref{thm.520}, especially \eqref{det52}, \eqref{WWT52}, as
well as the recursion relation \eqref{52qdiffp}, can be proved by
taking the $u=0$ limit of the analogue identities in
Theorem~\ref{thm.52b}.

Using the expansion formulas of holomorphic blocks \eqref{eq:A-uexp-52},
\eqref{eq:B-uexp-52}, \eqref{eq:tA-uexp-52}, \eqref{eq:tB-uexp-52},
the Wronskians can be expanded as
\begin{align}
  W_m(\re^u;q) =
  &\left(\sum_{j=0}W^+_{m,j}(q)u^j\right)
  \begin{pmatrix}
    (q;q)_\infty^{2}&0&0\\
    0&u^{-2}(q;q)_\infty^{-2}&0\\
    0&0&u^{-2}(q;q)_\infty^{-2}
  \end{pmatrix},
         \label{eq:Wp-uexp-52}\\
  W_m(\re^u;q^{-1}) =
  &\left(\sum_{j=0}W_{m,j}^-(q)u^j\right)
  \begin{pmatrix}
    -u^{-2}(q;q)_\infty^{-2}&0&0\\
    0&\frac{1}{2}(q;q)_\infty^{2}&0\\
    0&0&\frac{1}{2}(q;q)_\infty^2
  \end{pmatrix}
  \label{eq:Wn-uexp-52}
\end{align}
where
\begin{equation}
  W_{m,j}^+(q) =
  \begin{pmatrix}
    \alpha_j^{(m)} & \beta_j^{(m)} & (-1)^j\beta_j^{(m)}\\
    \alpha_j^{(m+1)} & \beta_j^{(m+1)} & (-1)^j\beta_j^{(m+1)}\\
    \alpha_j^{(m+2)} & \beta_j^{(m+2)} & (-1)^j\beta_j^{(m+2)}
  \end{pmatrix}(q), \quad W_{m,j}^-(q) =
  \begin{pmatrix}
    \wt{\alpha}_j^{(m)} & \wt{\beta}_j^{(m)} & (-1)^j\wt{\beta}_j^{(m)}\\
    \wt{\alpha}_j^{(m+1)} & \wt{\beta}_j^{(m+1)} & (-1)^j\wt{\beta}_j^{(m+1)}\\
    \wt{\alpha}_j^{(m+2)} & \wt{\beta}_j^{(m+2)} &
    (-1)^j\wt{\beta}_j^{(m+2)}
  \end{pmatrix}(q).
\end{equation}
Note that $\alpha_j^{(m)}(q) = \wt{\alpha}_j^{(m)}(q) = 0$ if $j$ is odd.
Taking the determinant of \eqref{eq:Wp-uexp-52}, we find
\begin{equation}
  \det W_m(\re^u;q) = u^{-1}(q;q)_\infty^{-2} \det W_m(q) + O(u^0)
\end{equation}
which together with the $u$-expansion of the right hand side of
\eqref{W520} leads to the determinant identity \eqref{det52}.
Furthermore, by substituting \eqref{eq:Wp-uexp-52},
\eqref{eq:Wn-uexp-52} into the Wronskian relation \eqref{WW52b},
we find the left hand side reduces to
\begin{equation}
  \frac{1}{2}W_{-1}(q)
  \begin{pmatrix}
    0&0&1\\0&2&0\\1&0&0
  \end{pmatrix}
  W_{-1}(q^{-1})
  \begin{pmatrix}
    0&0&1\\0&1&0\\1&0&0
  \end{pmatrix} + O(u^1),
\end{equation}
which together with the $u$-expansion of the right hand side, leads in
the leading order to \eqref{WWT52b} for $m=1$.
The more general case follows from the identity of $m=1$ by applying
the recursion relations \eqref{52qdiffp} on the Wronskians.

Finally, from the expression \eqref{eq:a0-52},\eqref{eq:b0-52} of the
leading order coefficients $\alpha_0^{(m)}(q)$,
$\beta_0^{(m)}(q)$, of $A_m(\re^u;q)$,
$B_m(\re^u;q)$, $C_m(\re^u;q)$ in the expansion of $u$,
one concludes that the recursion relation \eqref{52qdiffp} should be
the $u=0$ limit of the recursion relation \eqref{rec52m} in $m$, and
one can easily check it is indeed the case.

\subsection{Stokes matrices near $u=0$}
\label{sub.52stokes}

\begin{figure}[htpb!]
\leavevmode
\begin{center}
\includegraphics[height=7cm]{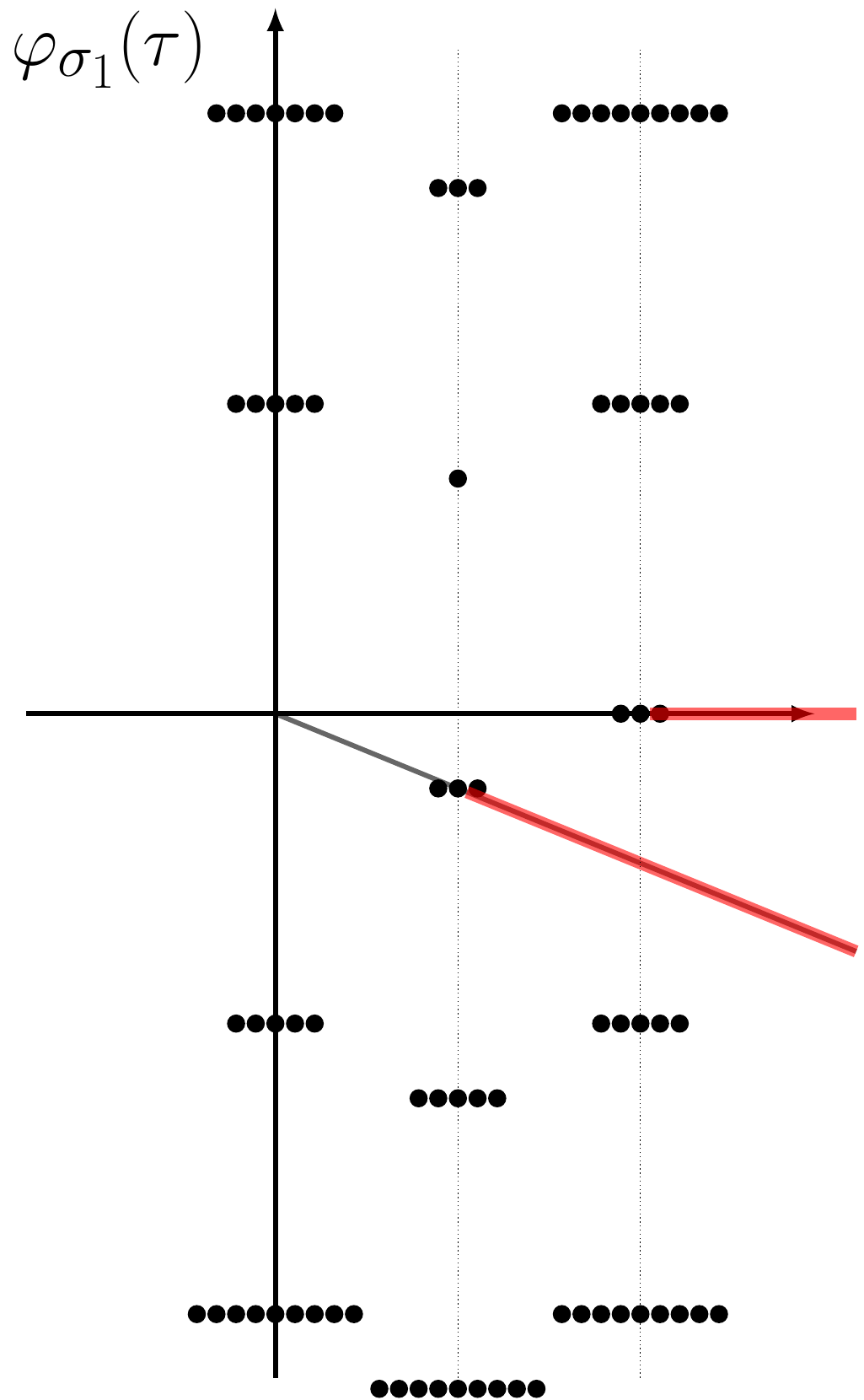}\hspace{4ex}
\includegraphics[height=7cm]{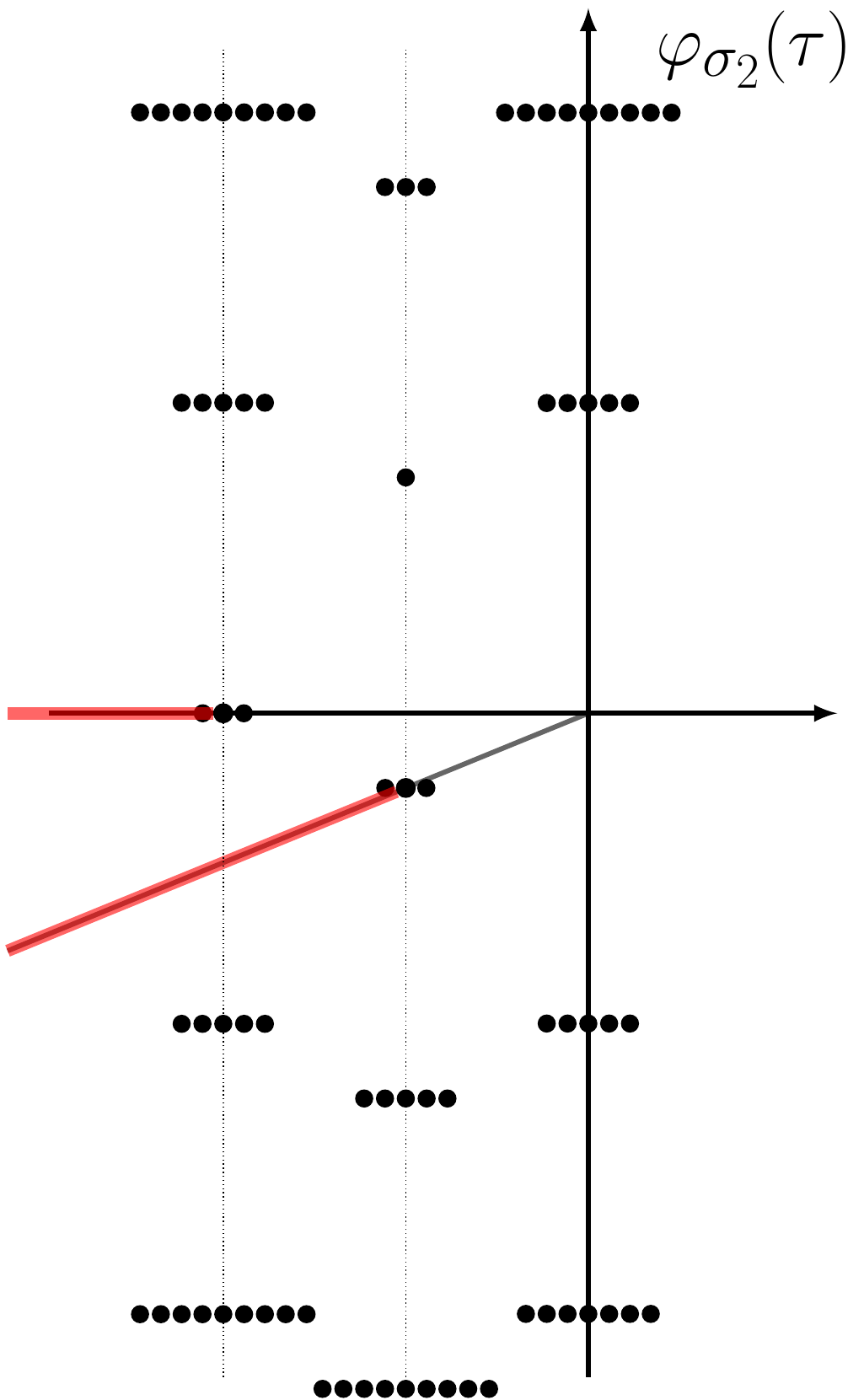}\hspace{4ex}
\includegraphics[height=7cm]{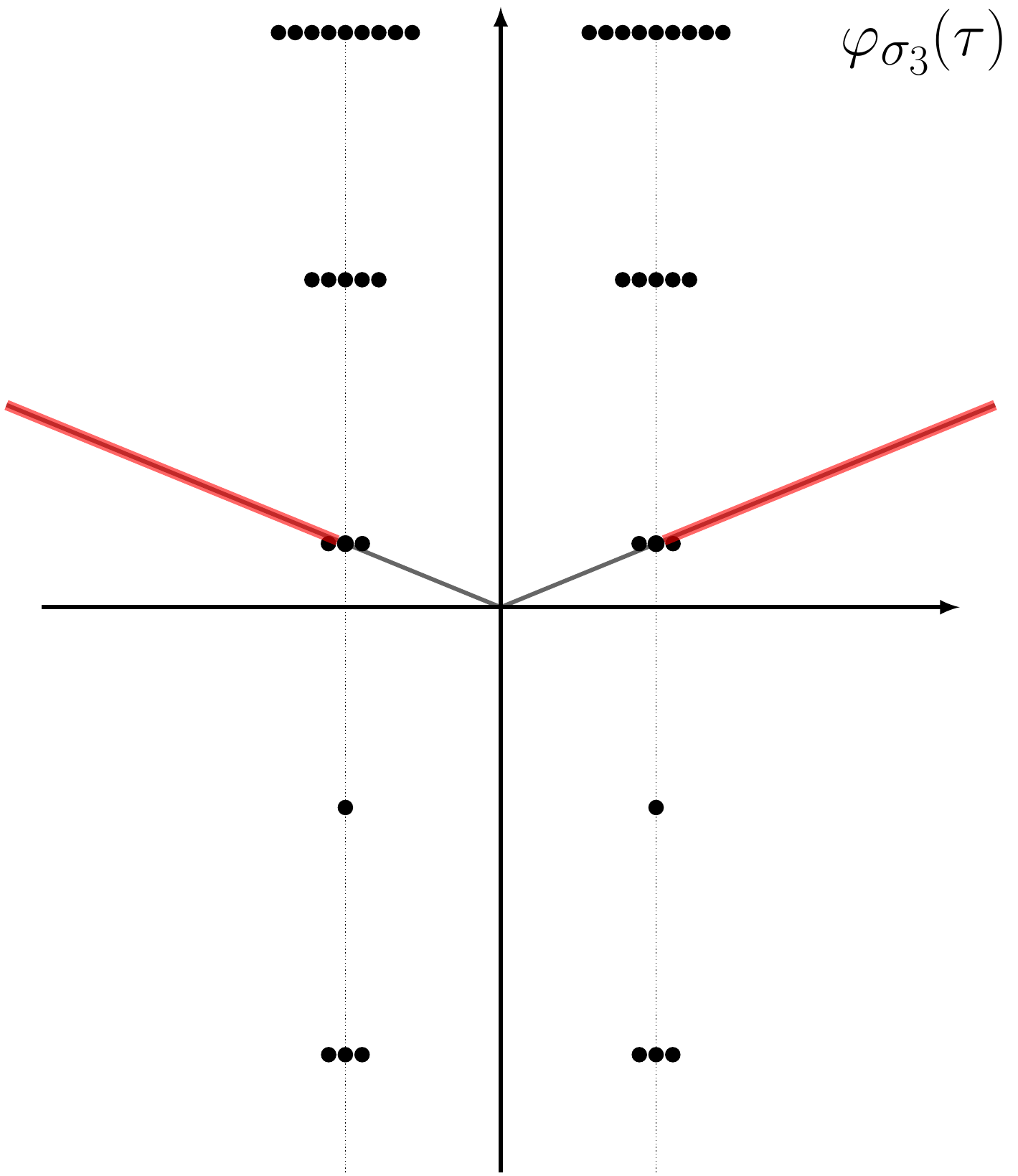}
\end{center}
\caption{The singularities in the Borel plane for the series
  $\varphi_{\sigma_{j}} (u;\tau)$ with $j=1,2,3$ of knot $\knot{5}_2$
  where $u$ is close to zero.}
\label{fig:sings-m-52}
\end{figure}

\begin{figure}[htpb!]
\leavevmode
\begin{center}
\includegraphics[height=7cm]{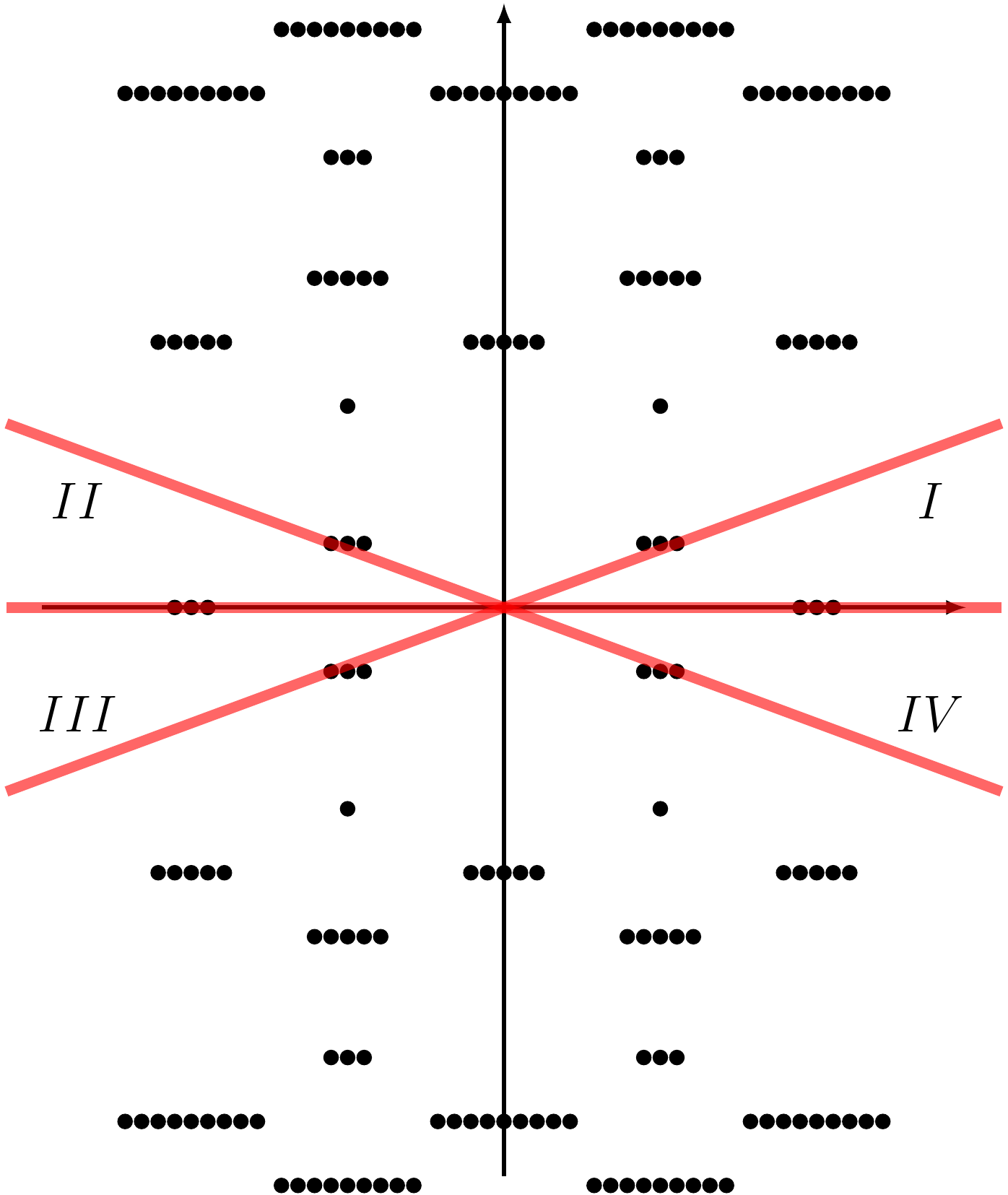}
\end{center}
\caption{Four different sectors (and additional two auxiliary regions)
  in the $\tau$-plane for $\Phi(u;\tau)$ of knot $\knot{5}_2$ with $u$
  close to zero.}
\label{fig:secs-m-52}
\end{figure} 

In this section we give a conjecture for the Stokes matrices of
the asymptotic series $\varphi(x;\tau)$.

We only consider the case when $u$ is not far away from zero, or
equivalently $x$ not far away from 1.  To be more precise, we focus on
real $x$ and constrain $x$ to be in the interval containing 1 between
the two real solutions to the discriminant \eqref{disc52y}.  In this
case, there are mild changes to the resurgent structure discussed in
Section~\ref{sub.52x=1}.  Each singular point in
Figure~\ref{fig:sings-u0-52} splits to a cluster of neighboring
singular points as shown in Figure~\ref{fig:sings-m-52}.  In
particular, each of the six singular points $\iota_{i,j}$ ($i\neq j$)
splits to a cluster of neighboring three separated from each other by
$\log x$.  We label the four regions separating singularities on
positive and negative real axis and the other singularities by
$I,II,III,IV$ (see Figure~\ref{fig:secs-m-52}).  In each of the four
regions, we have the following the results.

\begin{conjecture}
  \label{conj.52M}
  The asymptotic series and the holomorphic blocks are related
  by~\eqref{abrelx} with the diagonal matrix $\Delta(u)$ as
  in~\eqref{D52} where the matrices $M_R(x;q)$ are given in terms of
  the $W_{-1}(x;q)$ as follows 
\begin{subequations}
\begin{align}
  M_{I}(x;q) =
  &W_{-1}(x;q)^T
    \begin{pmatrix}
      0&0&-1\\1&-s(x)&0\\0&1&0
    \end{pmatrix}, & |q|<1\,,
  \label{eq:M1-52}\\
  M_{II}(x;q) =
  &W_{-1}(x^{-1};q)^T
    \begin{pmatrix}
      0&0&-1\\-s(x)&1&0\\1&0&0
    \end{pmatrix}, & |q|<1\,,
  \label{eq:M2-52}\\
  M_{III}(x;q) =
  &W_{-1}(x^{-1};q)^T
    \begin{pmatrix}
      0&0&-1\\1&1&0\\1&0&0
    \end{pmatrix}, & |q|>1\,,
  \label{eq:M3-52}\\
  M_{IV}(x;q) =
  &W_{-1}(x;q)^T
    \begin{pmatrix}
      0&0&-1\\1&1&0\\0&1&0
    \end{pmatrix}, & |q|>1\,,
                             \label{eq:M4-52}
\end{align}
\end{subequations}
and where $s(x)$ is given in~\eqref{sx}. 
\end{conjecture}
Just as in the case of the $\knot{4}_1$ knot, 
the above conjecture completely determines the resurgent structure of
$\Phi(\tau)$. Indeed, it implies that the Stokes matrices, given by
Equations~\eqref{4G} and~\eqref{SG} are explicitly
given by:
\begin{subequations}
\begin{align}
    \ms{S}^+(x;q) =
    &\begin{pmatrix}
      0&1&0\\
      0&1&1\\
      -1&0&0
    \end{pmatrix}
               W_{-1}(x^{-1};q^{-1})
               W_{-1}(x;q)^T
              \begin{pmatrix}
                0&0&-1\\
                1&1&0\\
                0&1&0
              \end{pmatrix}\,,
  \label{eq:Sp-52}\\
    \ms{S}^-(x;q)=
    &\begin{pmatrix}
      0&-s(x)&1\\
      0&1&0\\
      -1&0&0
    \end{pmatrix}
            W_{-1}(x;q)
            W_{-1}(x^{-1};q^{-1})^T
              \begin{pmatrix}
                0&0&-1\\
                -s(x)&1&0\\
                1&0&0
              \end{pmatrix}
              \label{eq:Sn-52}
\end{align}
\end{subequations}
for $|q|<1$. Note that 
Equations~\eqref{52critb}, \eqref{V52b}, \eqref{52delta},
\eqref{eq:vf-52} imply that (one can also see this from \eqref{Z52x})
\begin{equation}
  s(\Phi_\s)(x^{-1};\tau) = s(\Phi_\s)(x;\tau)
\end{equation}
for any $\tau\in\IC$ whenever the asymptotic series is Borel summable.
It follows that the Stokes matrices must be invariant under the
reflection $\pi: x\mapsto x^{-1}$.  Using the property \eqref{eq:-xWm}
of $W_{m}(x;q)$, it is easy to show that the Stokes matrices
\eqref{eq:Sp-52},\eqref{eq:Sn-52} indeed satisfy this consistency
condition.

The $q\mapsto 0$ limit of the Stokes matrices factorizes
\begin{align}
  \notag
  \ms{S}^+(x;0)  &=
  \mfS_{\s_3,\s_1}(x)\mfS_{\s_3,\s_2}(x)\mfS_{\s_1,\s_2}(x)
    \\ &=
     \begin{pmatrix}
       1&0&0\\0&1&0\\-s(x)&0&1
     \end{pmatrix}
     \begin{pmatrix}
       1&0&0\\0&1&0\\0&s(x)&1
     \end{pmatrix}
     \begin{pmatrix}
      1&s(x)+1&0\\0&1&0\\0&0&1
    \end{pmatrix},  \label{eq:S0px52}\\
    \notag
  \ms{S}^-(x;0)  &=
  \mfS_{\s_1,\s_3}(x)\mfS_{\s_2,\s_3}(x)\mfS_{\s_2,\s_1}(x)
    \\ &=
     \begin{pmatrix}
       1&0&s(x)\\0&1&0\\0&0&1
     \end{pmatrix}
     \begin{pmatrix}
       1&0&0\\0&1&-s(x)\\0&0&1
     \end{pmatrix}
     \begin{pmatrix}
      1&0&0\\-s(x)-1&1&0\\0&0&1
    \end{pmatrix}, \label{eq:S0nx52}
\end{align}
where the non-vanishing off-diagonal entries of $\mfS_{\s_i,\s_j}(x)$
encode the Stokes constants associated to the Borel singularities
split from $\iota_{i,j}$.  Using the unique factorization
Lemma~\ref{lem.fac} and the Stokes matrix $\sf{S}$ from above, we can
compute the Stokes constants and the corresponding matrix $\calS$ of
Equation~\eqref{calS} to arbitrary order in $q$, and we find that
\begin{subequations}
\begin{align}
  \mc{S}^+_{\s_1,\s_1} =
  &\ms{S}^+(q)_{1,1}-1\nn=
  &(-4-x^{-2}-3x^{-1}-3 x-x^2) q+(-1+x^{-3}+x^{-2}+x^2+x^3) q^2+\cO(q^3),\\
  \mc{S}^+_{\s_1,\s_2} =
  &(4+x^{-2}+3x^{-1}+3 x+x^2) q\nn+
  &(37+x^{-4}+5x^{-3}+16x^{-2}+30x^{-1}+30 x+16 x^2+5 x^3+x^4) q^2+\cO(q^3),\\
  \mc{S}^+_{\s_1,\s_3} =
  &q+(1+x^{-1}+x) q^2+\cO(q^3),\\
  \mc{S}^+_{\s_2,\s_1} =
  &-(4+x^{-2}+3x^{-1}+3 x+x^2) q\nn-
  &(37+x^{-4}+5x^{-3}+16x^{-2}+30x^{-1}+30 x+16 x^2+5 x^3+x^4) q^2+\cO(q^3),\\
  \mc{S}^+_{\s_2,\s_2} =
  &(4+x^{-2}+3x^{-1}+3 x+x^2) q\nn+
  &(37+x^{-4}+5x^{-3}+16x^{-2}+30x+30 x+16 x^2+5 x^3+x^4) q^2+\cO(q^3),\\
  \mc{S}^+_{\s_2,\s_3} =
  &-(7+2x^{-2}+5x^{-1}+5 x+2 x^2) q\nn-
  &(59+2x^{-4}+10x^{-3}+27x^{-2}+49x^{-1}+49 x+27 x^2+10 x^3+2 x^4) q^2+\cO(q^3),\\
  \mc{S}^+_{\s_3,\s_1} =
  &-q-(1+x^{-1}+x) q^2+\cO(q^3)\\
  \mc{S}^+_{\s_3,\s_2} =
  &(7+2x^{-2}+5x^{-1}+5 x+2 x^2) q\nn+
  &(59+2x^{-4}+10x^{-3}+27x^{-2}+49x^{-1}+49 x+27 x^2+10 x^3+2 x^4) q^2+\cO(q^3),\\
  \mc{S}^+_{\s_3,\s_3} =
  &0.
\end{align}
\end{subequations}
Note that the series $\mc{S}^+_{\s_1,\s_2}$ and $\mc{S}^+_{\s_2,\s_2}$
are different, even though their first few terms are coincidental. They
differ in higher orders, as one can already see in the $u=0$ limit in
Section~\ref{sub.52x=1}. The matrix $\calS$ satisfies the symmetry
\begin{equation}
  \mc{S}^+_{\s_i,\s_j}(x;q) =
  -\mc{S}^+_{\s_{\varphi(i)},\s_{\varphi(j)}}(x;q),\quad i\neq j,
\end{equation}
with $\varphi(1)=2,\varphi(2)=1,\varphi(3)=3$, and they display the
familiar feature that the entries of the
matrix$\mc{S}^+(x;q)=(\mc{S}^+_{\s_i,\s_j}(x;q))$ (except the
upper-left one) are (up to a sign) in $\IN[x^{\pm 1}][[q]]$.
Similarly we can extract the Stokes constants
$\mc{S}^{(\ell,k)}$ ($k<0$) associated to the singularities below
$\iota_{i,j}$ in the lower half plane and assemble into
$q^{-1}$-series $\mc{S}^-_{\s_i,\s_j}(x;q^{-1})$.  We find the
relation
\begin{gather}
  \mc{S}^-_{\s_i,\s_j}(x;q) = -\mc{S}^+_{\s_j,\s_i}(x;q),\;\; i\neq
  j\quad\text{and}\quad
  \mc{S}^-_{\s_i,\s_i}(x;q) = +\mc{S}^+_{\s_{\varphi(i)},\s_{\varphi(i)}}(x;q)\,.
\end{gather}

Let us now verify Conjecture~\ref{conj.S3D} for the $\knot{5}_2$ knot.
The same logic presented at the end of Section~\ref{sub.41stokes} also
holds here.
From the form of the Stokes matrix \eqref{eq:Sp-52} as well as
\eqref{Wmatrec},\eqref{eq:W52xm}, we immediately conclude that
\begin{equation}
  \ms{S}^+(x;q) \stackrel{\cdot\cdot}{=} \cW_{0}(x^{-1};q^{-1})\cdot
  \cW_{0}(x;q)^T.
  \label{eq:SW52}
\end{equation}
Using the uniform notation for all holomorphic blocks
\begin{equation}
  (B_{\knot{5}_2}^\alpha(x;q))_{\alpha=1,2,3} = (A_0(x;q),B_0(x;q),C_0(x;q)),
\end{equation}
the right hand side of \eqref{eq:SW52} reads\footnote{Note that if we
  compare with the DGG indices computed in \cite{Beem} we have to
  modify the last line in \eqref{eq:WW52} slightly due to different
  conventions for holomorphic blocks 
  \begin{equation}
    \left(\cW_{0}(x^{-1};q^{-1})\cdot \cW_{0}(x;q)^T\right)_{i+1,j+1} =
    q^{\frac{3}{2}(i^2-j^2)}x^{-3(j-i)}
    \text{Ind}_{\knot{5}_2}^{\text{rot}}
    (j-i,q^{-\frac{j+i}{2}}x^{-1};q),\quad i,j=0,1,2.
  \end{equation}
  The right hand side of \eqref{eq:SInd} should be modified
  accordingly.  This however does not affect \eqref{S3Dx}.} 
\begin{align}
  \cW_{0}(x^{-1};q^{-1})\cdot \cW_{0}(x;q)^T =
  &\left(
    \sum_{\alpha}B^{\alpha}_{\knot{5}_2}(q^jx;q)
    B^{\alpha}_{\knot{5}_2}(q^{-i}x^{-1};q^{-1})
    \right)_{i,j=0,1,2}\nn = 
  &\left(
    \text{Ind}_{\knot{5}_2}^{\text{rot}}(j-i,q^{\frac{j+i}{2}}x;q)
    \right)_{i,j=0,1,2},
  \label{eq:WW52}
\end{align}
reproducing the right hand side of \eqref{eq:SInd} in
Conjecture~\ref{conj.S3D}.  Furthermore, the forms of the accompanying
matrices on the left and on the right are such that the $(1,1)$ entry
of $\ms{S}^+(x;q)$ equates exactly the DGG index with no magnetic
flux.  By explicit calculation,
\begin{align}
  \ms{S}^+(x;q)_{1,1} =
  &H(x,x^{-1},1;q)H(x^{-1},x,1;q^{-1})
    +H(x,x^2,x^2;q)H(x^{-1},x^{-2},x^{-2};q^{-1})\nn
  &+H(x^{-1},x^{-2},x^{-2};q)H(x,x^2,x^2;q^{-1})\nn=
  &1 -(x^{-2}+ 3x^{-1}+4+3 x + x^2) q + (x^{-3} + x^{-2} -1+ x^2 + 
    x^3) q^2 + O(q^3).
    \label{eq:S1152}
\end{align}
The right hand side of the first two lines in \eqref{eq:S1152} is the formula
for the DGG index given in \cite{Beem}.

\subsection{The Borel resummations of the asymptotic series $\Phi$}
\label{sub.borel52}

As in the case of the $\knot{4}_1$ knot, Conjecture~\ref{conj.52M}
identifies the Borel resummations of the factorially divergent series
$\Phi(x;\tau)$ with the descendant state-integrals, thus lifting the
Borel resummation to holomorphic functions on the cut-plane $\BC'$.

\begin{corollary}(of Conjecture~\ref{conj.52M})
  \label{cor.52borel}
  We have 
  \begin{subequations}
    \begin{align}
      s_{I}(\Phi)(x;\tau) =
      &\begin{pmatrix}
        0&1&1\\0&1&0\\-1&0&0
      \end{pmatrix} W_{-1}(\tx;\tq^{-1})\Delta(\tau)B(x;q)\,,
                        \label{eq:HU1-52}\\
  s_{II}(\Phi)(x;\tau) =
  &\begin{pmatrix}
    0&1&0\\0&1&1\\-1&0&0
  \end{pmatrix} W_{-1}(\tx^{-1};\tq^{-1})
                                    \Delta(\tau) B(x;q)\,,
                         \label{eq:HU2-52}\\
  s_{III}(\Phi)(x;\tau) =
  &\begin{pmatrix}
    0&1&0\\0&-s(\tx)&1\\-1&0&0
  \end{pmatrix} W_{-1}(\tx^{-1};\tq^{-1})
                                    \Delta(\tau) B(x;q)\,,
                         \label{eq:HU3-52}\\
  s_{IV}(\Phi)(x;\tau)=
  &\begin{pmatrix}
    0&-s(\tx)&1\\0&1&0\\-1&0&0
  \end{pmatrix}W_{-1}(\tx,\tq^{-1})\Delta(\tau) B(x;q)\,.
                  \label{eq:HU4-52}
\end{align}
\end{subequations}
where the right hand side of~\eqref{eq:HU1-52}--\eqref{eq:HU4-52} are
holomorphic functions of $\tau \in \IC'$, as they are linear
combinations of the descendants \eqref{52x-desc-fac}.
\end{corollary}

\subsection{Numerical verification}
\label{sub.52num}

\begin{table}
  \centering%
  \subfloat[Region $I$: $\tau = \frac{1}{18}\re^{\frac{\pi\ri}{12}}$]
  {\begin{tabular}{*{5}{>{$}c<{$}}}\toprule
     &|\frac{s_{I}(\Phi)(x;\tau)}{P_{I}(x;\tau)}- 1|
     &|\frac{s_{I}(\Phi)(x;\tau)}{s'_{I}(\Phi)(x;\tau)}- 1|
     &|\tq(\tau)|& |\tx(x,\tau)^{-1}|\\\midrule
     \s_1& 1.1\times 10^{-38} & 1.1\times 10^{-38}
     &\multirow{3}{*}{$1.9\times 10^{-13}$}
     &\multirow{3}{*}{$0.042$}\\
     \s_2& 2.2\times 10^{-63} & 1.7\times 10^{-62}&&\\
     \s_3& 2.7\times 10^{-48} & 2.5\times 10^{-48}&&\\
     \bottomrule
   \end{tabular}}\\
 \subfloat[Region $II$: $\tau = \frac{1}{18}\re^{\frac{11\pi\ri}{12}}$]
 {\begin{tabular}{*{5}{>{$}c<{$}}}\toprule
    &|\frac{s_{II}(\Phi)(x;\tau)}{P_{II}(x;\tau)}- 1|
    &|\frac{s_{II}(\Phi)(x;\tau)}{s'_{II}(\Phi)(x;\tau)}- 1|
    &|\tq(\tau)|& |\tx(x,\tau)|\\\midrule
    \s_1& 2.2\times 10^{-63} & 1.7\times 10^{-62}
    &\multirow{3}{*}{$1.9\times 10^{-13}$}
    &\multirow{3}{*}{$0.042$}\\
    \s_2& 1.1\times 10^{-38} & 1.1\times 10^{-38}&&\\
    \s_3& 2.7\times 10^{-48} & 2.5\times 10^{-48}&&\\
    \bottomrule
   \end{tabular}}
 \caption{Numerical tests of holomorphic lifts of Borel
   sums of asymptotic series for knot $\knot{5}_2$.  We perform the
   Borel-Pad\'e resummation on $\Phi(x;\tau)$ with 180 terms at
   $x=6/5$ and $\tau$ in regions $I,II$, and compute the relative
   difference between them and the right hand side of
   \eqref{eq:HU1-52}--\eqref{eq:HU2-52}, which we denote by
   $P_R(x;\tau)$.  They are within the error margins of Borel-Pad\'e
   resummation, which are estimated by redo the resummation with 276
   terms, denoted by $s'_R(\bullet)$ in the tables. The relative
   errors are much smaller than $|\tq^{\pm 1}|, |\tx^{\pm 1}|$,
   possible sources of additional corrections.}
  \label{tab:num-52}
\end{table}

In this section we explain the numerical verification of
Conjecture~\ref{conj.52M}, which involves a richer resurgent structure
than that of the $\knot{4}_1$ knot. 
We found ample numerical evidence for the resurgent data
\eqref{eq:M1-52}--\eqref{eq:M4-52}.  These numerical tests are
parallel to those performed for knot $\knot{4}_1$, so we will be
sketchy here. Besides, we will mostly focus on $\tau$ in the upper
half plane, while the lower half plane is similar.

The first test is the analysis of radial asymptotics of the left hand
side of \eqref{eq:abrelx-row}, which can be easily done.  The second
test is to compute the Borel resummation $s_R(\Phi_\s')(x;\tau)$ and
by comparing with the left hand side extract terms of
$M_R(\tx,\tq)_{\s,\s'}$ order by order.
To expediate the operation of extraction, instead of $M_R(\tx;\tq)$
we consider
\begin{equation}
  \wt{M}_R(\tx;\tq) =
  \begin{pmatrix}
    1&0&0\\
    0&\th(-\tq^{-1/2}\tx;\tq)^2&0\\
    0&0&\th(-\tq^{1/2}\tx;\tq)^2
  \end{pmatrix} M_R(\tx;\tq)
\end{equation}
whose entries are $\tq$-series
with coefficients in $\IZ[\tx^{\pm 1}]$ instead of in $\IZ(\tx)$.
Using 180 terms
of $\Phi_\s(x;\tau)$ at various values of $x$ and $\tau$, we find
entries of $\wt{M}_{I,II}(x;q)$ up to $O(q^2)$
following results
  \begin{align}
    \wt{M}_{I}(x;q)_{1,1}=
    &1-(x^{-1}+x)q-(x^{-1}-2+x)q^2+O(q^3)\,,\nn
      \wt{M}_{I}(x;q)_{1,2}=
    &-x^{-1}-x+(x^{-2}+2+x^2)q+(x^{-2}-2x^{-1}+1-2x+x^2)q^2+O(q^3)\,,\nn
      \wt{M}_{I}(x;q)_{1,3}=
    &-1+(x^{-1}-1+x)q+(x^{-1}-2+x)q^2+O(q^3)\,,\nn
      \wt{M}_{I}(x;q)_{2,1}=
    &x^2-(x^3+x^4)q-(x^3-x^5)q^2+O(q^3)\,,\nn
      \wt{M}_{I}(x;q)_{2,2}=
    &-x-x^2+(x^2+2x^3+x^4)q+(x^2+x^3-x^4-2x^5)q^2+O(q^3)\,,\nn
      \wt{M}_{I}(x;q)_{2,3}=
    &-x+x^2q+(x^2-x^4)q^2+O(q^3)\,,\nn
      \wt{M}_{I}(x;q)_{3,1}=
    &x^{-2}-(x^{-4}+x^{-3})q+(x^{-5}-x^{-3})q^2+O(q^3)\,,\nn
      \wt{M}_{I}(x;q)_{3,2}=
    &-x^{-2}-x^{-1}+(x^{-4}+2x^{-3}+x^{-2})q
      -(2x^{-5}+x^{-4}-x^{-3}-x^{-2})q^2+O(q^3)\,,\nn
      \wt{M}_{I}(x;q)_{3,3}=
    &-x^{-1}+x^{-2}q-(x^{-4}-x^{-2})q^2+O(q^3)\,.
  \end{align}
and
  \begin{align}
    \wt{M}_{II}(x;q)_{1,1} =
    &-x-x^{-1}+(x^2+2+x^{-2})q+(x^2-2x+1-2x^{-1}+x^{-2})q^2+O(q^3)\,,\nn
      \wt{M}_{II}(x;q)_{1,2} =
    &1-(x+x^{-1})q-(x-2+x^{-1})q^2+O(q^3)\,,\nn
      \wt{M}_{II}(x;q)_{1,3} =
    &-1+(x-1+x^{-1})q+(x-2+x^{-1})q^2+O(q^3)\,,\nn
      \wt{M}_{II}(x;q)_{2,1} =
    &-x-1+(1+2x^{-1}+x^{-2})q+(1+x^{-1}-x^{-2}-2x^{-3})q^2
      +O(q^3)\,,\nn
      \wt{M}_{II}(x;q)_{2,2} =
    &1-(x^{-1}+x^{-2})q-(x^{-1}-x^{-3})q^2+O(q^3)\,,\nn
      \wt{M}_{II}(x;q)_{2,3} =
    &-x+q+(1-x^{-2})q^2+O(q^3)\,,\nn
      \wt{M}_{II}(x;q)_{3,1} =
    &-1-x^{-1}+(x^2+2x+1)q-(2x^3+x^2-x-1)q^2+O(q^3)\,,\nn
      \wt{M}_{II}(x;q)_{3,2} =
    &1-(x^2+x)q+(x^3-x)q^2+O(q^3)\,,\nn
      \wt{M}_{II}(x;q)_{3,3} =
    &-x^{-1}+q-(x^2-1)q^2+O(q^3)\,.
  \end{align}
  They are in agreement with \eqref{eq:M1-52},\eqref{eq:M2-52}.  More
  decisively, we can compare the numerical evaluation of both sides of
  the equations of holomorphic lifts
  \eqref{eq:HU1-52},\eqref{eq:HU2-52}.  We find the relative
  difference between the two sides is always within the error margin
  of Borel-Pad\'e resummation, and much smaller than
  $\tq^{\pm 1},\tx^{\pm 1}$, possible sources of additional
  corrections.  We illustrate this by one example with $x=6/5$ and
  $\tau=\frac{1}{18}\re^{\frac{\pi\ri}{12}},
  \frac{1}{18}\re^{\frac{11\pi\ri}{12}}$ in regions
  $I,II$ in Table~\ref{tab:num-52}.  Finally we can test the resurgent
  data by checking that in the $x\mapsto 1$ limit the Stokes matrices
  \eqref{eq:Sp-52}, \eqref{eq:Sn-52} reduce properly to \eqref{S52p},
  \eqref{S52m}.  This is a non-trivial test as $W_{-1}(x;q^{-1})$
  ($|q|<1$) in \eqref{eq:Sp-52}, \eqref{eq:Sn-52} itself diverges in
  the limit $x\mapsto 1$.


\section{One-dimensional state-integrals and their
    descendants}
\label{sec.onedim}

In a sense, the results of our paper are not about the asymptotics and
resurgence of complex Chern--Simons theory, but rather involve power
series and $q$-difference equations that arise from $K_2$-Lagrangians
(clearly advocated in Kontsevich's talks~\cite{kontsevich-talk}), or
from symplectic matrices (advocated in~\cite[Sec.7]{GZ:kashaev}). The
connection with complex Chern--Simons theory comes via ideal
triangulations of a 3-manifold with torus boundary components, a
concept introduced by Thurston for the study of complete hyperbolic
structures and their deformations~\cite{Thurston}.  The gluing
equations of such triangulations are encoded by matrices which are the
upper half of a symplectic matrix (see Neumann-Zagier~\cite{NZ}). The
upper half of these symplectic matrices define state-integrals, as well 
as the asymptotic series $\Phi(x;\tau)$ (this was the approach
taken in~\cite{DG}) and the 3D-index (see~\cite{DGG1}).

In this section we discuss briefly general one-dimensional
state-integrals and their descendants, and their asymptotic series. We
will not aim for maximum generality, but instead consider the
one-dimensional state-integral
\begin{equation}
  \label{I1d}
  Z_{A,r}(u,t;\tau)
  =\int_{\BR+\ri 0}
  \left(\prod_{j=1}^r\Phi_\bb(v+u_j)\right)
  \re^{-A\pi\ri v^2 + 2\pi\ri v t} \rd v \,,
\end{equation}
where $u=(u_1,\dots,u_r) \in \BC^r$ with $|\Im u_j| < |\Im c_\bb|$
(this ensures that all poles of the integrand are above the real
axis), $t \in \BC$ and $A$ and $r$ integers with $r>A>0$. (This
ensures that the integrand decays exponentially at infinity, hence the
integral is absolutely convergent.) We have already encountered two
special cases in Equations~\eqref{Z41x-desc} and~\eqref{Z52x-desc}:
\begin{align}
  \label{I4152a}
  \knot{4}_1:
  & \,\, (A,r)=(1,2), \,\, (u_1,u_2) = (u,0), \,\, w = -2u, \\
  \label{I4152b} \knot{5}_2:
  & \,\, (A,r)=(2,3), \,\, (u_1,u_2,u_3) = (0,u,-u),
    \,\, w = 0 \,.
\end{align}
We are interested in the descendants of $Z$ defined by
\begin{equation}
  \label{I1dd}
  z_{A,r,m,\mu}(u,t;\tau) = (-1)^{m+\mu} \,
  q^{-\frac{1}{2}m} \, \tq^{-\frac{1}{2}\mu} \,
  Z_{A,r}(u,t - m \ri \bb + \mu \ri \bb^{-1};\tau)
\end{equation}
for integers $m$ and $\mu$, where the extra factor was inserted to
simplify the formulas below.

We will express the factorization of the state-integral~\eqref{I1dd}
in terms of the auxiliary function
\begin{equation}
  \label{Gar}
  G_{A,r}(y,z;q)
  =\frac{1}{(q;q)_\infty} \sum_{n=0}^\infty
  (-1)^{An} q^{\frac{A}{2}n(n+1)} z^n \prod_{j=1}^{r}(q^{1+n}y_j;q)_\infty 
\end{equation}
for $y=(y_1,\dots,y_r)$ and its specialization 
\begin{equation}
  \label{bmxw}
  b^{(k)}_{m}(x,w;q)
  = \frac{1}{x_k^{m}}
  G_{A,r}\left(\frac{1}{x_k} x, x_k^{-A} w^{-1} q^{m};q \right)
  q\end{equation}
for $x=(x_1,\dots,x_r)$ 
and its renormalization 
\begin{equation}
  \label{Bmxw}
  B^{(k)}_{m}(x,w;q) =
  \th(-q^{-1/2}x_k;q)^{-A+1}\th(-q^{-1/2}x_k w;q)^{-1}\th(w;q)
  b^{(k)}_{m}(x,w;q)
\end{equation}
for $k=1,\dots,r$.

\begin{theorem}
  \label{thm.Zfac}
  \rm{(a)} The descendant state-integral can be expressed in terms of
  the descendant holomorphic blocks by
\begin{equation}
  \label{Z-fac}
  z_{A,r,m,\mu}(u,t;\tau) 
  =
  B_{-\mu}(\tx,\tw;\tq^{-1})^T \Delta(\tau) B_m(x,w;q),
  \qquad (m, \mu \in \BZ)
\end{equation}
where 
\begin{equation}
  x_j = \re^{2\pi \bb u_j},\quad \tx_j =
  \re^{2\pi \bb^{-1}u_j},\quad
  w = \re^{2\pi\bb t},\quad \tw = \re^{2\pi\bb^{-1}t} \,.
\end{equation}
and 
\begin{equation}
\label{Dr}
\Delta(\tau) =
\re^{-\frac{\ri\pi}{4}+\frac{(A-2)\ri\pi}{12}(\tau+\tau^{-1})}\md{1}
\end{equation}
and $B_m(x,w;q)=(B^{(1)}_m(x,w;q), \dots, B^{(r)}_m(x,w;q))^T$. 
Consider the matrix $W_m(x,w;q)$ defined by
\begin{equation}
  \label{Wrm}
  W_m(x,w;q) =
  \begin{pmatrix}
    B^{(1)}_{m}(x,w;q) & \ldots & B^{(r)}_{m}(x,w;q) \\
    \vdots & & \vdots \\
    B^{(1)}_{m+r-1}(x,w;q) & \ldots & B^{(r)}_{m+r-1}(x,w;q)
  \end{pmatrix}
\end{equation}
\rm{(b)} The entries of $W_m(x,w;q)$ are holomorphic
functions of $|q| \neq 1$ and meromorphic functions of $(x,w) \in
(\BC^*)^r \times \BC^*$
with poles in $x_j \in q^\BZ$ (for $j=1,\dots,r$) and
$w \in q^\BZ$ of order at most $r$. 
  \newline
  \rm{(c)} The columns of $W_m(x,w;q)$ are a basis of solutions of the
  linear $q$-difference equation $\Bhat(S_m,q^m,x,w,q) f_m(x;q)=0$ for
  $|q| \neq 1$ and $m \in \BZ$ where 
    \begin{equation}
  \label{bkm}
  \Bhat(S_m,q^m,x,w,q) = \prod_{k=1}^r(1-x_i S_m) - (-q)^A w^{-1} q^{m} S_m^A \,.
  \end{equation}
  In particular, $m \mapsto z_{A,r,m,\mu}(u,t;\tau)$ is annihilated by
  the operator $\Bhat(S_m,q^m,x,w,q)$.
\end{theorem}

\begin{proof}
  For part (a), summing up all the residues of the integrand
  of~\eqref{I1d} in the upper half-plane as in~\cite{GK:qseries}, we
  find that 
\begin{align}
  \label{eq:Ira-fac}
  z_{A,r,m,\mu}(u,t;\tau)
  &=
    \re^{-\frac{\pi\ri}{12} +\frac{1}{3}\pi\ri c_\bb^2+2\pi\ri c_\bb t}
    \sum_{k=1}^r \re^{-A\pi\ri(u_k-c_\bb)^2-2\pi\ri u_k t} 
    b^{(k)}_{m}(x,w;q) b^{(k)}_{\mu}(\tx,\tw;\tq^{-1}) \\
    \label{eq:Ira-fac2}
  &=
    \re^{-\frac{\pi\ri}{12}+\frac{1}{3}\pi\ri c_\bb^2+\frac{\pi\ri}{12}(A-1)(\tau+\tau^{-1})}
    \sum_{k=1}^r B^{(k)}_{m}(x,w;q) B^{(k)}_{\mu}(\tx,\tw;\tq^{-1})
  \end{align}
  The last equation follows from~\eqref{ebx2} (which takes care of
  $\re^{-A \pi \ri (u_k-c_\bb)^2}$) and~\eqref{ebx3} (which takes
  care of the $t$-terms under the assumption that $u_k t = \rho_k u$
  and $t = \rho u$ for integers $\rho$ and $\rho_k$) 

  For part (b), note that $G_{A,r}(y,z;q)$ is symmetric with respect
  to permutation of the coordinates of $y$ and that the specialization
  to $y_r=1$ is given by
\begin{equation}
\label{Garr}
 G_{A,r}(y,z;q)|_{y_r=1} 
    = \sum_{n=0}^\infty
    (-1)^{An} \frac{q^{\frac{A}{2}n(n+1)}}{(q;q)_n} z^n
    \prod_{j=1}^{r-1}(q^{1+n}y_j;q)_\infty, \qquad (|q| \neq 1)\,.
\end{equation}
It follows that $G_{A,r}(y,z;q)|_{y_r=1}$ is holomorphic for $(y,z) \in
\BC^{r-1} \times \{1\} \times \BC$ when $|q|<1$ and meromorphic in
$(y,z) \in (\BC^*)^{r-1} \times \{1\} \times \BC^*$ with
poles in $y_j \in q^\BN$ for $j=1,\dots,r$. Since $b^{(k)}_m(x,w;q)$
are expressed in terms of a specialization of $G_{A,r}(y,z;q)|_{y_r=1}$,
part (b) follows. 

  Part (c) follows from Equation~\eqref{bmxw} and the fact (proven
  by a standard creative telescoping argument) that the function
  $z \mapsto G_{A,r}(y,z;q)$ is annihilated by the operator
  \begin{equation}
  \label{Gzrec}
  \prod_{k=1}^r (1-y_k L_z) - (-q)^A z L_z^A 
  \end{equation}
  where $L_z$ shifts $z$ to $q z$.
\end{proof}


\bibliographystyle{hamsalpha}
\bibliography{biblio2}
\end{document}